 \documentclass[final,5p,times,twocolumn]{elsarticle}

\usepackage{amssymb}
\usepackage{amsmath}
\usepackage{amsthm}
\usepackage{ulem} 
\usepackage{algorithm} 
\usepackage{algorithmicx}%
\usepackage{algpseudocode}%
\usepackage{multicol}
\usepackage{booktabs}
\usepackage{tabularx}
\usepackage{color}
\usepackage{wrapfig}
\usepackage{graphicx}
\usepackage{array}
\usepackage{subcaption}
\usepackage[T1]{fontenc} 
\usepackage{multirow}
\newtheorem{theorem}{Theorem}

\newtheorem{lemma}{Lemma}
\newtheorem{remark}{Remark}
\newtheorem{proposition}{Proposition}

\journal{} 

\begin{document}

\begin{frontmatter}



\title{Block coordinate descent for joint delay-energy optimization in multi-hop D2D networks}

\cortext[cor1]{Corresponding author}
\author[label1,label2]{Kai-Xiang Hu\corref{cor1}} 
\affiliation[label1]{organization={School of Mathematical Sciences, Beijing University of Posts and Telecommunications},
	city={Beijing},
	postcode={100876}, 
	country={China}}
\ead{hukx@bupt.edu.cn}
%
%
%
%
\author{Jacek Gondzio\fnref{label2}}
\affiliation[label2]{organization={School of Mathematics and Maxwell Institute for Mathematical Sciences, The University of Edinburgh},
	city={Edinburgh},
	postcode={EH9 3FD}, 
	state={Scotland},
	country={United Kingdom}}
\ead{J.Gondzio@ed.ac.uk}

\author[label1]{Caixia Kou}

\ead{koucx@bupt.edu.cn}
\begin{abstract}
	In multi-hop device-to-device (D2D) networks, the optimization of network-level metrics is particularly difficult due to the tight coupling between network-layer routing and physical-layer resource allocation.
	Departing from traditional average-performance metrics, this paper addresses the joint optimization of routing paths, transmission power, and bandwidth allocation. We formulate a generalized cost function to minimize the maximum transmission time (i.e., the bottleneck delay) alongside the total energy consumption. To tackle the resulting highly non-convex formulation, we propose a novel block coordinate descent (BCD) framework. At the network layer, we develop two adaptive routing algorithms: a matrix-free Frank-Wolfe (MF-FW) algorithm for fast execution in dense topologies, and a low-rank primal-dual interior-point method (LR-PDIPM) that bypasses dense matrix inversions via the Sherman-Morrison formula for high-precision solutions. At the physical layer, we design a parallel dual ascent algorithm leveraging a time-domain perspective transformation to solve the resource allocation subproblem to global optimality. The proposed BCD framework is proven to converge to an $\varepsilon$-neighborhood of a stationary point.
	Through comprehensive experiments, the proposed BCD framework establishes its superiority in achieving the optimal delay-energy trade-off. Specifically, the LR-PDIPM variant achieves a maximum $9.14\times$ reduction in total energy consumption and up to an order of magnitude improvement in energy efficiency, while maintaining a bounded maximum delay gap (up to $3.78\times$) relative to the best baseline.  Meanwhile, the warm-start MF-FW variant identifies near-optimal solutions in mere seconds, serving as a highly practical engineering approach.
\end{abstract}


\end{frontmatter}
\section{Introduction}
\label{sec:intro}

Device-to-device (D2D) communications have emerged as a cornerstone for future B5G/6G networks, facilitating proximity-based data exchange without traversing the base station \cite{asadi2014survey,de2026joint,ansari20175g,wang2025dacs}. While single-hop D2D offers spectral efficiency gains, it is limited by short transmission ranges and susceptibility to signal blockage. Consequently, {multi-hop multi-path } architectures are indispensable for extending network coverage and ensuring robust connectivity in dense urban  or industrial environments \cite{xie2022dynamic,deng2023multi,huang2024distributed}.


However, realizing the potential of multi-hop multi-path networks requires navigating a trade-off between end-to-end delay and energy efficiency. 
Time-sensitive applications (e.g., industrial automation) demand low transmission time \cite{behnke2023real,hu2021deep}. 
Crucially, in a multi-path transmission setting, the overall service completion time is dictated by the {slowest routing path} (i.e., the  bottleneck flow). 
Therefore, simply minimizing the average delay is insufficient; instead, minimizing the min-max transmission time across all concurrent routing paths is paramount to prevent any single path from becoming a performance bottleneck.
Yet, minimizing this delay typically necessitates aggressive power allocation and bandwidth usage, which conflicts directly with the limited battery capacity of relay nodes \cite{zhao2019throughput,nguyen2019collaborative}. 
Purely delay-centric strategies risk draining the power of critical nodes  prematurely, while purely energy-centric strategies result in unacceptable transmission time.


Beyond the conflicting nature of these objectives, optimizing such a network presents two mathematical challenges. 
First, because our goal is to minimize the {maximum} transmission time among all routing paths, the resulting objective function is non-smooth and intractable. 
This leads to a ``jagged'' optimization landscape where the identity of the slowest path can suddenly jump from one route to another as resources are adjusted. 
This non-smoothness prevents gradient-based  optimization tools from finding an exact descent direction, as they often get stuck at cusps where differentiability fails. 
Second, the decision variables for routing and resource allocation are strongly coupled. 
The transmission time on any given path depends not only on the traffic load but also on the transmission rates of its constituent links. 
Crucially, these rates are not fixed; they are variables that change dynamically based on the allocation of limited power and bandwidth. 
This creates a complex interdependency where routing strategy and resource allocation are intertwined: altering the resource allocation changes the transmission rate, which in turn necessitates a re-evaluation of the optimal routing paths. 



Existing literature inadequately addresses these challenges, leaving a gap between theoretical requirements and practical solutions.
\begin{itemize}
	\item 
	While some studies employ alternating strategies to simplify the strongly coupled problem, they typically optimize the original problem or its subproblems using heuristic ideas \cite{chen2022energy,paul2021machine}. 
	These alternating updates lack rigorous theoretical guarantees, often leading to solution oscillation or failure to converge to a stationary point.	
	
	\item 
	Some existing studies focus on optimizing average delay or sum throughput \cite{feng2013device,zhao2017joint,al2022joint}. 
	These aggregate metrics {mask the bottleneck flows}. The resulting algorithms might achieve excellent average performance while leaving the bottleneck path congested, thereby violating the strict service completion time required by time-sensitive applications.
	
	\item 
	Standard gradient-based methods fail at cusps of the min-max function, prompting most researchers to resort to heuristic algorithms \cite{al2024min,li2023min}. 
	While flexible, these heuristics lack the theoretical convergence guarantees required for high-reliability communications.
\end{itemize}

To overcome these challenges, we  adopt an {overlay D2D architecture} with orthogonal resource slicing. Unlike underlay D2D, which suffers from severe mutual interference within cellular networks, the overlay architecture allocates dedicated spectrum bandwidth to D2D users. This architecture eliminates cross-tier interference. 
However, the total bandwidth is constrained. To prevent severe resource overutilization, the system must jointly optimize its  network-layer routing and physical-layer resources. While the overlay mode removes external interference, the joint optimization of routing and Shannon-capacity-based resource allocation remains highly coupled and non-convex. 
To tackle this, we formulate a generalized cost function. This function unifies the conflicting objectives of worst-case transmission time and network energy consumption. Furthermore, rather than relying on heuristics, we apply a smooth approximation to navigate the non-differentiable min-max objective. This transformation facilitates the application of the gradient-based optimization techniques and allows us to provide theoretical guarantees of convergence to a stationary point.


The main contributions of this paper are summarized as follows:
\begin{itemize}
	\item We formulate the joint routing and resource allocation (JRRA) problem in overlay D2D networks as a highly coupled, non-convex optimization problem. To overcome the prohibitive computational overhead, we propose a block coordinate descent (BCD) framework that jointly optimizes the network-layer routing and the physical-layer resource allocation. This framework decouples the intractable formulation into a sequence of solvable convex subproblems.
	
	\item 
	To address the routing subproblem, we develop an adaptive algorithmic framework. For dense topologies, we propose a matrix-free Frank-Wolfe (MF-FW) algorithm to achieve rapid execution. For high-precision scenarios, we develop a low-rank primal-dual interior-point method (LR-PDIPM). The LR-PDIPM leverages the Sherman-Morrison (SM) rank-1 formula to bypass the direct inversion of the dense reduced Newton system, reducing the per-iteration computational complexity from $\mathcal{O}(|\mathcal{K}|^3 |\mathcal{V}|^3)$ to $\mathcal{O}(|\mathcal{K}| |\mathcal{V}|^3)$. Regarding the physical-layer allocation, we exploit a time-domain perspective transformation and Lagrangian dual decomposition to split the problem into link-level subproblems. These subproblems are then solved to global optimality in parallel via coordinate descent integrated with inner 1D bisection search.
	
	\item Through rigorous derivation and comprehensive experiments, we prove that the proposed BCD framework converges to  an $\varepsilon$-neighborhood of a stationary point. Numerically, we validate its superiority in achieving the optimal delay-energy trade-off. The BCD variant integrated with LR-PDIPM achieves a maximum $9.14\times$ reduction in total energy consumption and up to an order of magnitude improvement in energy efficiency relative to the best baseline.  Meanwhile, the warm-started MF-FW variant identifies near-optimal solutions in mere seconds, serving as a highly practical engineering approach.
\end{itemize}

The remainder of this paper is organized as follows. Section \ref{sec:rel} reviews the related work. Section \ref{sec:model} presents the system model and formulates the JRRA problem. Section \ref{sec:algo} details the proposed BCD framework, including the adaptive network-layer routing algorithms and the physical-layer resource allocation algorithm. Section \ref{sec:convergence} provides a rigorous theoretical analysis regarding algorithmic convergence and computational complexity. Section \ref{sec:numerical_results} presents extensive numerical results and performance evaluations. Finally, Section \ref{sec:conclusion} concludes the paper.

\section{Related works}
\label{sec:rel}

\subsection{JRRA in multi-hop D2D networks}

Given the inherent complexity of joint multi-dimensional resource management in D2D communications, early literature predominantly relied on simplified models. To ensure tractability, early works often assumed predetermined link scheduling or single-channel scenarios \cite{jameel2018survey}. Consequently, many studies separated the tightly coupled variables, focusing strictly on single-dimensional power control and ignoring the joint allocation of spectrum resources. For instance, the work in \cite{abrardo2016distributed} addressed the non-convex sum-rate maximization problem by modeling power allocation as a potential game, achieving the convergence to local maxima. Similarly, the work in \cite{wu2015optimal} derived closed-form optimal power control strategies by partitioning circuit power consumption into distinct operational regions. While effective for single-dimensional resource management, these decoupled approaches fall short in modeling the strict interdependency among routing, spectrum, and power in multi-hop D2D networks. 

To capture these physical-layer dependencies, subsequent research transitioned toward joint resource allocation.  In \cite{wang2014energy}, the authors proposed an iterative combinatorial auction algorithm to allocate both power and radio resources efficiently. Furthermore, to address practical battery limitations, works such as \cite{liu2017economic} and \cite{tao2017review} incorporated Peukert’s law to model nonlinear battery lifetime, demonstrating the necessity of strict distance constraints for D2D pairs to maintain energy efficiency. Notably, the study in \cite{yang2015energy} investigated energy-efficient allocation specifically in D2D overlay networks. They highlighted that allocating spectrum resources orthogonally completely eliminates cross-tier interference, significantly simplifying interference management. However, these studies were confined strictly to physical-layer resource allocation for predefined routing paths, leaving the combinatorial complexities of network-layer routing completely unexplored.

In multi-hop settings, the coupling between physical-layer resource allocation and network-layer routing becomes the dominant performance bottleneck, prompting a shift toward cross-layer joint optimization. For example, the work in \cite{tran2021spectrum} investigated joint spectrum and power optimization tailored for multi-hop multi-path D2D video delivery. In the context of D2D-assisted decentralized learning, the work in \cite{liu2023communication} formulated a joint problem encompassing computing power, wireless resource allocation, and link selection to minimize a weighted sum of learning latency and energy consumption.

Despite these cross-layer advancements, a gap remains in fully capturing the intricate coupling between physical-layer resource allocation and network-layer routing, particularly when attempting to bound the worst-case delay for bottleneck flows.  Most of these joint frameworks either rely on aggregate metrics (e.g., weighted sums \cite{liu2023communication}) that mask bottleneck flows, or they assume a static network topology where routing paths are predetermined before resource allocation \cite{tran2021spectrum}. They fall short in accurately modeling the dynamic dependency of routing decisions on the continuously updating transmission rate, which fluctuate based on power and bandwidth allocation. These inaccurate allocation and routing strategies result in suboptimal network utilization, particularly failing to guarantee strict service completion times in latency-sensitive applications.

\subsection{Algorithmic paradigms: heuristics, learning, and mathematical programming}

Traditionally, the joint optimization of resource allocation and   multi-hop routing  is formulated using explicit discrete link selection alongside continuous variables for power and bandwidth. This  modeling approach results in a highly complicated mixed-integer nonlinear programming (MINLP) problem. To solve this problem, existing literature generally relies on three primary algorithmic paradigms: meta-heuristics, learning-based approaches, and deterministic mathematical programming.

\paragraph{Heuristic and meta-heuristic approaches}
To bypass the mathematical intractability of the original mixed-integer formulation, early implementations often resorted to heuristic rules and meta-heuristic algorithms. For instance, the authors in \cite{ergiz2021joint} resorted to a low-complexity heuristic that utilizes Breadth-First Search (BFS) for multi-path routing discovery, followed by a greedy resource allocation based on quality-to-rate ratios. The work in \cite{gures2026joint} proposed an alternating decomposition framework: the discrete routing variables are determined via a sequential greedy strategy, while the continuous resource allocation is optimized through an iterative potential game.  In the context of D2D communications within 5G heterogeneous networks, the authors in \cite{valiveti2017ehsd} proposed an exemplary handover scheme based on fuzzy logic to decentralize software-defined radio control. When addressing the mixed binary-continuous nature of joint routing and  resource allocation, genetic algorithms are frequently employed. For instance, the work in \cite{tran2021spectrum} utilized genetic algorithms to solve the spectrum and power optimization problem, demonstrating high empirical video throughput and energy efficiency.  However, despite their flexibility, meta-heuristics and fuzzy logic approaches lack theoretical convergence guarantees. They are particularly vulnerable to the non-smooth, ``jagged'' landscape of min-max objective functions, where they easily stagnate at suboptimal points without the rigorous theory to verify optimality or stationarity.

\paragraph{Learning-based and data-driven approaches}
To handle the inherent uncertainties and high computational costs of D2D networks, learning-based approaches have recently gained traction for managing the intractable nonlinear coupling and combinatorial complexity of joint routing and  resource allocation.  For instance, the authors in \cite{kazemi2018learning} formulated the joint channel and power allocation problem as an interactive learning task, specifically for scenarios where channel state information  is unknown a priori. By employing a recency-based Q-learning method, they successfully tracked the dynamic environment, achieving an order-optimal policy with strict mathematical convergence. Shifting the focus toward algorithmic efficiency, the work in \cite{xu2018d2d} leveraged a hierarchical extreme learning machine to bypass the heavy computational overhead of traditional iterative power control solvers. Their approach achieved near-optimal transmission rates while significantly reducing the algorithmic execution time. For fully decentralized generation of D2D networks, the study in \cite{ioannou2020distributed} utilized extended belief-desire-intention  agents and the AI-based algorithm to enable autonomous transmission mode selection, effectively balancing data rate maximization with power minimization under local computational constraints. To explicitly handle multi-hop topologies, the work in \cite{chen2022federated} implemented federated learning  with in-network aggregation to reduce mesh network bottlenecks. Furthermore, the work in \cite{zhang2022topology} leveraged graph neural networks  to capture spatial dependencies, predicting link usage for multi-commodity flows and drastically reducing computation time. While these data-driven paradigms offer rapid execution, they suffer from inherent black-box uninterpretability and require extensive retraining when network topologies dynamically change. More crucially, statistical learning methods cannot provide the strict, deterministic worst-case delay guarantees mandated by critical overlay D2D  networks.

\paragraph{Deterministic and mathematical programming}
Given the limitations of heuristic and learning-based models in highly dynamic environments, it is imperative to rely on deterministic mathematical programming. 
For pure physical-layer resource allocation, researchers have extensively leveraged convex and generalized convex optimization tools. In \cite{wang2017power}, the authors developed an approximate interior point method (IPM) to maximize sum throughput by simplifying the Hessian matrix inversion. To address fractional objectives such as energy efficiency, the authors in \cite{yang2015energy} transformed the mixed-integer nonlinear fractional programming problem into a parametric subtractive form, solving it via Dinkelbach's method and augmented Lagrangian techniques. Similarly, the work in \cite{abrardo2016distributed} modeled power allocation as a potential game to ensure convergence to a local maximum. While mathematically rigorous, these approaches primarily focus on continuous variables in single-hop or fixed-routing scenarios, ignoring the topological complexities of multi-hop paths.

To explicitly handle the integer constraints inherent in routing and link scheduling, classical operations research techniques have been applied. For instance, the pioneering research in \cite{capone2009routing} established a  mixed-integer linear programming framework for service overlay networks and solved it via the solver CPLEX.   In \cite{zhai2017energy}, the authors exploited the minimum-cost maximum-flow graph structure to design an exact branch-and-bound algorithm, providing a strict global optimality benchmark. The study in \cite{xu2022branch} formulated the unsplittable multi-commodity flow problem using mixed-integer programming and solved it via a branch-and-price algorithm. To accelerate these exact but computationally demanding methods, hybrid frameworks have emerged. For instance, the work in \cite{zhang2021resource} proposed a column generation framework that provides not only  an optimality benchmark but also structural support for a deep reinforcement learning agent in dynamic mmWave networks.  However, despite their mathematical exactness, these integer-focused methods suffer from severe scalability issues. More crucially, to maintain computational tractability, they often require linearizing or oversimplifying the nonlinear physical-layer capacity models. Therefore, these methods have limited capability to capture the exact, tight coupling between routing decisions and continuous resource allocation.  

To bridge the gap between discrete routing and continuous resource allocation, recent works have adopted decomposition techniques for the resulting MINLP models. In \cite{chen2022energy}, The authors utilized a two-layer approach, combining Dinkelbach's algorithm in the outer loop with alternating optimization in the inner loop.  The authors in \cite{omidkar2022reinforcement} addressed a highly non-convex MINLP by decomposing it into sequential subproblems, systematically integrating reinforcement learning for dynamic environment exploration with majorization-minimization and Dinkelbach's method for tractable continuous resource allocation. Furthermore, dealing with long-term system stability, the work in \cite{budhiraja2024joint} combined Lyapunov-based distributed schemes with successive convex approximation (SCA) to maximize the time-averaged network sum-rate.  

\subsection{Min-max objectives and the non-smooth coupling challenge}

While the aforementioned mathematical programming methods (e.g., alternating optimization and SCA) provide robust frameworks for JRRA  problem, they are challenged by the specific requirement of  worst-case performance in multi-hop networks. As highlighted earlier, minimizing the maximum path delay introduces a  non-differentiable objective function. Crucially, this min-max formulation intricately couples the resource variables across intersecting paths, rendering standard decoupled optimization techniques ineffective. Existing decomposition methods often bypass this mathematical bottleneck by reverting to tractable time-averaged metrics \cite{budhiraja2024joint, pu2016d2d} or sum-rate objectives \cite{abrardo2016distributed, wang2017power}, masking the network's most critical bottleneck flows. 

Historically, the majority of literature has favored aggregate metrics due to their smooth and differentiable nature. For instance, numerous studies focus on maximizing the sum rate or total throughput of D2D pairs \cite{abrardo2016distributed, wang2017power, zhou2016joint}. To handle larger topologies, the work in \cite{rasekh2019joint} maximized total backhaul throughput in mmWave mesh networks by transforming the resource allocation into a scalable mixed-integer linear programming. For delay-centric applications, the work in \cite{liu2022routing} minimized the total weighted transmission time in air-ground mesh networks via a three-stage decomposition. While some works, such as \cite{gu2015heuristic}, attempt to introduce proportional fairness by maximizing the logarithmic sum of data rates, these aggregate or weighted-sum objectives mask the performance of bottleneck flows. Because the objective function landscape is smooth, standard gradient descent or linear programming techniques apply seamlessly, but the resulting resource allocation often leaves bottleneck nodes or links severely congested.

Recognizing the necessity of strict worst-case guarantees, recent literature has shifted toward min-max (or max-min) formulations. In \cite{yang2016energy}, the authors investigated energy-efficient power control by maximizing the minimum individual energy efficiency among D2D pairs. In the realm of mobile edge computing, the study in \cite{zhao2020mobile} formulated a joint hybrid beamforming and resource allocation strategy explicitly designed to minimize the maximum system delay. However, explicitly targeting the worst-case path introduces a severe mathematical barrier: the $\max(\cdot)$ operator renders the objective function  non-differentiable at points where multiple path delays intersect. 
To handle this non-smoothness, standard approaches typically rely on epigraph transformations, introducing auxiliary variables to shift the non-differentiability from the objective function into a set of inequality constraints. 
Although theoretically sound, this reformulation severely inflates the problem's dimensionality. More critically, these joint constraints couple the feasible domains of the optimization variables, thereby destroying the Cartesian product structure essential for BCD. Consequently, when dealing with the highly nonlinear flow and resource variables inherent in multi-path routing, this forced constraint coupling renders standard alternating optimization ineffective, often trapping the algorithm in deadlocks or necessitating computationally prohibitive inner loops.

To resolve this non-smooth intractability without incurring a dimensional inflation, optimization theory offers rigorous smooth approximation techniques \cite{boyd2004convex, nesterov2005smooth}.
Recently, the Log-Sum-Exp (LSE) approximation has gained significant traction in signal processing and communications to handle complex saddle point and min-max formulations \cite{lu2020hybrid}. For example, the work in \cite{chen2019joint} converted a non-convex, non-differentiable min-max fairness problem in an uplink edge computing system into a tractable form using LSE approximation, effectively managing the nonlinear coupling of design variables. Similarly, the study in \cite{liesegang2026emf} employed the LSE approximation combined with successive convex optimization for power control. By replacing the non-smooth discrete minimum objective with a concave LSE lower bound, they bypassed the need for iterative feasibility-check loops (e.g., bisection searches), thereby enabling gradient-based optimization.

Despite these theoretical advancements, the application of smooth approximations in wireless networks remains largely confined to continuous physical-layer resource allocation or single-hop computation offloading. This min-max non-differentiability is rarely addressed in multi-hop networks, where flow routing and resource allocation are tightly coupled. Motivated by this limitation, we introduce the LSE approximation into the tightly coupled domain of multi-path flow and resource allocation, thereby transforming the non-smooth worst-case objective into a mathematically tractable structure. Unlike conventional methods that struggle with the combinatorial explosion inherent to discrete network topologies, we adopt a {node-arc continuous flow formulation}. Within this purely continuous yet highly nonlinear formulation, our LSE approach fully unlocks the tractability of gradient-based alternating optimization for the JRRA problem in multi-hop D2D networks.

\section{System model and problem formulation}
\label{sec:model}
The primary symbols and variables used in this paper are summarized in Table \ref{tab:notation}. 


\begin{table}[htbp]
	\centering
	\small 
	\caption{Summary of key notations}
	\label{tab:notation}
	\begin{tabularx}{\columnwidth}{@{}lX@{}}
		\toprule
		{Symbol} & {Description} \\ 
		\midrule
		\multicolumn{2}{@{}l}{\textit{{Sets and indices}}} \\
		$\mathcal{G}(\mathcal{V}, \mathcal{E})$ & Directed  network graph with node set $\mathcal{V}$ and link set $\mathcal{E}$ \\
		$(i,j) \in \mathcal{E}$ & Directed link from node $i$ to node $j$ \\
		$\mathcal{K}$ & Set of D2D communication commodities (traffic) \\
		$k \in \mathcal{K}$ & Index of a specific commodity \\

		\midrule
		\multicolumn{2}{@{}l}{\textit{{System parameters}}} \\
		$M^k$ & Total traffic demand for commodity $k$ (Mbits) \\
		$B$ & Total available bandwidth allocated to the overlay D2D network (MHz) \\
		$P_i^{\max}$ & Maximum transmission power budget at node $i$ (dBm) \\
		$h_{ij}$ & Channel gain over link $(i,j)$ \\
		$N_0$ & Noise power spectral density   (dBm/MHz) \\
		$\alpha, (1-\alpha)$ & Weighting coefficients for delay and energy in the generalized cost \\
		$\mu$ & Smoothing parameter for the LSE worst-case delay approximation \\
		
		\midrule
		\multicolumn{2}{@{}l}{\textit{{Decision variables}}} \\
		$x_{ij}^k \in [0,1]$ & Flow proportion of commodity $k$ routed over link $(i,j)$ \\
		$l_{ij} \ge 0$ & Bandwidth allocated to link $(i,j)$  \\
		$p_{ij} \ge 0$ & Transmission power allocated to link $(i,j)$ \\
		$t_{ij} \ge 0$ & Transmission time over link $(i,j)$ \\
		$r_{ij} \ge 0$ & Transmission rate over link $(i,j)$ \\
		$m_{ij} \ge 0$ & Traffic load over link $(i,j)$ \\
		\bottomrule
	\end{tabularx}
\end{table}

\subsection{Network and traffic model}

Consider a multi-hop D2D network, modeled as a directed graph $\mathcal{G} = (\mathcal{V}, \mathcal{E})$, where $\mathcal{V}$ denotes the set of nodes and $\mathcal{E}$ represents the set of directed wireless links (edges).  The network serves a set of data commodities $k$. Each commodity $k \in \mathcal{K}$ is defined by a source-destination node pair, and a specific end-to-end data demand $M^k$ (in Mbits).

To capture the routing flows without enumerating all possible paths, we define continuous arc-based flow variables. Let $x_{ij}^k \in [0, 1]$ denote the flow proportion of commodity $k$ routed through link $(i,j)$. The topological structure of the network is governed by the node-arc incidence matrix $A$. 
These flow variables  $x_{ij}^k $ must satisfy strict flow conservation constraints across all nodes, balancing the incoming and outgoing flows to match the source-destination requirements defined by vector $\mathbf{b}^k$.

\subsection{Physical layer and resource model}
The physical layer coordinates the allocation of spectrum bandwidth and transmission power. Under the overlay D2D network architecture, D2D communications are isolated from cellular traffic. Furthermore, the system employs exclusive subcarrier allocation to eliminate mutual co-channel interference among active D2D links. Within this interference-free environment, the physical resources are defined as follows:
\begin{itemize}
	\item Let $l_{ij}$ denote the spectrum bandwidth allocated to link $(i,j)$. The allocation is constrained by the  available bandwidth budget $B$, which is exclusively partitioned among all active links.
	\item Let $p_{ij}$ represent the transmission power assigned to link $(i,j)$. Each node $i$ is bounded by a maximum power budget $P_i^{\max}$, limiting the energy radiated by all its outgoing links.
\end{itemize}

In the considered D2D overlay network, the achievable transmission rate $r_{ij}$ over link $(i,j)$ is governed by the allocated resources $(\mathbf{l}, \mathbf{p})$. Given the exclusive resource allocation inherent to the overlay architecture, multi-user interference is eliminated. Assuming an additive white Gaussian noise channel, the noise-limited link transmission rate is formulated as
\begin{equation}
	r_{ij} = l_{ij} \log_2 \left( 1 + \frac{p_{ij} h_{ij}}{N_0 l_{ij}} \right),
	\label{eq:shannon_rate}
\end{equation}
where $N_0$ denotes the noise power spectral density. The parameter $h_{ij}$ defines the channel gain of link $(i,j)$, which follows the standard 3GPP propagation model $h_{ij} = 10^{-\mathrm{PL}(d_{ij})/10}$. The distance-dependent path loss $\mathrm{PL}(d_{ij})$ is given by
\begin{equation}
	\mathrm{PL}(d_{ij}) = 128.1 + 37.6 \log_{10}\left(\frac{d_{ij}}{1000}\right),
\end{equation}
where $d_{ij}$ represents the Euclidean distance  in meters between nodes $i$ and $j$.

\subsection{Problem formulation and arc-based approximation}
The primary objective is to strike an optimal trade-off between two competing performance metrics: the worst-case transmission time and the total energy consumption. Following standard practices in multi-objective network optimization \cite{liu2023communication}, we scalarize these metrics into a unified composite cost function. Specifically, we introduce a weighting parameter $\alpha \in [0, 1]$ to govern the trade-off between the two metrics.

Ideally, the transmission time for commodity $k$ must be evaluated over its explicit routed paths. Let $y_p^k$ denote the flow proportion of commodity $k$ routed through a specific path $p$ belonging to the predefined path set $\mathcal{P}_k$. The exact formulation of the objective function seeks to minimize the maximum path delay and the total energy simultaneously:
\begin{equation} \label{eq:original_obj}
	\begin{split}
		\min_{\mathbf{x}, \mathbf{y}, \mathbf{l}, \mathbf{p}} \quad & \alpha \max_{k \in \mathcal{K}, \, p \in \mathcal{P}_k} \left( y_p^k \sum_{(i,j)\in p}\frac{M^k}{r_{ij}} \right) \\
		& + (1-\alpha) \sum_{(i,j)\in \mathcal{E}} p_{ij} \frac{\sum_{k \in \mathcal{K}} x_{ij}^k M^k}{r_{ij}}.
	\end{split}
\end{equation}
For clarity in presenting our numerical evaluations in Section \ref{sec:numerical_results}, we introduce the shorthand notations $\hat{T} = \max_{k \in \mathcal{K}} \hat{T}_k$ to denote the network bottleneck delay, where $\hat{T}_k= \max_{ p \in \mathcal{P}_k} \left( y_p^k \sum_{(i,j)\in p}{M^k}/{r_{ij}} \right)$ represents the maximum delay of commodity $k$. Accordingly, the total network energy consumption is denoted as $\hat{E}= \sum_{(i,j)\in \mathcal{E}} p_{ij} {\sum_{k \in \mathcal{K}} x_{ij}^k M^k}/{r_{ij}}$.

However, directly solving the formulation \eqref{eq:original_obj} is computationally intractable due to the exponential cardinality of the candidate path set $\mathcal{P}_k$. Although modeled with continuous variables, this exponential dimensionality precludes direct optimization. Furthermore, to couple the path-based delay with the link-based energy consumption, the objective necessitates a global mapping constraint, i.e., $x_{ij}^k = \sum_{p \in \mathcal{P}_k: (i,j) \in p} y_p^k$. Such explicit mapping breaks the block-separable structure inherent to node-arc formulations, thereby invalidating standard decomposition methods. From a numerical optimization perspective, applying an LSE smooth approximation to the $\max(\cdot)$ operator over this exponentially large set $\mathcal{P}_k$ induces severe numerical overflow.

To resolve these bottlenecks, we approximate the exact path-based delay using arc-based variables. Specifically, we replace the single-path delay $y_p^k \sum_{(i,j)\in p}{M^k}/{r_{ij}}$ with the aggregate transmission time $\sum_{(i,j)\in \mathcal{E}} x_{ij}^k {M^k}/{r_{ij}}$. Because the total time consumed across all utilized links is  greater than or equal to the delay of any single end-to-end path, this summation establishes an upper bound for the original objective \eqref{eq:original_obj} and  eliminates the dependency on the intractable path set $\mathcal{P}_k$.  Crucially, unlike traditional formulations that sum times across all commodities (which masks the bottleneck flow), our aggregate time is calculated independently for each individual commodity $k$. By preserving the outer $\max_{k \in \mathcal{K}}$ operator, our objective identifies the bottleneck flow and minimizes the upper bound of its maximum transmission time.

Building upon this upper-bound approximation, we formulate the JRRA problem in multi-hop overlay D2D networks as the following continuous nonlinear programming  model:
\begin{subequations} \label{d2d_form_app}
	\begin{align}
		\min_{\mathbf{x}, \mathbf{l}, \mathbf{p}} \quad 
		& \alpha \max_{k \in \mathcal{K}} \sum_{(i,j)\in \mathcal{E}} x_{ij}^k \frac{M^k}{r_{ij}}
		+ (1-\alpha) \sum_{(i,j)\in \mathcal{E}} p_{ij} \frac{\sum_{k \in \mathcal{K}} x_{ij}^k M^k}{r_{ij}} \label{obj_app}\\
		\text{s.t.} \quad 
		& A \mathbf{x}^{k} = \mathbf{b}^k, \quad \forall k \in \mathcal{K}, \label{flow_cons}\\
		& r_{ij} = l_{ij} \log_2\!\left(1+\frac{p_{ij} h_{ij}}{N_{0}l_{ij}}\right),
		\quad \forall(i, j) \in \mathcal{E}, \label{rate_cons}\\
		& \sum_{(i,j)\in \mathcal{E}} l_{ij} \leq B, \label{bandwidth_cons}\\
		& \sum_{j \in \mathcal{V}_{out}(i)} p_{ij} \leq P_i^{\max}, \quad \forall i \in \mathcal{V}, \label{power_cons}\\
		& x_{ij}^k \geq 0, \quad \forall k \in \mathcal{K}, \ \forall(i, j) \in \mathcal{E}, \label{fraction_cons}\\
		& l_{ij} \geq 0, \ p_{ij} \geq 0, \quad \forall(i, j) \in \mathcal{E}. \label{nonneg_cons}
	\end{align}
\end{subequations}

The objective \eqref{obj_app} minimizes a generalized cost function, balancing the maximum delay across all commodities and the total energy consumption. Constraints \eqref{flow_cons} enforce the flow conservation, where $A$ is the node-arc incidence matrix and $\mathbf{b}^k$ denotes the normalized source-destination indicator vector for commodity $k$. The cross-layer coupling is captured by Shannon-capacity formula \eqref{rate_cons}, which maps the physical-layer resources $(l_{ij}, p_{ij})$ to the achievable network-layer rate $r_{ij}$. Furthermore, \eqref{bandwidth_cons} and \eqref{power_cons} impose the resource budgets, bounding the aggregate bandwidth allocation by $B$ and  limiting the cumulative transmission power over the outgoing neighbour set $\mathcal{V}_{out}(i)$ of each node $i$ to its budget $P_i^{\max}$. Finally, \eqref{fraction_cons} and \eqref{nonneg_cons} establish the non-negativity domain for all optimization variables.

Although formulation \eqref{d2d_form_app} bypasses the exponential dimensionality of explicit path enumeration, direct optimization remains precluded by two mathematical challenges: non-smoothness and non-convexity. First, the delay metric introduces the $\max_{k \in \mathcal{K}}(\cdot)$ operator, rendering the objective non-differentiable at cusps where multiple commodities achieve identical transmission times. While subgradient methods can address non-smoothness, deploying them in highly non-convex domains typically induces severe iterative zig-zagging and fails to provide rigorous convergence guarantees to a stationary point. Second, the objective \eqref{obj_app} introduces a highly non-linear coupling between the routing variable $x_{ij}^k$ and physical-layer resources $(l_{ij}, p_{ij})$ through the transimission rate $ r_{ij} $. Although optimizing the physical-layer resources under a fixed routing strategy yields a convex subproblem (as proven in Section \ref{sec:algo}), the joint continuous domain remains highly non-convex. Consequently, achieving algorithmic tractability necessitates a smooth approximation of the max operator to ensure continuous differentiability, followed by a BCD framework to jointly optimize the cross-layer variables.

\section{BCD framework for joint optimization}
\label{sec:algo}
To tackle the intractable non-convexity and variable coupling in the joint optimization problem \eqref{d2d_form_app}, we decouple this problem via a BCD framework, iteratively alternating between the network and physical layers.  For the network-layer routing subproblem, we propose two customized algorithms to accommodate diverse scenario requirements: the MF-FW algorithm tailored for fast execution by leveraging the first-order gradient direction, and the LR-PDIPM that bypasses dense matrix inversions via  SM formula for high-precision scenarios. For the physical-layer resource allocation, we employ a time-domain transformation to expose its hidden convexity. This subproblem is  exactly solved via a parallel dual ascent (PDA) algorithm based on the Lagrangian decomposition.  The rigorous theoretical analyses about these proposed algorithms are deferred to Section \ref{sec:convergence}.

\subsection{Path routing subproblem and smooth approximation}
With the physical-layer resources $(\mathbf{l}, \mathbf{p})$ fixed, the overall joint optimization \eqref{d2d_form_app} reduces to a multi-commodity network flow (MCNF) model. As established previously, the $\max_{k \in \mathcal{K}}(\cdot)$ operator renders objective function \eqref{obj_app} non-differentiable. To restore gradient tractability without compromising the bottleneck-optimization objective, we employ LSE  smoothing \cite{boyd2004convex}. By approximating the max operator as $\frac{1}{\mu} \ln \big( \sum_{k} \exp^{\mu (\cdot)} \big)$, we obtain the following differentiable and  convex MCNF formulation:
\begin{subequations}\label{eq:smoothed_routing}
	\begin{align}
		\min_{\mathbf{x}} \quad & F(\mathbf{x}) = \frac{\alpha}{\mu} \ln \left( \sum_{k \in \mathcal{K}} \exp^{\mu T_k(\mathbf{x})} \right) + \sum_{k \in \mathcal{K}} \sum_{(i,j) \in \mathcal{E}}  c_{ij}^{\text{E}} x_{ij}^k \label{obj_smoothed} \\
		\text{s.t.} \quad & A \mathbf{x}^{k} = \mathbf{b}^k, \quad \forall k \in \mathcal{K}, \\
		& x_{ij}^k \ge 0, \quad \forall k \in \mathcal{K}, \forall (i,j) \in \mathcal{E},
	\end{align}
\end{subequations}
where the parameter $\mu > 0$ dictates the approximation tightness,  $T_k(\mathbf{x}) = \sum_{(i,j)\in \mathcal{E}} x_{ij}^k {M^k}/{r_{ij}}$ is the aggregate transmission time for commodity $k$, and $c_{ij}^{\text{E}} = (1-\alpha) M^k {p_{ij}}/{r_{ij}}$ denotes the fixed link energy cost.

To solve \eqref{eq:smoothed_routing}, we exploit the structural sparsity of its Hessian matrix. Let $\beta_k(\mathbf{x}) = \exp^{\mu T_k} / \sum_m \exp^{\mu T_m}$ define the LSE gradient component. Despite the exponential inter-flow coupling introduced by LSE, the exact Hessian $\mathbf{H} = \nabla^2 F(\mathbf{x})$ elegantly decomposes into a block-diagonal matrix minus a rank-one correction:
\begin{equation} \label{eq:hessian_structure}
	\mathbf{H} = \alpha \mu \left( \mathbf{W} - \mathbf{v}\mathbf{v}^{\mathsf{T}} \right).
\end{equation}
Specifically, let the indices corresponding to the routing variables $x_{ij}^k$ and $x_{mn}^{k'}$ be denoted by the tuples $(k, ij)$ and $(k', mn)$. The elements of the block-diagonal matrix $\mathbf{W}$ and the column vector $\mathbf{v}$ are given by:
\begin{equation} \label{eq:hessian_elements}
	\begin{split}
			W_{(k, ij), (k', mn)} &= 
		\begin{cases} 
			\beta_k \left( {M^k}/{r_{ij}} \right) \left( {M^k}/{r_{mn}} \right), & \text{if } k = k' \\ 
			0, & \text{otherwise}
		\end{cases},\\
		v_{(k, ij)} &= \beta_k \left( {M^k}/{r_{ij}} \right).
	\end{split}
\end{equation}
Here, the block-diagonal term $\mathbf{W}$  captures the independent intra-commodity routing curvature, while the rank-one correction $\mathbf{v}\mathbf{v}^{\mathsf{T}}$ encapsulates the global cross-commodity coupling.  Leveraging \eqref{eq:hessian_structure} and \eqref{eq:hessian_elements}, we propose two customized algorithms.

\subsubsection{MF-FW algorithm for path routing}
For the requirement of fast execution, we propose a MF-FW algorithm to solve path routing subproblem \eqref{eq:smoothed_routing}, detailed in Algorithm \ref{algo:fw}

To prevent the zero-gradient stalling  phenomenon associated with non-bottleneck commodities, we augment the LSE objective \eqref{obj_app} with a linear transmission time regularization term, scaled by $\epsilon > 0$ (e.g., $10^{-6}$). The gradient-based link weights in \eqref{eq:fw_weights}  incorporate this regularization:
\begin{equation} \label{eq:fw_weights}
	w_{ij}^{k,(n)} = \alpha (\beta_k(\mathbf{x}^{(n)}) + \epsilon) \frac{M^k}{r_{ij}} + c_{ij}^{\text{E}}.
\end{equation}
Crucially, since this regularization introduces only a linear penalty, the second-order curvature of the objective remains unaffected. This preserves the exactness of the matrix-free Hessian computation in \eqref{eq:fw_curvature}.
Solving $|\mathcal{K}|$ independent shortest-path problems using these dynamic weights $ w_{ij}^{k,(n)} $ yields an auxiliary flow $\mathbf{y}^{(n)}$ and a descent direction $\mathbf{d}^{(n)} = \mathbf{y}^{(n)} - \mathbf{x}^{(n)}$. The algorithm's convergence is monitored via the standard Frank-Wolfe (FW) duality gap:
\begin{equation} \label{eq:fw_gap}
	\text{Gap}^{(n)} = -\nabla F(\mathbf{x}^{(n)})^{\mathsf{T}} \mathbf{d}^{(n)}.
\end{equation}

To suppress the zig-zagging phenomenon inherent to the standard FW algorithm near the optimum, we employ the directional curvature $C^{(n)} = (\mathbf{d}^{(n)})^{\mathsf{T}} \mathbf{H} \mathbf{d}^{(n)}$ to determine the step size. Using the rank-one correction structure in \eqref{eq:hessian_structure}, $C^{(n)}$ is calculated exactly in $\mathcal{O}(|\mathcal{K}|)$ operations using the pre-aggregated directional differences $u_k = \sum_{(i,j) \in \mathcal{E}} d_{ij}^k {M^k}/{r_{ij}}$:
\begin{equation} \label{eq:fw_curvature}
	C^{(n)} = \alpha \mu \left[ \sum_{k \in \mathcal{K}} \beta_k (u_k)^2 - \left( \sum_{k \in \mathcal{K}} \beta_k u_k \right)^2 \right].
\end{equation}

This algorithm then determines a truncated Newton-based step size:
\begin{equation} \label{eq:fw_step}
	\lambda^* = \min \left\{ 1, \frac{\text{Gap}^{(n)}}{C^{(n)}} \right\}.
\end{equation}
Crucially, the strict convexity of the quadratic approximation for objective \eqref{obj_smoothed} guarantees that this quadratic function of $\lambda$ decreases monotonically on the interval $[0,1]$. Thus, the truncation enforces $\lambda^* = 1$ whenever ${\text{Gap}^{(n)}}/{C^{(n)}} > 1$.

\begin{algorithm}[htbp]
	\caption{MF-FW algorithm for path routing} \label{algo:fw}
	\textbf{Input:} Network graph $\mathcal{G}(\mathcal{V}, \mathcal{E})$, demands $\{M^k\}$, initial flow $\mathbf{x}^{(0)}$, physical resources $(\mathbf{l}, \mathbf{p})$, weight $\alpha$, smoothing $\mu$, tolerance $\varepsilon$, bounds $N_{\max}$. \\
	\textbf{Initialize:} Set $n \leftarrow 0$, $\text{Gap}^{(0)} \leftarrow \infty$.
	\begin{algorithmic}[1]
		\While{ $n < N_{\max}$}
		\State Compute link weights $w_{ij}^{k,(n)}$  via \eqref{eq:fw_weights}.
		\State For each $k \in \mathcal{K}$, solve the shortest path problem using weights $w_{ij}^{k,(n)}$ to construct auxiliary flow $\mathbf{y}^{(n)}$.
		\State Determine the descent direction $\mathbf{d}^{(n)} = \mathbf{y}^{(n)} - \mathbf{x}^{(n)}$.
		
		\State Compute the FW duality gap $\text{Gap}^{(n)}$ via \eqref{eq:fw_gap}.
		\State \textbf{if} $\text{Gap}^{(n)} \le \varepsilon$ \textbf{then} \textbf{break}
		
		\State Compute directional time variations $u_k$ and the directional curvature $C^{(n)}$ via \eqref{eq:fw_curvature}.
		\State Compute Newton-based step size $\lambda^*$ via \eqref{eq:fw_step}.
		\State $\mathbf{x}^{(n+1)} \leftarrow \mathbf{x}^{(n)} + \lambda^* \mathbf{d}^{(n)}$, and $n \leftarrow n + 1$.
		\EndWhile
		\State \textbf{Output:} Extract optimal loop-free flows $ \mathbf{x}^* $ from $\mathbf{x}^{(n)}$ via BFS.
	\end{algorithmic}
\end{algorithm}
Due to the nature of convex combinations in standard FW, the flow  $\mathbf{x}^{(n)}$ may contain microscopic topological loops. In the output stage, we implement a standard BFS-based flow decomposition to extract loop-free flows.

\subsubsection{LR-PDIPM for path routing}
For the requirement of high-precision solutions, we propose a tailored LR-PDIPM, detailed in Algorithm \ref{algo:pd_ipm}. 
The core computational bottleneck of conventional IPMs lies in solving the KKT Newton system \cite{gondzio2012interior,gondzio2025interior}, which necessitates the dense matrix inversion of the modified Hessian $\boldsymbol{\Theta} = \mathbf{H} + \mathbf{X}^{-1}\mathbf{S}$. Here, $\mathbf{X} = \text{diag}(\mathbf{x})$ and $\mathbf{S} = \text{diag}(\mathbf{s})$  denote the diagonal matrices constructed from the primal variable vector $\mathbf{x}$ and the corresponding dual slack vector $\mathbf{s}$, respectively.

Substituting \eqref{eq:hessian_structure}, the modified Hessian inherits a rank-one correction: $\boldsymbol{\Theta} = \mathbf{D} - \tilde{\mathbf{v}}\tilde{\mathbf{v}}^{\mathsf{T}}$, where the base matrix $\mathbf{D} = \alpha \mu \mathbf{W} + \mathbf{X}^{-1}\mathbf{S}$ remains block-diagonal, and the scaled vector $\tilde{\mathbf{v}} = \sqrt{\alpha \mu} \mathbf{v}$. Let $\mathbf{A} = I_{|\mathcal{K}|} \otimes A$ denote the block-diagonal incidence matrix spanning all $|\mathcal{K}|$ commodities, distinguishing it from the single-commodity incidence matrix $A$. Finding the Newton direction requires solving the Schur complement system:
\begin{equation} \label{eq:schur}
	(\mathbf{A} \boldsymbol{\Theta}^{-1} \mathbf{A}^{\mathsf{T}}) \Delta \mathbf{y} = \mathbf{g},
\end{equation}
where $ \Delta \mathbf{y} $ is the dual step and RHS vector $\mathbf{g}$ is defined in \eqref{eq:rhs_g}.
Furthermore, the standard primal, dual, and central path KKT residuals are defined as:
\begin{equation} \label{eq:kkt_residuals}
	\mathbf{r}_p = \mathbf{A}\mathbf{x} - \mathbf{b}, \quad \mathbf{r}_d = -\nabla F(\mathbf{x}) + \mathbf{A}^{\mathsf{T}} \mathbf{y} + \mathbf{s}, \quad \mathbf{r}_c = \mathbf{X} \mathbf{S} \mathbf{1} - \tau \mathbf{1}.
\end{equation}

Direct factorization of the system \eqref{eq:schur} is computationally prohibitive, demanding $\mathcal{O}(|\mathcal{K}|^3 |\mathcal{V}|^3)$ operations. Instead, we invert $\boldsymbol{\Theta}$ via the SM formula: $\boldsymbol{\Theta}^{-1} = \mathbf{D}^{-1} + \rho \hat{\mathbf{v}} \hat{\mathbf{v}}^{\mathsf{T}}$, where $\hat{\mathbf{v}} = \mathbf{D}^{-1}\tilde{\mathbf{v}}$ and $\rho = (1 - \tilde{\mathbf{v}}^{\mathsf{T}} \hat{\mathbf{v}})^{-1}$. This decouples the Schur matrix into $\mathbf{M}_0 + \rho \mathbf{z}\mathbf{z}^{\mathsf{T}}$, where $\mathbf{M}_0 = \mathbf{A} \mathbf{D}^{-1} \mathbf{A}^{\mathsf{T}}$ and $\mathbf{z} = \mathbf{A} \hat{\mathbf{v}}$. To maintain this low-rank structure on the right-hand side, we construct the composite RHS vector $\mathbf{g}$ using the intermediate residual $\mathbf{v}_{\text{rhs}} = \mathbf{r}_d - X^{-1}\mathbf{r}_c$:
\begin{equation} \label{eq:rhs_g}
	\mathbf{g} = -\mathbf{r}_p - \mathbf{A} \mathbf{D}^{-1} \mathbf{v}_{\text{rhs}} - \rho (\hat{\mathbf{v}}^{\mathsf{T}} \mathbf{v}_{\text{rhs}}) \mathbf{z}.
\end{equation}
Crucially, because both $\mathbf{D}^{-1}$ and $\mathbf{A}$ are block-diagonal, the matrix $\mathbf{M}_0$ decomposes into $|\mathcal{K}|$ independent sparse graph Laplacians. By solving $\mathbf{M}_0 \hat{\mathbf{g}} = \mathbf{g}$ and $\mathbf{M}_0 \hat{\mathbf{z}} = \mathbf{z}$ via parallel Cholesky factorizations, the exact dual step $\Delta \mathbf{y}$ is recovered in merely $\mathcal{O}(|\mathcal{K}| |\mathcal{V}|^3)$ operations:
\begin{equation} \label{eq:dy_recovery}
	\Delta \mathbf{y} = \hat{\mathbf{g}} - \frac{\rho}{1 + \rho \mathbf{z}^{\mathsf{T}} \hat{\mathbf{z}}} (\mathbf{z}^{\mathsf{T}} \hat{\mathbf{g}}) \hat{\mathbf{z}}.
\end{equation}

Note that the block-diagonal matrix $\mathbf{D}_k$  embeds a local rank-one correction structure:
\begin{equation} \label{eq:local_rank1}
	\mathbf{D}_k = \mathbf{\Lambda}_k + \mathbf{u}_k \mathbf{u}_k^{\mathsf{T}},
\end{equation}
where $\mathbf{\Lambda}_k = \mathbf{X_k}^{-1}\mathbf{S_k}$ is a diagonal matrix, and $\mathbf{u}_k = \sqrt{\alpha \mu \beta_k} \mathbf{d}_k$ with its elements defined by $[\mathbf{d}_k]_{(i,j)} = M^k/r_{ij}$. Because $\mathbf{\Lambda}_k$ is diagonal, its inverse is trivial. Consequently, the block inverse $\mathbf{D}_k^{-1}$ can be computed in linear $\mathcal{O}(|\mathcal{E}|)$ operations via the SM formula:
\begin{equation} \label{eq:local_sm}
	\mathbf{D}_k^{-1} = \mathbf{\Lambda}_k^{-1} - \frac{1}{1 + \mathbf{u}_k^{\mathsf{T}} \mathbf{\Lambda}_k^{-1} \mathbf{u}_k} (\mathbf{\Lambda}_k^{-1}\mathbf{u}_k)(\mathbf{\Lambda}_k^{-1}\mathbf{u}_k)^{\mathsf{T}}.
\end{equation}
Since the IPM maintains $\mathbf{X} \succ 0$ and $\mathbf{S} \succ 0$, the scalar term $\mathbf{u}_k^{\mathsf{T}} \mathbf{\Lambda}_k^{-1} \mathbf{u}_k$ is positive. Subsequently, the primal and dual slack steps are recovered via straightforward back-substitution. This computation decomposes into sparse matrix-vector multiplications and linear-time inner products, utilizing the auxiliary residual $\mathbf{v}_{\text{step}} = \mathbf{A}^{\mathsf{T}} \Delta \mathbf{y} + \mathbf{v}_{\text{rhs}}$:
\begin{equation} \label{eq:dx_ds_recovery}
	\Delta \mathbf{x} = \mathbf{D}^{-1} \mathbf{v}_{\text{step}} + \rho (\hat{\mathbf{v}}^{\mathsf{T}} \mathbf{v}_{\text{step}}) \hat{\mathbf{v}}, \quad \Delta \mathbf{s} = -\mathbf{X}^{-1}\mathbf{r}_c - \mathbf{X}^{-1}\mathbf{S} \Delta \mathbf{x}.
\end{equation}
Finally, to maintain strict positivity of $ (\mathbf{x}, \mathbf{s}) $  and prevent severe step size truncation, Algorithm \ref{algo:pd_ipm} employs a wide neighborhood fraction-to-the-boundary rule. Let $\lambda \in (0, 1]$ denote the primal-dual step size and $\mathcal{F}^0$ denote the strictly feasible primal-dual interior set, defined as
\begin{equation}
	\mathcal{F}^0 = \left\{ (\mathbf{x}, \mathbf{y}, \mathbf{s}) \mid \mathbf{A}\mathbf{x} = \mathbf{b}, \mathbf{A}^{\mathsf{T}}\mathbf{y} + \mathbf{s} = \nabla f(\mathbf{x}), \mathbf{x} > \mathbf{0}, \mathbf{s} > \mathbf{0} \right\}.
\end{equation}
Consequently, the fraction-to-the-boundary rule confines the iterates within the wide neighborhood $\mathcal{N}_{-\infty}(\gamma)$, formulated as
\begin{equation}
	\mathcal{N}_{-\infty}(\gamma) = \left\{ (\mathbf{x}, \mathbf{y}, \mathbf{s}) \in \mathcal{F}^0 \mid x_i s_i \ge \gamma \mu_{\text{gap}}^{(n)} \right\},
\end{equation}
where $\mu_{\text{gap}}^{(n)} = (\mathbf{x}^{\mathsf{T}}\mathbf{s})/|\mathcal{K}||\mathcal{E}|$ defines the current duality gap. Furthermore, the adaptive centering strategy is used to update the parameter $\sigma^{(n)}$ based on the step size $\lambda$. This design enforces strong centering when the search direction is blocked by the boundary ($\lambda \ll 1$), while facilitating near-full Newton steps ($\lambda \approx 1$) to achieve rapid convergence within $\mathcal{N}_{-\infty}(\gamma)$.

\begin{algorithm}[t]
	\caption{LR-PDIPM for path routing} \label{algo:pd_ipm}
	\textbf{Input:} Network graph $\mathcal{G}(\mathcal{V}, \mathcal{E})$, demands $\{M^k\}$, initial flow $\mathbf{x}^{(0)}> \mathbf{0}$, physical resources $(\mathbf{l}, \mathbf{p})$, weight $\alpha$, smoothing $\mu$, tolerance $\varepsilon$, bound $\gamma \in (0,1) $, $ \sigma_{\min},\sigma_{\max} \in (0,1) $. \\
	\textbf{Initialize:}  Set $(\mathbf{y}^{(0)}, \mathbf{s}^{(0)}) > \mathbf{0}$, $\sigma^{(0)} \in (0,1)$, $n \leftarrow 0$, $ \mu_{\text{gap}} \leftarrow \infty $.
	\begin{algorithmic}[1]
		\While{$\max(\|\mathbf{r}_p\|_{\infty}, \|\mathbf{r}_d\|_{\infty}, \mu_{\text{gap}}) > \varepsilon$}
		\State Compute $\mu_{\text{gap}} = {\mathbf{x}^{\mathsf{T}} \mathbf{s}}/{(|\mathcal{K}||\mathcal{E}|)}$ and $\tau = \sigma^{(n)} \mu_{\text{gap}}$.
		\State Compute KKT residuals $(\mathbf{r}_p, \mathbf{r}_d, \mathbf{r}_c)$ via \eqref{eq:kkt_residuals}.
		
		\State For each block $k \in \mathcal{K}$, compute block inverse $\mathbf{D}_k^{-1}$ via \eqref{eq:local_sm}.
		\State $\mathbf{D}^{-1} = \text{blkdiag}(\mathbf{D}_1^{-1}, \dots, \mathbf{D}_K^{-1})$.
		
		\State Compute $\tilde{\mathbf{v}} = \sqrt{\alpha \mu} \mathbf{v}, \hat{\mathbf{v}} = \mathbf{D}^{-1}\tilde{\mathbf{v}}$, scalar $\rho = (1 - \tilde{\mathbf{v}}^{\mathsf{T}} \hat{\mathbf{v}})^{-1}$ and auxiliary vector $\mathbf{z} = \mathbf{A}\hat{\mathbf{v}}$ via \eqref{eq:hessian_elements}.
		\State Assemble the composite RHS $\mathbf{g}$ via \eqref{eq:rhs_g}.
		\State \textbf{For} each block $k \in \mathcal{K}$ \textbf{in parallel}:
		\State \quad Form $\mathbf{M}_{0,k} = \mathbf{A}_k \mathbf{D}_k^{-1} \mathbf{A}_k^{\mathsf{T}}$.
		\State \quad Solve $\mathbf{M}_{0,k} \hat{\mathbf{g}}_k = \mathbf{g}_k$ and $\mathbf{M}_{0,k} \hat{\mathbf{z}}_k = \mathbf{z}_k$ via sparse Cholesky.
		\State \textbf{End For}
		\State Stitch vectors $ \hat{\mathbf{g}}_k $ and $ \hat{\mathbf{z}}_k $ to form global $\hat{\mathbf{g}}$ and $\hat{\mathbf{z}}$.
		\State Recover dual step $\Delta \mathbf{y}$ via \eqref{eq:dy_recovery}.
		\State Recover primal and slack steps $(\Delta \mathbf{x}, \Delta \mathbf{s})$ via \eqref{eq:dx_ds_recovery}.
		\State Compute maximum step size $\lambda_{\max}$:
		\State $\lambda_{\max} = \min \left( 1, \ \min_{i: \Delta x_i < 0} \left( {-x_i}/{\Delta x_i} \right), \ \min_{i: \Delta s_i < 0} \left( {-s_i}/{\Delta s_i} \right) \right)$
		\State Initialize $\lambda \leftarrow 0.99 \lambda_{\max}$.  \label{linesea_start}
		\State Compute $\mu_{\text{gap}}(\lambda) = (\mathbf{x} + \lambda \Delta \mathbf{x})^{\mathsf{T}}(\mathbf{s} + \lambda \Delta \mathbf{s}) / (|\mathcal{K}||\mathcal{E}|)$
		\While{$\exists i \in \{1, \dots, |\mathcal{K}||\mathcal{E}|\} \text{ s.t. } (x_i + \lambda \Delta x_i)(s_i + \lambda \Delta s_i) < \gamma \mu_{\text{gap}}(\lambda)$}
		\State Update $\lambda \leftarrow c \cdot \lambda, \quad \text{where } c \in (0,1)$
		\EndWhile  \label{linesea_end}
		\State $\sigma^{(n+1)} = \max \left( \sigma_{\min}, \min \left( \sigma_{\max}, (1 - \lambda)^2 \right) \right)$.
		\State $(\mathbf{x}, \mathbf{y}, \mathbf{s}) \leftarrow (\mathbf{x}, \mathbf{y}, \mathbf{s}) + \lambda (\Delta \mathbf{x}, \Delta \mathbf{y}, \Delta \mathbf{s})$, and $n \leftarrow n + 1$.
		\EndWhile
		\State \textbf{Output:} Optimal primal-dual pair $(\mathbf{x}^*, \mathbf{y}^*, \mathbf{s}^*)$.
	\end{algorithmic}
\end{algorithm}

\subsection{Resource allocation subproblem and convex reformulation}
Given a fixed routing strategy $\mathbf{x}^*$ obtained from the network-layer optimization \eqref{eq:smoothed_routing}, the joint formulation \eqref{d2d_form_app} reduces to the physical-layer resource allocation subproblem. Let $m_{ij} = \sum_{k \in \mathcal{K}} x_{ij}^k M^k$ denote the aggregate traffic load assigned to link $(i,j)$. The primal subproblem is formulated as:
\begin{subequations} \label{eq:resource_primal}
	\begin{align}
		\min_{\mathbf{l}, \mathbf{p}, T} \quad 
		& \alpha T + (1-\alpha) \sum_{(i,j)\in \mathcal{E}} m_{ij} \frac{p_{ij}}{r_{ij}} \\
		\text{s.t.} \quad 
		& \sum_{(i,j)\in \mathcal{E}} x_{ij}^k \frac{M^k}{r_{ij}} \le T, \quad \forall k \in \mathcal{K}, \label{con_delay_primal} \\
		& \sum_{(i,j)\in \mathcal{E}} l_{ij} \leq B, \label{con_spec}\\
		&\sum_{j \in \mathcal{V}_{out}(i)} p_{ij} \leq P_i^{\max}, \quad \forall i \in \mathcal{V}, \label{con_power} \\
		& l_{ij} \geq 0, \quad p_{ij} \geq 0, \quad \forall (i,j) \in \mathcal{E},
	\end{align}
\end{subequations}
where the transmission rate $r_{ij} = l_{ij} \log_2(1 + {p_{ij}h_{ij}}/{N_0 l_{ij}})$  couples the bandwidth and power variables. This equation renders both the transmission time $1/r_{ij}$ in \eqref{con_delay_primal} and the energy consumption term $m_{ij} p_{ij}/r_{ij}$ highly non-convex with respect to $(l_{ij}, p_{ij})$, prohibiting direct global optimization.

To overcome this intractability, we project the resource allocation space from the power domain into the time domain. Let $t_{ij} = m_{ij}/r_{ij} > 0$  define the transmission time required to deliver the traffic load $m_{ij}$ over active link $(i,j)$. Substituting $r_{ij} = m_{ij}/t_{ij}$ yields $m_{ij}/t_{ij} = l_{ij} \log_2(1 + \frac{p_{ij}h_{ij}}{N_0 l_{ij}})$. Inverting this relation isolates the transmission power as a function of allocated bandwidth and transmission time:
\begin{equation} \label{eq:p_convex_func}
	p_{ij}(l_{ij}, t_{ij}) = \frac{N_0 l_{ij}}{h_{ij}} \left( 2^{{m_{ij}}/{l_{ij} t_{ij}}} - 1 \right).
\end{equation}
This time-domain transformation exposes a hidden convexity governed by the properties of perspective functions and convex composition.  In the following content, we optimize the subproblem \eqref{eq:resource_primal} over $ (l_{ij}, t_{ij}) $ and then recover the power variable $ p_{ij} $ via \eqref{eq:p_convex_func}.
\begin{lemma}[Convexity of perspective composite] \label{lemma:core_func}
	If $g(x)$ is a convex function, its perspective transformation $P(u,v) = v g(u/v)$ preserves joint convexity for all $v>0$. Consequently, the bivariate function $f(u, v) = uv(2^{1/uv}-1)$ is strictly and jointly convex for all $u, v > 0$.
\end{lemma}
\textit{Proof:} See \ref{app:proof_lemma_core}.

Leveraging Lemma \ref{lemma:core_func}, we linearize the delay constraints and convexify the power constraints, yielding the final convex formulation:
\begin{subequations} \label{eq:convex_global}
	\begin{align}
		\min_{\mathbf{l}, \mathbf{t}, T} \quad & \alpha T + (1-\alpha) \sum_{(i,j)\in \mathcal{E}} \frac{N_0}{h_{ij}} l_{ij} t_{ij} \left( 2^{{m_{ij}}/{l_{ij} t_{ij}}} - 1 \right) \label{con_obj_global} \\
		\text{s.t.} \quad & \sum_{(i,j)\in \mathcal{E}} \rho_{ij}^k t_{ij} \leq T, \quad \forall k \in \mathcal{K}, \label{con_delay_global} \\
		& \sum_{(i,j)\in \mathcal{E}} l_{ij} \leq B,\label{con_spec_global}  \\
		&\sum_{j \in \mathcal{V}_{out}(i)} p_{ij}(l_{ij}, t_{ij}) \leq P_i^{\max}, \quad \forall i \in \mathcal{V}, \label{con_power_global} \\
		& l_{ij} \ge 0, \quad t_{ij} > 0, \quad \forall (i,j)\in \mathcal{E},
	\end{align}
\end{subequations}
where the routing coefficient  $\rho_{ij}^k = x_{ij}^k M^k / m_{ij}$. 
To establish the theoretical foundation for global optimization, we formally prove the convex properties of \eqref{eq:convex_global}.

\begin{theorem}[Convexity of the time-domain formulation] \label{thm:convexity}
	The transformed time-domain resource allocation formulation \eqref{eq:convex_global} is a continuous convex optimization problem.
\end{theorem}
\begin{proof}[Proof sketch]
	The proof relies on establishing the joint convexity of the  non-linear energy and power functions via affine domain mapping and the scalar composition theorem. Detailed  derivations are provided in \ref{app:proof_thm_convexity}.
\end{proof}

Despite the theoretical convexity established in Theorem \ref{thm:convexity}, directly solving \eqref{eq:convex_global} using standard interior-point solvers (e.g., MOSEK) is practically infeasible. Modern solvers require models to be formulated exclusively using standard cone representations, primarily the exponential cone $\mathcal{K}_{\exp} = \{ (x_1, x_2, x_3) \mid x_1 \ge x_2 \exp(x_3 / x_2), x_2 > 0 \}$ and the rotated quadratic cone $\mathcal{K}_{\text{rquad}} = \{ (x_1, x_2, x_3) \mid 2 x_1 x_2 \ge x_3^2 \}$.

To map the power and energy functions into these cones, the bandwidth-time product $l_{ij} t_{ij}$ must be decoupled. If we introduce an auxiliary variable $Q_{ij}$ to replace the product, the required constraint $Q_{ij} \leq l_{ij} t_{ij}$  defines a non-convex region. Alternatively, representing the product via a squared variable $q_{ij}^2$ forms a valid rotated quadratic cone $2 l_{ij} t_{ij} \ge (\sqrt{2} q_{ij})^2$. However, substituting this squared term $q_{ij}^2$ into the energy function violates the 
requirement of the exponential cone $\mathcal{K}_{\exp}$, failing to obtain global optimum. Consequently, reformulating this model into solver-compatible cones inevitably introduces non-convex constraints, making standard  solvers inapplicable.

To circumvent the solver's incompatibility and  exploit the inherent convexity and separability of the network resources, we develop a PDA algorithm via Lagrangian dual decomposition.

\subsubsection{Parallel solution via dual decomposition}
By associating non-negative dual multipliers $\theta^k, \omega,$ and $\varphi_i$ with the delay, total bandwidth, and per-node power constraints respectively, the Lagrangian function associated with the convex formulation \eqref{eq:convex_global} is defined as:
\begin{align} \label{eq:lagrangian_func}
	\mathcal{L} &= \left( \alpha - \sum_{k \in \mathcal{K}} \theta^k \right) T - \omega B - \sum_{i \in \mathcal{V}} \varphi_i P_i^{\max} \nonumber \\
	&\quad + \sum_{(i,j)\in \mathcal{E}_{act}} \underbrace{\left[ (1-\alpha){E}_{ij} + \varphi_i p_{ij} + \omega l_{ij} + \left( \sum_{k \in \mathcal{K}} \theta^k \rho_{ij}^k \right) t_{ij} \right]}_{\Phi_{ij}(l_{ij}, t_{ij})},
\end{align}
where $\mathcal{E}_{act} \subseteq \mathcal{E}$ denotes the subset of active links with non-zero traffic load ($m_{ij} > 0$), and ${E}_{ij}= p_{ij} t_{ij}$ represents the energy consumption function of link $(i,j)$. To prevent numerical singularities (e.g., division by zero), the resource allocation is exclusively executed over the active link set $\mathcal{E}_{act}$, while inactive links are pruned.

Applying the KKT stationarity condition with respect to the auxiliary variable $T$ demands that $\frac{\partial \mathcal{L}}{\partial T} = 0$, restricting the delay multipliers to a scaled simplex: $\sum_{k \in \mathcal{K}} \theta^k = \alpha$. Consequently, the global $T$ term vanishes from the Lagrangian, decoupling the primal minimization into $|\mathcal{E}|$ independent, link-level 2D subproblems. 

Given fixed dual variables, each link  minimizes its aggregate cost function $\Phi_{ij}(l_{ij}, t_{ij})$ in parallel. Because $\Phi_{ij}$ is continuously differentiable and jointly convex in $(l_{ij}, t_{ij})$, its global minimum can be  obtained via coordinate descent. 
Specifically, for a fixed transmission time $t_{ij}$, the strict convexity of the link-level subproblem (Theorem \ref{thm:convexity}) guarantees the strict monotonicity of its partial derivative $\frac{\partial \Phi_{ij}}{\partial l_{ij}}$. This property ensures a unique root for the stationarity condition $\frac{\partial \Phi_{ij}}{\partial l_{ij}} = 0$, allowing the optimal bandwidth $l_{ij}$ to be exactly located via  1D bisection search.  A symmetric bisection procedure is subsequently applied to optimize $t_{ij}$ for a fixed $l_{ij}$. This coordinate descent method alternates iteratively until the link-level subproblem convergence is achieved. 
Once the primal variables $(l_{ij}, t_{ij})$ are  optimized, the problem \eqref{eq:convex_global} progressively approaches the global optimum by updating the dual variables via projected gradient ascent. To mitigate the numerical ill-conditioning caused by heterogeneous parameter scales (e.g., MHz for bandwidth versus mW for power) and to suppress severe dual multiplier oscillations during early iterations, we employ normalized relative gradients and non-negative projections. Let $\delta^{(n)}$ denote the step size satisfying standard divergent sum conditions (e.g., $\delta^{(n)} = \delta_0 / (n+1)$). The bandwidth, nodal power, and delay multipliers are updated as:
\begin{subequations} \label{eq:dual_update_resources}
	\begin{align}
			\omega^{(n+1)} &= \left[ \omega^{(n)} + \delta^{(n)} \left( \frac{\sum_{(i,j)\in \mathcal{E}_{act}} l_{ij}^{(n)} - B}{B} \right) \right]_+, \label{eq:update_omega} \\
			\varphi_i^{(n+1)} &= \left[ \varphi_i^{(n)} + \delta^{(n)} \left( \frac{\sum_{j \in \mathcal{V}_{out}(i)} p_{ij}^{(n)} - P_i^{\max}}{P_i^{\max}} \right) \right]_+, \quad \forall i \in \mathcal{V}, \label{eq:update_phi} \\
			\boldsymbol{\theta}^{(n+1)} &= \mathcal{P}_{\Delta} \left( \boldsymbol{\theta}^{(n)} + \delta^{(n)} \nabla_{\boldsymbol{\theta}} \mathcal{L} \right), \label{eq:update_theta}
		\end{align}
\end{subequations}
where $[\cdot]_+$ denotes the non-negative projection, $\nabla_{\theta^k} \mathcal{L} = \sum_{(i,j)\in \mathcal{E}_{act}} \rho_{ij}^k t_{ij}^{(n)}$, and $\mathcal{P}_{\Delta}$ projects the vector $\boldsymbol{\theta}$ onto the scaled simplex $\sum_{k \in \mathcal{K}} \theta^k = \alpha$. 
Specifically, the projection operator $\mathcal{P}_{\Delta}$ is evaluated by first computing the unprojected step $\tilde{\theta}^{k} = \theta^{k,(n)} + \delta^{(n)} \sum_{(i,j)\in \mathcal{E}_{act}} \rho_{ij}^k t_{ij}^{(n)}$. The exact projection is then achieved by finding a unique root $\nu^*$ such that:
\begin{equation} \label{eq:simplex_projection}
	\sum_{k \in \mathcal{K}} \left[ \tilde{\theta}^{k} - \nu^* \right]_+ = \alpha \implies \theta^{k,(n+1)} = \left[ \tilde{\theta}^{k} - \nu^* \right]_+, \quad \forall k \in \mathcal{K}.
\end{equation}
This scalar root-finding is a standard projection onto a simplex, which can be exactly solved in $\mathcal{O}(|\mathcal{K}| \log |\mathcal{K}|)$  via classical sorting-based algorithms.

To evaluate the convergence, we define the primal constraint violation $\Delta_{\text{prim}}$ and the dual stability metric $\Delta_{\text{dual}}$ based on the relative infinity norm:
\begin{subequations} \label{eq:stopping_criteria}
	\begin{align}
		\Delta_{\text{prim}} &= \max \left( \left[ \frac{\sum_{(i,j)\in \mathcal{E}_{act}} l_{ij}^{(n)} - B}{B} \right]_+, \ \max_{i \in \mathcal{V}} \left[ \frac{\sum_{j \in \mathcal{V}_{out}(i)} p_{ij}^{(n)} - P_i^{\max}}{P_i^{\max}} \right]_+ \right), \label{eq:delta_prim} \\
		\Delta_{\text{dual}} &= \max_{\mathbf{z} = \{\omega, \boldsymbol{\varphi}, \boldsymbol{\theta}\}} \left( \frac{\|\mathbf{z}^{(n+1)} - \mathbf{z}^{(n)}\|_\infty}{\|\mathbf{z}^{(n)}\|_\infty + \epsilon_{\text{num}}} \right),\  \varepsilon_{\text{num}}>0. \label{eq:delta_dual}
	\end{align}
\end{subequations}
The iterative procedure terminates when $\max(\Delta_{\text{prim}}, \Delta_{\text{dual}}) \le \varepsilon$. The detailed execution of PDA is presented in Algorithm \ref{algo:dual_ascent}.

\begin{algorithm}[t]
	\caption{PDA algorithm for resource allocation} 
	\label{algo:dual_ascent}
	\textbf{Input:} Initial $(\mathbf{l}^{(0)}, \mathbf{t}^{(0)})$, active link set $\mathcal{E}_{act}$,   traffic loads $m_{ij}$, routing fractions $\rho_{ij}^k$, parameters $B, P_i^{\max}, N_0, h_{ij}$, $l_{\min}, t_{\min}, T_{\max}^{\text{QoS}} $, $\varepsilon, \varepsilon_{\text{inner}}$,  $ Q_{\max}, N_{\max} $. \\
	\textbf{Initialize:} Initialize  $\omega^{(0)} = 0$, $\boldsymbol{\varphi}^{(0)} = \mathbf{0}$, and $\boldsymbol{\theta}^{(0)} = {\alpha}\mathbf{1}/{|\mathcal{K}|}$. Set  $n \gets 0$,  $\delta^{(0)} \gets \delta_0$, and  $\Delta_{\text{best}} \gets \infty$.
	\begin{algorithmic}[1]
		\While{$\max(\Delta_{\text{dual}}, \Delta_{\text{prim}}) > \varepsilon$ \textbf{and} $n \le N_{\max}$} \label{linksub_start}
		\State \textbf{Parallel For}  $(i,j) \in \mathcal{E}_{act}$:
		\State \quad Initialize  $q \gets 0$.
		\State \quad Warm start: $l_{ij}^{(0)} \gets l_{ij}^{(n)}$ and $t_{ij}^{(0)} \gets t_{ij}^{(n)}$.
		\State \quad \textbf{Repeat}
		\State \quad\quad $q \gets q + 1$
		\State \quad\quad Update $l_{ij}^{(q)}$ by solving $\frac{\partial \Phi_{ij}}{\partial l_{ij}} = 0$ given $t_{ij}^{(q-1)}$ via 1D bisection on $ [l_{\min}, B] $.  
		\State \quad\quad Update $t_{ij}^{(q)}$ by solving $\frac{\partial \Phi_{ij}}{\partial t_{ij}} = 0$ given $l_{ij}^{(q)}$ via 1D bisection on $ [t_{\min}, T_{\max}^{\text{QoS}}] $. 
		\State \quad \textbf{Until} $\max\big(|l_{ij}^{(q)} - l_{ij}^{(q-1)}|, |t_{ij}^{(q)} - t_{ij}^{(q-1)}|\big) \le \varepsilon_{\text{inner}}$ \textbf{or} $q \ge Q_{\max}$. \label{linksub_end}
		\State \quad Assign updated block solutions: $l_{ij}^{(n+1)} \gets l_{ij}^{(q)}$ and $t_{ij}^{(n+1)} \gets t_{ij}^{(q)}$.
		\State \textbf{End Parallel For}
		\State Update  $\omega^{(n+1)}, \boldsymbol{\varphi}^{(n+1)},\boldsymbol{\theta}^{(n+1)}$ via \eqref{eq:dual_update_resources}.
		\State Compute $\Delta_{\text{prim}}, \Delta_{\text{dual}}$ via \eqref{eq:stopping_criteria}.
		\State \textbf{If} $\Delta_{\text{prim}} \le \varepsilon$ \textbf{and} $\Delta_{\text{dual}} < \Delta_{\text{best}}$ \textbf{then}  \label{pda_poc_start}
		\State \quad Update  $(\mathbf{l}^*, \mathbf{t}^*) \gets (\mathbf{l}^{(n+1)}, \mathbf{t}^{(n+1)})$ and $\Delta_{\text{best}} \gets \Delta_{\text{dual}}$.
		\State \textbf{End If}  \label{pda_poc_end}
		\State Compute  $T^{(n+1)} = \max_{k \in \mathcal{K}} \sum_{(i,j)\in \mathcal{E}_{act}} \rho_{ij}^k t_{ij}^{(n+1)}$.
		\State Update  $\delta^{(n+1)} = \delta_0 / \sqrt{n+2}$.
		\State $n \gets n + 1$.
		\EndWhile
		\State Recover $ \mathbf{p}^* $ via \eqref{eq:p_convex_func}.
		\State \textbf{Output:} Optimal resources $(\mathbf{l}^*, \mathbf{p}^*)$.
	\end{algorithmic}
\end{algorithm}

\subsection{Overall BCD framework}
The joint optimization problem \eqref{d2d_form_app} is  non-convex due to the coupling between the routing variables $\mathbf{x}$ and the physical resources $(\mathbf{l}, \mathbf{p})$. To solve  this problem while exploiting the bi-convex structure of the decomposed two subproblems, we propose a BCD framework. This framework iteratively alternates between the network-layer routing and the physical-layer resource allocation.

The strictly feasible initial point $(\mathbf{l}^{(0)}, \mathbf{p}^{(0)})$ is  generated via one iteration of  PDA algorithm upon the $ K $-shortest paths \cite{yen1971finding}. Let $\tilde{n}$ denote the outer BCD iteration index.  Each subsequent iteration $n \ge 0$ proceeds through the following two phases:

\textbf{Phase 1: path routing optimization.}
Given the fixed resources $(\mathbf{l}^{(\tilde{n})}, \mathbf{p}^{(\tilde{n})})$, the transmission rate $r_{ij}^{(\tilde{n})}$ and energy costs $c_{ij}^{\text{E},(\tilde{n})}$ degenerate into constants. The joint problem  collapses into a convex routing subproblem. We acquire the optimal flow $\mathbf{x}^{(\tilde{n}+1)}$ by solving the routing subproblem \eqref{eq:smoothed_routing}.
The proposed framework adaptively deploys either the MF-FW (Algorithm \ref{algo:fw}) for fast execution, or the LR-PDIPM (Algorithm \ref{algo:pd_ipm}) for high-precision solutions. Regardless of the deployed algorithm, the exact global optimum of this convex routing subproblem is guaranteed, thereby preserving the monotonic descent property of the overall BCD framework.

\textbf{Phase 2: resource allocation optimization.}
With the updated routing strategy $\mathbf{x}^{(\tilde{n}+1)}$ fixed, the link traffic loads $m_{ij}^{(\tilde{n}+1)} = \sum_{k \in \mathcal{K}} x_{ij}^{k,(\tilde{n}+1)} M^k$ transition from decision variables to fixed parameters. The optimization reduces to the convex time-domain formulation \eqref{eq:convex_global}. By executing the PDA  (Algorithm \ref{algo:dual_ascent}), we compute the optimal bandwidth and transmission time $(\mathbf{l}^{(\tilde{n}+1)}, \mathbf{t}^{(\tilde{n}+1)})$ under current $\mathbf{x}^{(\tilde{n}+1)}$. Subsequently, the physical resource vector is recovered to $ (\mathbf{l}^{(\tilde{n}+1)}, \mathbf{p}^{(\tilde{n}+1)})$ via \eqref{eq:p_convex_func}.

\textbf{Termination.}
To prevent  premature stopping of the heuristic, the BCD procedure evaluates the relative successive differences of both variable blocks. The algorithm terminates  when:
\begin{equation} \label{eq:bcd_termination}
	\Delta_{\text{out}} :=\max_{\mathbf{z} = \{\mathbf{}, \mathbf{l}, \mathbf{p}\}} \left( \frac{\|\mathbf{z}^{(\tilde{n}+1)} - \mathbf{z}^{(\tilde{n})}\|_\infty}{\|\mathbf{z}^{(\tilde{n})}\|_\infty + \varepsilon_{\text{num}}} \right) \le \varepsilon_{\text{out}},
\end{equation}
where $\varepsilon_{\text{out}}$ is the outer tolerance, and $\varepsilon_{\text{num}} > 0$. The execution of the BCD framework is summarized in Algorithm \ref{algo:bcd_master}.

\begin{algorithm}[htbp]
	\caption{BCD framework for joint optimization \eqref{d2d_form_app}} \label{algo:bcd_master}
	\textbf{Input:} Network graph $\mathcal{G}(\mathcal{V}, \mathcal{E})$, demands $\{M^k\}$, physical parameters $(\mathbf{h}, N_0, B, \mathbf{P}^{\max})$, objective weight $\alpha$, smoothing $\mu$, tolerances $\varepsilon_{\text{out}}$, max outer iterations $N_{\text{out}}$. \\
	\textbf{Initialize:} Initial solution $ (\mathbf{x}^{(0)},\mathbf{l}^{(0)}, \mathbf{p}^{(0)}) $. Set outer iteration $\tilde{n} \leftarrow 0$, global error $\Delta_{\text{out}} \leftarrow \infty$, pocket parameter $\mathcal{J}^* \leftarrow \infty$. 
	\begin{algorithmic}[1]
		\While{$\Delta_{\text{out}} > \varepsilon_{\text{out}}$ \textbf{and} $\tilde{n} < N_{\text{out}}$}
		\State \quad Given $(\mathbf{l}^{(\tilde{n})}, \mathbf{p}^{(\tilde{n})})$, compute transmission rate $r_{ij}^{(\tilde{n})}$ via \eqref{eq:shannon_rate} and energy costs $c_{ij}^{\text{E},(\tilde{n})} = (1-\alpha) M^k p_{ij}^{(\tilde{n})}/r_{ij}^{(\tilde{n})}$.
		\State \quad Update routing variables $\mathbf{x}^{(\tilde{n}+1)}$ via Algorithm \ref{algo:fw} (MF-FW) or Algorithm \ref{algo:pd_ipm} (LR-PDIPM).
		\State \quad Compute link load $m_{ij}^{(\tilde{n}+1)} = \sum_{k \in \mathcal{K}} x_{ij}^{k,(\tilde{n}+1)} M^k$.
		\State \quad Update time-domain resources $(\mathbf{l}^{(\tilde{n}+1)}, \mathbf{t}^{(\tilde{n}+1)})$ via Algorithm \ref{algo:dual_ascent} (PDA).
		\State \quad Recover transmission power $ p_{ij}^{(\tilde{n}+1)} $ via \eqref{eq:p_convex_func}.
		
		\State \quad Compute the  objective $\mathcal{J}^{(\tilde{n}+1)}$ via \eqref{obj_app}.
		\State \quad \textbf{if} $\mathcal{J}^{(\tilde{n}+1)} < \mathcal{J}^*$ \textbf{then} $\mathcal{J}^* \leftarrow \mathcal{J}^{(\tilde{n}+1)}$ and $(\mathbf{x}^*, \mathbf{l}^*, \mathbf{p}^*) \leftarrow (\mathbf{x}^{(\tilde{n}+1)}, \mathbf{l}^{(\tilde{n}+1)}, \mathbf{p}^{(\tilde{n}+1)})$. \label{bcd_pocket}
		
		\State \quad Update $\Delta_{\text{out}}$ via \eqref{eq:bcd_termination}.		
		\State \textbf{Iteration Update:} $\tilde{n} \leftarrow \tilde{n} + 1$.
		\EndWhile
		\State \textbf{Output:} Globally stationary routing and resource allocation $(\mathbf{x}^*, \mathbf{l}^*, \mathbf{p}^*)$. 
	\end{algorithmic}
\end{algorithm}

\section{Theoretical analysis and convergence guarantees}
\label{sec:convergence}

In this section, we provide a rigorous theoretical analysis of the proposed BCD framework. We establish the convergence analysis for the customized subproblem algorithms. Specifically, MF-FW, LRPDIPM  and PDA algorithms. Subsequently, we synthesize these components to prove the convergence of BCD algorithm. Finally, we conclude with an algorithmic complexity analysis.

\subsection{Convergence of Network-Layer Routing Algorithms}

The network-layer optimization reduces to a convex routing subproblem  \eqref{eq:smoothed_routing} featuring the smoothed objective $F(\mathbf{x})$. Define $T_{\max}(\mathbf{x}) = \max_{k \in \mathcal{K}} T_k(\mathbf{x})$ as the maximum aggregate transmission time across all commodities.
\begin{proposition}[LSE approximation bound] \label{prop:smoothing_gap}
	For any scaling parameter $\mu > 0$, the LSE-smoothed bottleneck function $\tilde{T}_{\max}(\mathbf{x}, \mu) = \frac{1}{\mu} \ln \left( \sum_{k \in \mathcal{K}} \exp^{\mu T_k(\mathbf{x})} \right)$ uniformly bounds the true maximum delay for any feasible routing flow $\mathbf{x}$:
	\begin{equation}
		T_{\max}(\mathbf{x}) \le \tilde{T}_{\max}(\mathbf{x}, \mu) \le T_{\max}(\mathbf{x}) + \frac{\ln |\mathcal{K}|}{\mu}.
	\end{equation}
\end{proposition}
\begin{proof}
	This bounding property is a fundamental result in convex analysis (see, e.g., Boyd and Vandenberghe \cite[Sec. 3.1.5]{boyd2004convex} and Nesterov's smoothing techniques \cite{nesterov2005smooth}). Thus, the detailed  derivation is omitted here.
\end{proof} 

\begin{remark}[Theoretical trade-off on smoothing parameter $\mu$]
	Proposition \ref{prop:smoothing_gap} reveals a trade-off governed by the smoothing parameter $\mu$. As $\mu \to \infty$, the approximation gap ${\ln |\mathcal{K}|}/{\mu}$ vanishes, recovering the exact non-smooth bottleneck objective. However, as established in Lemma \ref{lemma:lipschitz_hessian}, the Lipschitz constant of the Hessian scales quadratically with the smoothing parameter, i.e., $L_H = \mathcal{O}(\mu^2)$. Consequently, an excessively large $\mu$ quadratically amplifies the local curvature variations, rendering the Newton system severely ill-conditioned. This forces the LR-PDIPM to take restricted step sizes ($\lambda \ll 1$) to maintain interiority, thereby  degrading the overall convergence rate. In practice, selecting a moderately large constant for $\mu$ strikes a necessary balance, guaranteeing numerical stability in the Newton directions while maintaining a sufficiently tight approximation of the true maximum delay.
\end{remark}
\subsubsection{Convergence of the MF-FW algorithm}
To establish the convergence rate of the MF-FW (Algorithm \ref{algo:fw}), we first prove that the LSE-smoothed objective exhibits Lipschitz continuous gradients (i.e., $L$-smoothness) with respect to the flow routing variables. Let $\mathcal{X}$ denote the feasible routing domain defined by the flow conservation and non-negativity constraints, i.e., $\mathcal{X} = \left\{ \mathbf{x} \mid A \mathbf{x}^k = \mathbf{b}^k, \mathbf{x}^k \ge \mathbf{0}, \forall k \in \mathcal{K} \right\}$.

\begin{lemma}[Convexity and $L$-Smoothness] \label{lem:L_smooth}  
	The smoothed objective function $F(\mathbf{x})$ in \eqref{eq:smoothed_routing} is convex and $L$-smooth over  $\mathcal{X}$. Since the individual commodity delay $T_k(\mathbf{x})$ is linear with respect to $\mathbf{x}$, the Lipschitz constant $L$ is bounded by the maximal gradient norm of the delay functions:
	\begin{equation}
		L \le \alpha \mu \max_{k \in \mathcal{K}} \|\nabla T_k(\mathbf{x})\|_2^2 \le \alpha \mu C_{\max}^2,
	\end{equation}
	where $C_{\max} = \max_{k \in \mathcal{K}} \sum_{(i,j)\in \mathcal{E}} {M^k}/{r_{ij}}$ denotes the maximum cumulative path delay coefficient over the network.
\end{lemma}

\begin{proof}
	The detailed derivation is deferred to \ref{app:proof_lemma_L_smooth}.
\end{proof}

Leveraging Lemma \ref{lem:L_smooth} and the standard Descent Lemma, we establish the sublinear convergence rate of the MF-FW algorithm. Let $n$ denote the iteration index, and let $D = \max_{\mathbf{x}, \mathbf{y} \in \mathcal{X}} \|\mathbf{x} - \mathbf{y}\|_2$ denote the compact diameter of the flow conservation polytope.

\begin{theorem}[$\mathcal{O}(1/n)$ convergence rate of MF-FW] \label{thm:fw_rate}
	The primal optimality gap of the MF-FW algorithm at any iteration $n$ satisfies:
	\begin{equation}
		F(\mathbf{x}^{(n)}) - F(\mathbf{x}^*) \le \frac{2L D^2}{n+1}, \quad \forall n \ge 0,
	\end{equation}
	where $\mathbf{x}^* = \arg\min_{\mathbf{x}\in\mathcal{X}} F(\mathbf{x})$ is the exact global optimum of the smoothed subproblem \eqref{eq:smoothed_routing}, and $L$ is the Lipschitz constant established in Lemma \ref{lem:L_smooth}.
\end{theorem}

\begin{proof}[Proof sketch]
	By invoking the standard descent lemma for $L$-smooth functions, the objective reduction achieved by the truncated Newton step size $\lambda^* = \min\{1, \text{Gap}^{(n)}/C^{(n)}\}$ is guaranteed to be no less than the reduction obtained via the classic diminishing step size $\gamma^{(n)} = {2}/({n+2})$. By applying the convexity of $  F(\mathbf{x}) $, we establish a recursive contraction inequality for the primal optimality gap $h_n = F(\mathbf{x}^{(n)}) - F(\mathbf{x}^*)$. The global $\mathcal{O}(1/n)$ convergence rate is then obtained by resolving this recurrence via mathematical induction. Detailed derivations are deferred to  \ref{app:proof_fw_rate}.
\end{proof}

\begin{remark}[Polytope diameter bound] \label{rem:polytope_diameter}
	While the diameter $D$ serves as a theoretical worst-case metric, it admits an  upper bound governed by the network topology. Since the decision variables $\mathbf{x}$ represent routing proportions ($x_{ij}^k \in [0, 1]$), any extreme point of the feasible polytope $\mathcal{X}$ corresponds to a set of single paths for the $\mathcal{K}$ commodities. Given that a loop-free path traverses at most $|\mathcal{V}|-1$ edges, the squared $L_2$-norm of any path vector is strictly bounded by $|\mathcal{V}|-1$. Consequently, the maximum Euclidean distance between any two feasible flow assignments is  bounded by $D \le \sqrt{2 |\mathcal{K}| (|\mathcal{V}| - 1)}$. This demonstrates that the optimization error bound in Theorem \ref{thm:fw_rate} grows sublinearly with the network size, guaranteeing stable convergence even for dense D2D topologies.
\end{remark}

Ultimately, we bridge the gap between the solution obtained by MF-FW and the original non-smooth min-max formulation to quantify the total approximation error. Let $\mathcal{J}$ denote  the non-smooth objective \eqref{obj_app} and $\mathcal{J}^*$ be the minimum value of $\mathcal{J}$. 

\begin{theorem}[Total approximation error bound] \label{thm:total_error}
	The routing solution $\mathbf{x}^{(n)}$ obtained at iteration $n$ satisfies:
	\begin{equation}
		\mathcal{J}(\mathbf{x}^{(n)}) - \mathcal{J}^* \le \underbrace{\frac{2 \alpha \mu C_{\max}^2 D^2}{n+1}}_{\text{Optimization Error}} + \underbrace{\frac{\alpha}{\mu} \ln |\mathcal{K}|}_{\text{Smoothing Error}}.
	\end{equation}
\end{theorem}
\begin{proof}
	Let $\mathbf{x}^*_F $ denote the  optimum of formulation \eqref{eq:smoothed_routing} with the smoothed objective $ F(\mathbf{x}) $. By the approximation bound established in Proposition \ref{prop:smoothing_gap}, we have $\mathcal{J}(\mathbf{x}^{(n)}) \le F(\mathbf{x}^{(n)})$ and $\mathcal{J}^* \ge F(\mathbf{x}^*_F) - {\alpha\ln |\mathcal{K}|}/{\mu} $. Subtracting the latter from the former yields:
	\begin{equation*}
		\mathcal{J}(\mathbf{x}^{(n)}) - \mathcal{J}^* \le \big( F(\mathbf{x}^{(n)}) - F(\mathbf{x}^*_F) \big) + \frac{\alpha}{\mu} \ln |\mathcal{K}|.
	\end{equation*}
	Substituting the $\mathcal{O}(1/n)$ primal gap bound of $F(\mathbf{x}^{(n)}) - F(\mathbf{x}^*_F)$ established in Theorem \ref{thm:fw_rate} directly completes the proof.
\end{proof}
This theorem characterizes the trade-off governed by the smoothing parameter $\mu$: a larger $\mu$ suppresses the smoothing error but  inflates the Lipschitz constant $L$, decelerating the MF-FW convergence.

%

\begin{remark}[Per-iteration computational complexity of MF-FW] \label{rem:complexity}
	In each inner MF-FW iteration, the gradient evaluation is computed in  $\mathcal{O}(|\mathcal{K}| |\mathcal{E}|)$ operations. The dominant computational overhead in each inner FW iteration is computing the descent direction $\mathbf{d}^{(n)}$, which requires executing Dijkstra's algorithm for $|\mathcal{K}|$ commodities. Thus, assuming a Fibonacci heap implementation, the per-iteration complexity is bounded by $\mathcal{O}(|\mathcal{K}| (|\mathcal{E}| + |\mathcal{V}| \log |\mathcal{V}|))$, making the MF-FW algorithm exceptionally scalable for dense D2D networks.
\end{remark}

\subsubsection{Convergence of the LR-PDIPM algorithm}
Unlike the first-order methods which typically exhibit sublinear convergence, the Newton-based LR-PDIPM achieves rapid, high-precision convergence for dense networks. Establishing the  polynomial-time complexity of this algorithm requires bounding the nonlinearity introduced by the LSE smoothing and proving that the Newton steps  decrease the duality gap.

To guarantee the convergence and step-size stability of LR-PDIPM, we must establish the second-order smoothness of the smoothed objective $ F(\mathbf{x}) $. Let $R_{\max} = \max_{k, (i,j)} \left( M^k/r_{ij} \right)$ denote the maximum path delay coefficient across the network.

\begin{lemma}[Lipschitz continuous hessian] \label{lemma:lipschitz_hessian}
	The Hessian matrix $\nabla^2 F(\mathbf{x})$ is $L_H$-Lipschitz continuous over the feasible domain $\mathcal{X}$. Specifically, there exists a constant $L_H > 0$ such that:
	\begin{equation}
		\|\nabla^2 F(\mathbf{x}) - \nabla^2 F(\mathbf{y})\|_2 \le L_H \|\mathbf{x} - \mathbf{y}\|_2, \quad \forall \mathbf{x}, \mathbf{y} \in \mathcal{X},
	\end{equation}
	where the Lipschitz constant scales quadratically with the smoothing parameter, bounded by $L_H = \mathcal{O}(\alpha \mu^2 |\mathcal{E}|^{3/2} R_{\max}^3)$.
\end{lemma}

\begin{proof}[Proof sketch]
	Establishing the Lipschitz continuity of the Hessian requires bounding the spectral norm of the third-order directional derivative of the LSE function. The detailed derivation is provided in \ref{app:proof_lemma_hessian}.
\end{proof}

To preclude algorithmic stagnation, the backtracking line search yields a step size that is lower bounded by a positive constant. For any iterate within the wide neighborhood $\mathcal{N}_{-\infty}(\gamma)$, the second-order cross terms of the Newton direction are quadratically bounded: $\|\Delta \mathbf{X} \Delta \mathbf{s}\|_2^2 \le \mathcal{C} (\mu_{\text{gap}})^2$, where $\Delta \mathbf{X} = \text{diag}(\Delta \mathbf{x})$. Exploiting this bound alongside the Lipschitz continuous Hessian (Lemma \ref{lemma:lipschitz_hessian}), we establish the following invariance and convergence guarantee.

\begin{lemma}[Wide-neighborhood invariance] \label{lemma:step_size}
	The backtracking line search (Algorithm \ref{algo:pd_ipm}, Steps \ref{linesea_start}-\ref{linesea_end}) is guaranteed to terminate with a step size $\lambda \ge \lambda_{\min} > 0$, ensuring a monotonic gap reduction:
	\begin{equation}
		\mu_{\text{gap}}^{(n+1)} \le \big( 1 - \lambda_{\min}(1-\sigma_{\max}) \big) \mu_{\text{gap}}^{(n)}.
	\end{equation}
\end{lemma}
\begin{proof}[Proof sketch]
	The proof hinges on bounding the nonlinear Taylor residual of the Newton step. First, standard IPMs guarantee that the second-order cross terms are quadratically bounded by the duality gap. Second, because the Hessian is $L_H$-Lipschitz continuous (Lemma \ref{lemma:lipschitz_hessian}), the deviation from the linear KKT prediction grows quadratically with the step size (i.e., $\mathcal{O}(\lambda^2)$). Consequently, for sufficiently small $\lambda$, the linear descent of the duality gap (i.e., $\mathcal{O}(\lambda)$) dominates the quadratic penalty, ensuring that the backtracking line search terminates at a positive step size $\lambda_{\min} > 0$ before violating the wide neighborhood conditions. The detailed derivation is deferred to \ref{app:proof_step_size}. 
\end{proof}

Building upon Lemmas \ref{lemma:lipschitz_hessian} and \ref{lemma:step_size}, we summarize the convergence and computational complexity of the proposed LR-PDIPM.

\begin{theorem}[Computational complexity of LR-PDIPM] \label{thm:ipm_complexity}
	Under the strictly interior initialization $(\mathbf{x}^{(0)}, \mathbf{s}^{(0)}) > \mathbf{0}$, the LR-PDIPM converges to an optimal solution of the convex smoothed routing subproblem \eqref{eq:smoothed_routing}. To achieve an $\varepsilon$-optimal solution, the algorithm requires at most $\mathcal{O}\left(\sqrt{|\mathcal{K}||\mathcal{E}|} \ln(1/\varepsilon)\right)$ iterations. 
\end{theorem}
\begin{proof}[Proof sketch]
	The core logic relies on the interaction between the sequence feasibility and the monotonic gap reduction. First, Lemma \ref{lemma:step_size} ensures that the line search generates a sequence bounded within the wide neighborhood, yielding a reduction rate of $\mu^{(n+1)} \le (1-\delta)\mu^{(n)}$. Second, guided by standard interior-point path-following theories, the single-step reduction factor $\delta$ is inversely proportional to the square root of the inequality constraint dimension $ |\mathcal{K}||\mathcal{E}|$. Unrolling this reduction recursion to satisfy the $\varepsilon$-KKT tolerance yields the logarithmic iteration bound $\mathcal{O}(\sqrt{|\mathcal{K}||\mathcal{E}|} \ln(1/\varepsilon))$. The detailed derivation is deferred to \ref{app:proof_thm_complexity}. 
\end{proof}
  
\begin{remark}[Computational complexity of LR-PDIPM] \label{com_ipm}
	Since the nested SM rank-1 formula  preserves the exact algebraic structure of the Schur complement without introducing numerical approximation errors, the total arithmetic complexity is bounded by $\mathcal{O}\left(|\mathcal{K}||\mathcal{V}|^3 \sqrt{|\mathcal{K}||\mathcal{E}|} \ln(1/\varepsilon)\right)$, overcoming the prohibitive $\mathcal{O}(|\mathcal{K}|^3|\mathcal{V}|^3)$ per-iteration cost of conventional dense Hessian inversion.
\end{remark}


\subsection{Convergence of physical-layer resource allocation}
To establish a rigorous theoretical foundation, we analyze three different aspects of the proposed PDA algorithm: subproblem optimality, dual convergence, and computational complexity.

\subsubsection{Convergence of link-level subproblems}
The primal update involves a decoupled 2D subproblem $ \Phi_{ij}(l_{ij}, t_{ij}) $ (defined in \eqref{eq:lagrangian_func}) for each active link.
\begin{lemma}[Optimality of link-level subproblem] \label{lemma:cd_conv}
	The alternating coordinate descent procedure (Algorithm \ref{algo:dual_ascent}, Steps \ref{linksub_start}-\ref{linksub_end}) converges to the unique global minimizer $(l_{ij}^*, t_{ij}^*)$ of the Lagrangian subproblem $\Phi_{ij}$.
\end{lemma}

\begin{proof}[Proof sketch]
	According to \eqref{eq:lagrangian_func}, the subproblem objective $\Phi_{ij}$ is a positive linear combination of the energy consumption ${E}_{ij}$, transmission power $p_{ij}$, and linear resource/time terms. By Lemma \ref{lemma:core_func} and the proof of Theorem \ref{thm:convexity}, ${E}_{ij}$ is strictly convex and $ p_{ij} $ is convex. Given weighting factor $\alpha \in (0,1)$, the strict convexity is preserved in $\Phi_{ij}$. Over the compact Cartesian product domain $\mathcal{F}_{ij}$, the alternating minimization generates a descending sequence. The decoupled constraint structure ensures that every limit point is a stationary point, which is the unique global optimum. The detailed derivation is provided in  \ref{app:proof_cd_conv}. 
\end{proof}

\subsubsection{Convergence of PDA algorithm}
The convergence of the PDA algorithm hinges on the strict convexity of the Lagrangian function \eqref{eq:lagrangian_func} and the boundedness of the resulting dual gradients. 
At the $n$-th iteration, let $\mathbf{g}^{(n)} \triangleq [g_\omega^{(n)}, (\mathbf{g}_{\varphi}^{(n)})^{\mathsf{T}}, (\mathbf{g}_{\theta}^{(n)})^{\mathsf{T}}]^{\mathsf{T}}$ denote the gradient vector of \eqref{eq:lagrangian_func}. As formulated in the dual update rules \eqref{eq:dual_update_resources}, this vector evaluates the normalized residuals (i.e., relative violations) of the capacity and delay constraints at the current unique primal minimizer $(\mathbf{l}^{(n)}, \mathbf{t}^{(n)})$. Note that the auxiliary delay variable $T$ has been eliminated via the dual simplex constraint $\sum_{k \in \mathcal{K}} \theta^k = \alpha$.

\begin{lemma}[Bounded dual gradients] \label{lemma:bounded_g}
	The dual gradients are uniformly bounded over all iterations. That is, there exists a finite constant $G > 0$ such that $\|\mathbf{g}^{(n)}\|_2 \le G, \, \forall n \ge 0$.
\end{lemma}

\begin{proof}[Proof sketch]
	The gradient components correspond to the normalized constraint residuals of the bandwidth, power, and delay. Because the inner iterations confine the resource variables $(l_{ij}^{(n)}, t_{ij}^{(n)})$ within a compact Cartesian product domain $\mathcal{F}_{ij} = [l_{\min}, B] \times [t_{\min}, T_{\max}^{\text{QoS}}]$, both the linear combinations and the continuous nonlinear power functions attain finite values. Dividing these finite values by their respective constraint limits ensures the relative violations remain finite. Consequently, the gradient vector is bounded. The detailed proof is provided in \ref{app:proof_bounded_g}.
\end{proof}

To establish the overall convergence of the proposed PDA algorithm, we assume the original convex formulation \eqref{eq:convex_global} satisfies Slater's condition, meaning there exists at least one strictly interior feasible point. This standard assumption guarantees strong duality.

\begin{theorem}[Global optimality of PDA] \label{thm:pda_global}
	Suppose Slater's condition holds. The dual sequence $(\omega^{(n)}, \boldsymbol{\varphi}^{(n)}, \boldsymbol{\theta}^{(n)})$ generated by the PDA algorithm converges to the optimal dual set. Consequently, the primal sequence $(\mathbf{l}^{(n)}, \mathbf{t}^{(n)})$ converges to the optimum $(\mathbf{l}^*, \mathbf{t}^*)$ of the convex resource allocation subproblem \eqref{eq:convex_global}.
\end{theorem}

\begin{proof}[Proof sketch]
	In the dual domain, Lemma \ref{lemma:bounded_g} guarantees that the dual gradients are uniformly bounded. According to standard gradient optimization theory, applying a square summable but not summable step size to bounded gradients ensures the asymptotic  convergence of the dual objective to its optimum. In the primal domain, because the link-level Lagrangian subproblem $\Phi_{ij}$ is strictly convex (as established in Lemma \ref{lemma:cd_conv}), its primal minimizer is uniquely determined by the given dual variables. Therefore, the convergence of the dual sequence strictly forces the primal sequence to converge to the optimum without any duality gap.
\end{proof}

\begin{remark}[Pocket mechanism and dual stability]
	Although the dual function is differentiable, the dual objective may still exhibit non-monotonic behavior due to the overshooting characteristic of predetermined step sizes. The incorporated Pocket Mechanism (Steps \ref{pda_poc_start}-\ref{pda_poc_end}, Algorithm \ref{algo:dual_ascent}) tracks and preserves the feasible primal solution with the lowest objective value. This guarantees that the algorithm invariably terminates with the highest-quality feasible solution discovered across all iterations, neutralizing dual oscillations.
\end{remark}

\begin{remark}[Computational complexity of PDA]  \label{com_pda}
	The computational complexity of the PDA algorithm is evaluated based on its parallel execution time. In the primal phase, leveraging the decoupled structure, the 2D subproblems are solved  in parallel across all active links. Thus, the  time complexity for the primal update is  $\mathcal{O}(\log(1/\varepsilon_{\text{inner}}))$, where $\varepsilon_{\text{inner}}$ is the inner tolerance. In the dual phase, computing the gradient vector requires $\mathcal{O}(|\mathcal{E}_{\text{act}}|)$ operations. Subsequently, updating the scalar multipliers $\omega$ and $\boldsymbol{\varphi}$ requires $\mathcal{O}(1)$ and $\mathcal{O}(|\mathcal{V}|)$ operations respectively, while the simplex projection for $\boldsymbol{\theta}$ takes $\mathcal{O}(|\mathcal{K}| \log |\mathcal{K}|)$ operations. Given the network is connected ($|\mathcal{E}_{\text{act}}| \ge |\mathcal{V}| -1$), the per-iteration time complexity is bounded by $\mathcal{O}(\log(1/\varepsilon_{\text{inner}}) + |\mathcal{E}_{\text{act}}| + |\mathcal{K}| \log |\mathcal{K}|)$. Given the standard $\mathcal{O}(1/\varepsilon)$ iterations of the gradient method to achieve an $\varepsilon$-optimal dual solution, the total time complexity is bounded by $\mathcal{O}\left(\frac{1}{\varepsilon} \left[ \log(1/\varepsilon_{\text{inner}}) + |\mathcal{E}_{\text{act}}| + |\mathcal{K}| \log |\mathcal{K}| \right]\right)$. This linear-logarithmic scaling validates its practical efficiency for dense D2D deployments.
\end{remark}
\subsection{Convergence analysis of overall BCD framework}

Finally, we establish the theoretical convergence of the proposed inexact BCD framework. Let  $\tilde{\mathcal{J}}$ denote the LSE smoothed  objective of the non-smooth objective function $\mathcal{J}$ defined in \eqref{obj_app}. Because the inner block updates are iteratively solved via finite-step optimization rather than exact global oracles, the strict monotonic descent of objective $\mathcal{J}$ across outer BCD iterations cannot be guaranteed. Instead, the framework exhibits a bounded convergence behavior, necessitating the Pocket Mechanism to track the best feasible solution.
\begin{theorem}[Convergence of BCD framework] \label{thm:bcd_converge}
	The sequence generated by the inexact BCD framework converges to an $\varepsilon$-neighborhood of a stationary point of the problem under the smoothed objective $\tilde{\mathcal{J}}$. 
\end{theorem}


\begin{proof}[Proof sketch]
	The proof relies on the convergence properties of BCD framework under continuous differentiability and finite inner-loop truncation. First, the LSE approximation introduces a deterministic bounded gap $\sup |\tilde{\mathcal{J}} - \mathcal{J}| \le  \ln |\mathcal{K}|/{\mu}$ (Proposition \ref{prop:smoothing_gap}). Second, within each BCD outer iteration, the routing and resource blocks are solved to their respective $\varepsilon$-optimalities rather than exact minima.
	According to standard inexact BCD theory for non-convex smooth functions \cite{wright2015coordinate}, as long as the objective reduction achieved by each block update dominates the accumulated inner-loop truncation errors, the sequence guarantees monotonic descent. As the sequence approaches the stationary point and the gradients diminish, the truncation errors eventually lower-bound the descent, trapping the sequence within a bounded $\varepsilon$-neighborhood of the stationary point.
\end{proof}

\begin{remark}[Stabilization via Pocket Mechanism]
	While the sequence converges with respect to $\tilde{\mathcal{J}}$, the oscillations of the non-smooth objective $\mathcal{J}$ cannot be avoided due to the smoothing gap and finite inner-loop truncation. However, these are neutralized  by the Pocket Mechanism (Step \ref{bcd_pocket}, Algorithm \ref{algo:bcd_master}), which filters the non-monotonic solution ensuring that the output sequence of the objective is monotonically non-increasing.
\end{remark}


Building upon the established algorithmic analyses, the total computational complexity of the BCD framework (Algorithm \ref{algo:bcd_master}) can be obtained. 
The complexity analysis begins with a one-time initialization phase (i.e., applying the PDA over the $K$-shortest paths). Assuming a Fibonacci heap implementation, this procedure requires $\mathcal{C}_{\text{init}} = \mathcal{O}\left(|\mathcal{K}| N_{\text{path}} |\mathcal{V}| (|\mathcal{E}| + |\mathcal{V}| \log |\mathcal{V}|)\right)$ operations, where $N_{\text{path}}$ denotes the predefined maximum number of initial routing paths for each commodity.

Within the outer BCD loop, capped at a maximum of $N_{\text{out}}$ outer iterations, each iteration alternates between the routing optimization and the resource allocation. Let $\mathcal{C}_{\text{route}}$ and $\mathcal{C}_{\text{pda}}$ denote their respective operational costs. As derived in Remark \ref{rem:complexity} and Remark \ref{com_ipm}, the routing complexity yields $\mathcal{C}_{\text{route}} = \mathcal{O}\left(I_{\text{fw}}|\mathcal{K}| (|\mathcal{E}| + |\mathcal{V}| \log |\mathcal{V}|)\right) $ when utilizing the MF-FW algorithm (assuming $I_{\text{fw}}$ inner iterations), or $\mathcal{C}_{\text{route}} = \mathcal{O}\left(|\mathcal{K}||\mathcal{V}|^3 \sqrt{|\mathcal{K}||\mathcal{E}|} \ln(1/\varepsilon)\right)$ via the LR-PDIPM. Meanwhile, as established in Remark \ref{com_pda}, the operational cost for the  gradient-based resource allocation is bounded by $\mathcal{C}_{\text{pda}} =\mathcal{O}\left(\frac{1}{\varepsilon} \left[ \log(1/\varepsilon_{\text{inner}}) + |\mathcal{E}_{\text{act}}| + |\mathcal{K}| \log |\mathcal{K}| \right]\right) $.

Synthesizing these components, the overall worst-case time complexity of the proposed BCD framework is bounded by:
\begin{equation} \label{eq:total_complexity}
	\mathcal{O}\left( \mathcal{C}_{\text{init}} + N_{\text{out}} \left( \mathcal{C}_{\text{route}} + \mathcal{C}_{\text{pda}} \right) \right).
\end{equation}

\section{Numerical results and performance evaluation}
\label{sec:numerical_results}
In this section, we conduct comprehensive numerical experiments to validate the effectiveness of the proposed BCD framework. We first introduce the simulation parameters and performance metrics, followed by an evaluation of the Pareto trade-off to determine the optimal weighting coefficient $ \alpha $. Under this optimal $ \alpha $, we conduct a sequence of experiments, focusing on the resilience, scalability, and robustness. For these experiments, 50 test instances are randomly generated. The code and all topology datasets are publicly available in our repository.\footnote{https://github.com/Gthu/JRRA-for-D2D-Networks}

\paragraph{Network topology and traffic generation}  We consider a $R \times R $  m$^2 $ simulated area where $|\mathcal{V}| $ user nodes are randomly distributed within this area \cite{liu2023communication,gures2026joint}.  Two directional links are established between any node pair if their distance  satisfies $d_{i,j} \le D_{\max}$. To induce realistic multi-hop routing, the $|\mathcal{K}|$ commodities (source-destination pairs) are generated under a minimum distance constraint $d_{src,dst} \ge 0.6 R$. The traffic demands $ M^k $ are generated  following a heterogeneous bimodal distribution: $80\%$ of the streams are assigned light loads $M^k \in [0.1, 0.5]$ (e.g., periodic sensor telemetry), while the remaining $20\%$ are assigned heavy loads $M^k \in [1.0, 2.0]$ (e.g.,  video streams or control commands).
The key system and physical layer parameters are summarized in Table \ref{tab:sim_params}.

\begin{table}[htbp]
	\centering
	\small
	\caption{Simulation parameters}
	\label{tab:sim_params}
	\begin{tabular}{@{}ll@{}}
		\toprule
		\textbf{Parameter} & \textbf{Value} \\ 
		\midrule
		System radius ($ R $)  & $ 500 $ m \\
		Maximum D2D link distance ($D_{\max}$) & $200$ m \\
		Number of nodes ($|\mathcal{V}|$) & $60$ \\
		Average node degree & $ 5 \sim 7$ \\
		Number of D2D commodities ($|\mathcal{K}|$) & $20$ \\
		Commodity demand ($M^k$) & $ 0.1 \sim 2.0$ Mbits \\
		Total available bandwidth ($B$) & $100$ MHz \\
		Max transmission power per node ($P_i^{\max}$) & $23$ dBm \\
		Noise power spectral density ($N_0$) & $-174$ dBm/Hz \\
		Channel path loss model (3GPP) & $128.1 + 37.6 \log_{10}(\frac{d_{ij}}{1000})$ \\
		\bottomrule
	\end{tabular}
\end{table}

\begin{table*}[t] 
	\centering
	\small 
	\caption{Taxonomy of the evaluated algorithms}
	\label{tab:algorithm_baselines}
	\begin{tabular}{@{}l l l l l l@{}}
		\toprule
		\textbf{Algorithm} & \textbf{Path strategy} & \textbf{Routing solver} & \textbf{Resource allocation} & \textbf{Architecture} & \textbf{Initialization} \\ \midrule
		\textbf{BCD-FW} & Multi-path  & MF-FW & PDA & Joint optimization & KSP-PDA solution \\
		\textbf{BCD-FW-WARM} & Multi-path  & MF-FW & PDA & Joint optimization & GG solution \\
		\textbf{BCD-IPM} & Multi-path  & LR-PDIPM & PDA & Joint optimization & KSP-PDA solution \\ 
		\textbf{BCD-IPM-WARM}& Multi-path  & LR-PDIPM & PDA & Joint optimization & GG solution \\ 
		GG \cite{gures2026joint} & Single-path & Greedy Dijkstra & Potential Game & Decoupled optimization & N/A \\
		SP-SCA \cite{liu2022routing} & Single-path & Dijkstra & SCA, Projected gradient & Decoupled optimization & N/A \\
		SP-PDA & Single-path & Dijkstra & PDA & Decoupled optimization & N/A \\
		KSP-PDA & Multi-path  & $ K $-Shortest path & PDA & Decoupled optimization & N/A \\
		SP-SA & Single-path & Dijkstra & Equal partition & Decoupled optimization & N/A \\ \bottomrule
	\end{tabular}
\end{table*}
To evaluate the performance of our proposed BCD framework, we benchmark it against the state-of-the-art and conventional baselines. These evaluated algorithms are summarized in Table \ref{tab:algorithm_baselines}, categorized by their underlying path strategies, routing and resource allocation algorithms, optimization architectures, and initialization strategies. Notably, we introduce BCD-FW-WARM and BCD-IPM-WARM, which leverage the heuristic solution of GG as a high-quality initial point to accelerate the solution progress.

\paragraph{Performance metrics}
We record the normalized values of original objective \eqref{eq:original_obj} alongside its individual components: the  maximum transmission time and the total network energy consumption. The normalized value is computed as
\begin{equation} \label{normalized_obj}
	\alpha  \hat{T} /T_0
	 + (1-\alpha) \hat{E}/E_0.
\end{equation}
Here, $ \hat{T},\hat{E} $ denote the maximum delay and the total energy consumption defined in \eqref{eq:original_obj} respectively, and $ T_0,E_0 $ are the initial  values of these two objectives obtained via the baseline KSP-PDA.
Furthermore, we introduce the following multi-dimensional metrics to assess spatial utilization, energy efficiency, and commodity fairness:
\begin{itemize}
	\item {Active links \& multi-path count:} the total number of utilized links ($l_{ij} > 0$) and the number of commodities employing more than one routing path. 
	\item {Energy efficiency (EE):} defined as the ratio of total  delivered data  to the total energy consumed across the network, measured in Mbits/Joule:
	\begin{equation}
		\text{EE} = \frac{\sum_{k \in \mathcal{K}} M^k}{\hat{E}}.
	\end{equation}
	\item {Jain's fairness index (JFI):} evaluated based on the maximum transmission times $\hat{T}_k$ of all $\mathcal{K}$ commodities to quantify the severity of the bottleneck:
	\begin{equation}
		\text{JFI} = \frac{\left( \sum_{k \in \mathcal{K}} \hat{T}_k \right)^2}{|\mathcal{K}| \sum_{k \in \mathcal{K}} \hat{T}_k^2}.
	\end{equation}
\end{itemize}

\begin{figure}[htbp]
	\centering
	\includegraphics[width=0.85\columnwidth]{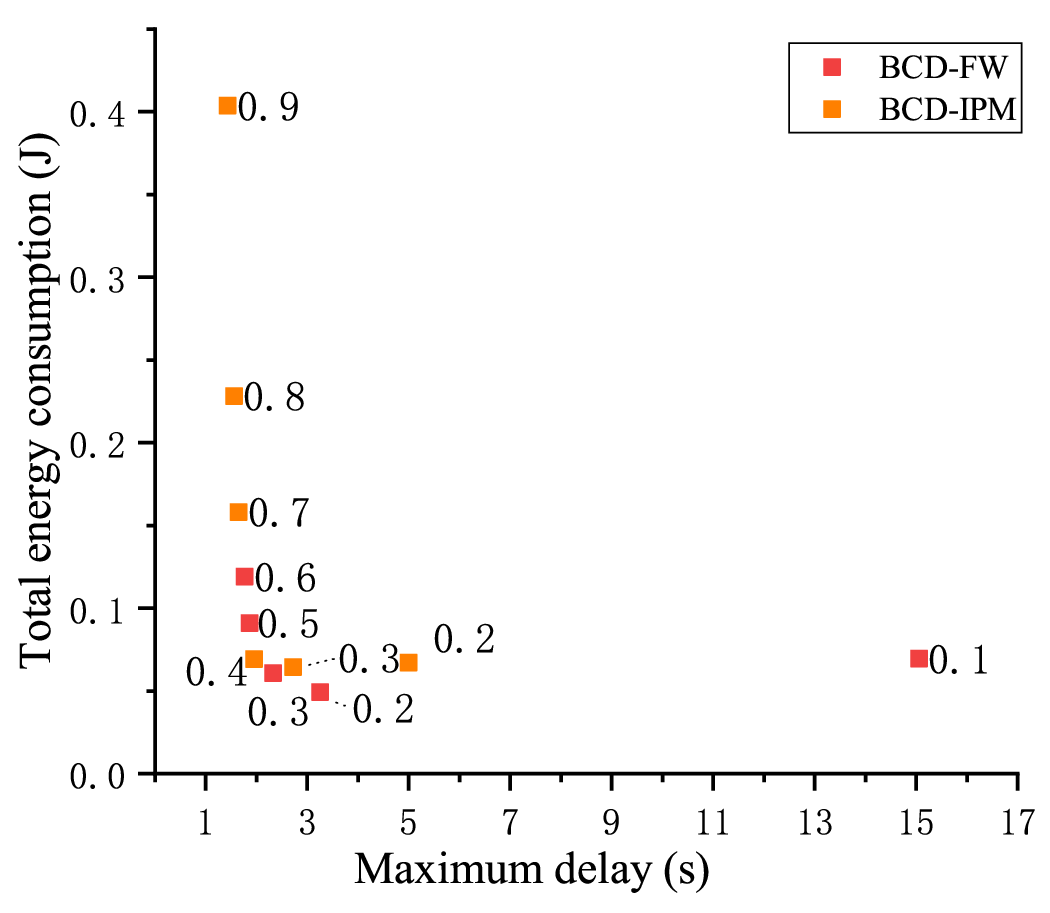}
	\caption{Pareto frontier illustrating the trade-off between maximum delay and total energy consumption.}
	\label{fig:pareto_frontier}
\end{figure}
\begin{figure*}[t]
	\centering
	
	\begin{subfigure}[b]{0.48\textwidth}
		\centering
		\includegraphics[width=\textwidth]{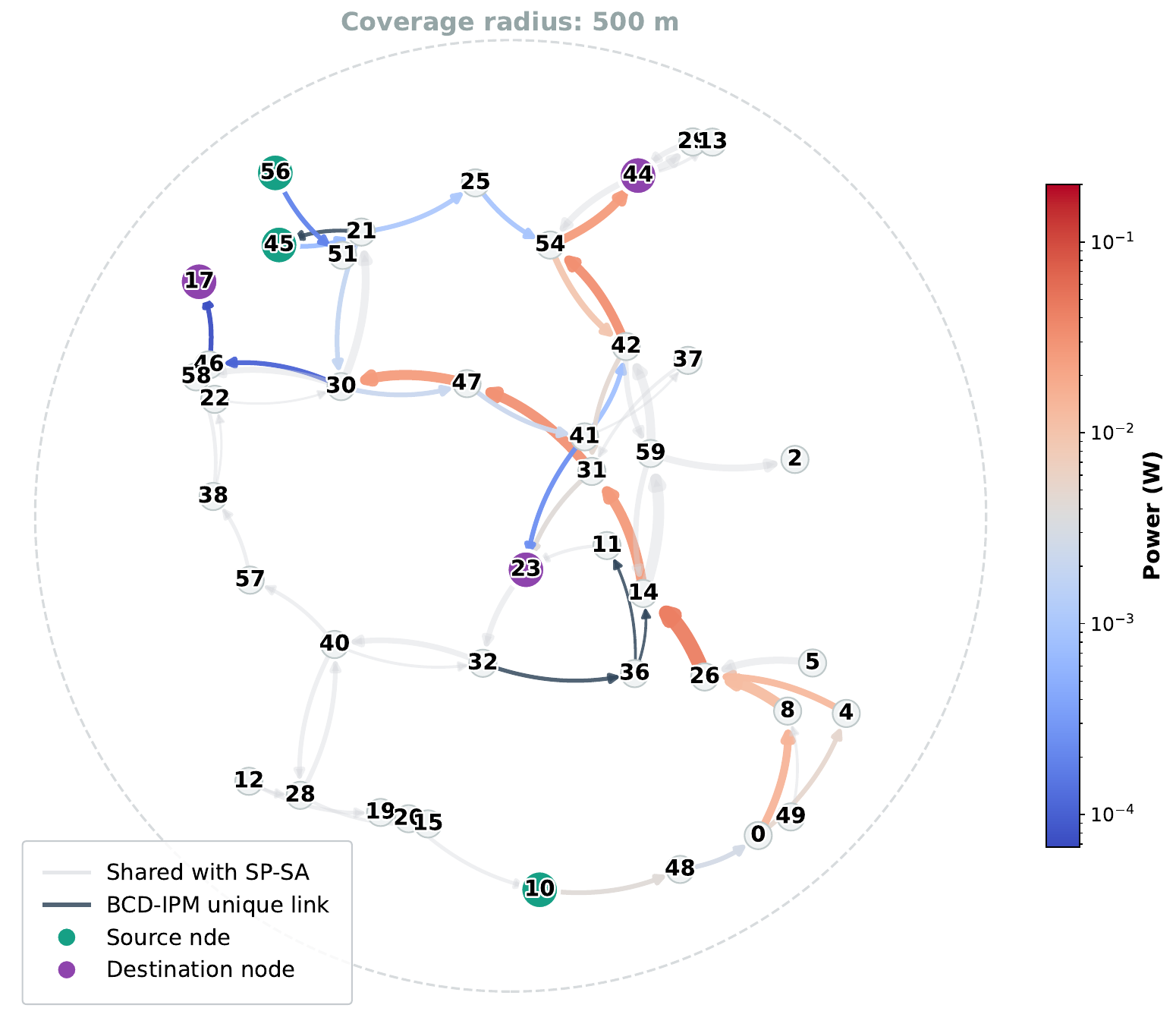}
		\caption{BCD-IPM: multi-path routing trajectories.}
		\label{fig:traj_ipm}
	\end{subfigure}
	\begin{subfigure}[b]{0.48\textwidth}
		\centering
		\includegraphics[width=\textwidth]{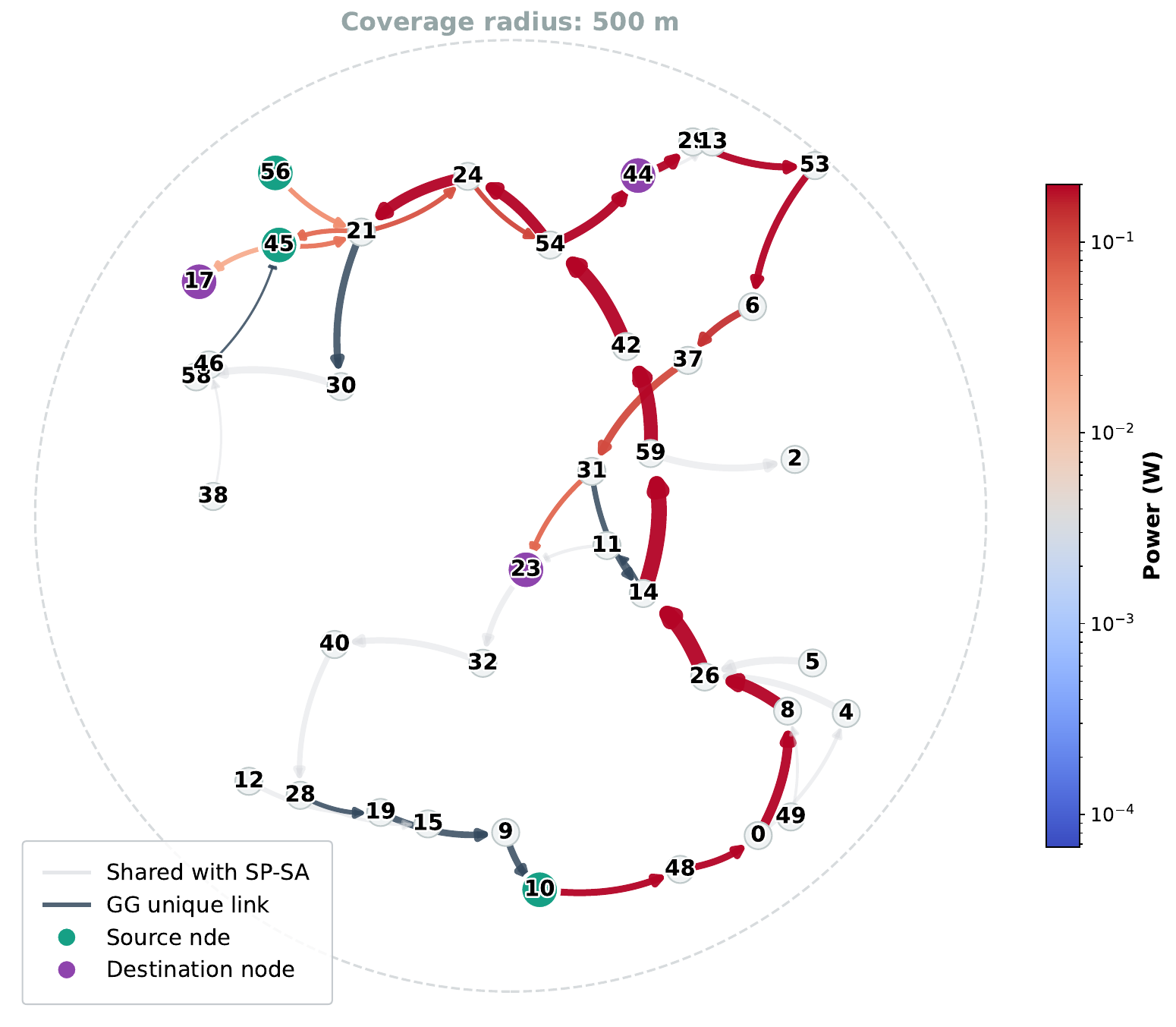}
		\caption{GG: single-path routing trajectories.}
		\label{fig:traj_gg}
	\end{subfigure}
	\caption{Microscopic visualization of routing trajectories and power allocation across Topology 10. The thickness of each link is proportional to its allocated flow volume. Furthermore, the specific links constituting the multi-path routing generated by the BCD-IPM are displayed with a thermal colormap (transitioning from deep blue to dark red), to reflect the transmission power usage.}
	\label{fig:topology10_case_study}
\end{figure*}
We first investigate the Pareto optimal frontier of the proposed BCD-FW and BCD-IPM under various $ \alpha $ values. As illustrated in Fig. \ref{fig:pareto_frontier}, the knee point  $ \alpha = 0.4  $ achieves the optimal trade-off between maximum delay and total energy consumption. Any further reduction in either delay or energy consumption beyond this point results in a severe surge in the other metric. Consequently, we set $\alpha = 0.4$ for all subsequent evaluations.

\subsection{Case study}
\label{subsec:case_study}
To demonstrate the optimization mechanism from a microscopic perspective, we extract a case study (Topology 10 in Section \ref{subsec:robustness_evaluation}). In this instance comprising $60$ nodes and $20$ commodities, baselines fall short in achieving a good trade-off between the two objectives, whereas the BCD-IPM discovers a superior routing and resource allocation solution. To expose the underlying mechanisms driving this performance gap, Fig. \ref{fig:topology10_case_study} visualizes the precise routing trajectories generated by the BCD-IPM and the GG heuristic. For clarity, the visualization isolates only the active nodes and links participating in data transmission.

\begin{table*}[t]
	\centering
	\small
	\caption{Microscopic resource allocation profile on the congested backbone trajectory (Commodity 17: Node 10 to 17)}
	\label{tab:microscopic_bottleneck}
	\renewcommand{\arraystretch}{1.2} 
	\begin{tabular}{l *{8}{>{\centering\arraybackslash}p{1.4cm}}}
		\toprule
		\multirow{3}{*}{\textbf{Link}} & \multicolumn{6}{c}{\textbf{Aggregated metrics (all commodities)}} & \multicolumn{2}{c}{\textbf{For Commodity 17}} \\
		\cmidrule(lr){2-7} \cmidrule(lr){8-9}
		& \multicolumn{2}{c}{\textbf{Routed commodities}} & \multicolumn{2}{c}{\textbf{Bandwidth (MHz)}} & \multicolumn{2}{c}{\textbf{Power (W)}} & \multicolumn{2}{c}{\textbf{Delay (s)}} \\
		\cmidrule(lr){2-3} \cmidrule(lr){4-5} \cmidrule(lr){6-7} \cmidrule(lr){8-9}
		& \textbf{GG} & \textbf{BCD-IPM} & \textbf{GG} & \textbf{BCD-IPM} & \textbf{GG} & \textbf{BCD-IPM} & \textbf{GG} & \textbf{BCD-IPM} \\
		\midrule
		$(10,48)$ & 5 & \textbf{2} & 1.909 & \textbf{0.546} & 0.050 & \textbf{0.002} & \textbf{0.031} & 0.155 \\
		$(48,0)$  & 5 & \textbf{2} & 1.909 & \textbf{0.466} & 0.067 & \textbf{0.001} & \textbf{0.023} & 0.132 \\
		$(0,8)$   & 7 & \textbf{4} & 3.687 & \textbf{1.386} & 0.046 & \textbf{0.004} & \textbf{0.017} & 0.059 \\
		$(8,26)$  & 9 & \textbf{6} & 5.887 & \textbf{1.430} & 0.033 & \textbf{0.002} & \textbf{0.010} & 0.050 \\
		$(26,14)$ & 12 & \textbf{9} & 10.113 & \textbf{3.005} & 0.033 & \textbf{0.005} & \textbf{0.007} & 0.026 \\
		\bottomrule
	\end{tabular}
\end{table*}
As depicted in Fig. \ref{fig:topology10_case_study}, we establish a microscopic visualization of the routing trajectories mapped across Topology 10. To construct a routing strategy reference, physical links shared with the SP-SA shortest-path baseline are rendered in light grey, whereas activated topological links are delineated in dark charcoal. Furthermore, the thickness of each active link is scaled in proportion to its allocated data flow volume. To deconstruct the localized power allocation, we isolate three specific commodities ($(10, 17)$, $(45, 44)$, and $(56, 23)$) that trigger multi-path routing within the BCD-IPM algorithm. The routing paths servicing these targeted commodities are displayed with a thermal colormap; the color continuously transitions from deep blue to dark red to reflect the escalating transmission power (in Watts) allocated in each link.

As illustrated in Fig. \ref{fig:traj_ipm} and  \ref{fig:traj_gg}, both algorithms  activate distinguished routing paths relative to the SP-SA baseline. The proposed BCD-IPM algorithm explores the spatial routing domain to achieve the optimal delay-energy trade-off. This algorithm dynamically activates multi-path offloading specifically when the energy-to-throughput translation rate on primary paths deteriorates. Conversely, the GG baseline aggressively selects single-path routes based on a greedy marginal-delay, evaluating the ratio between expected flows and  transmission rates without considering the physical resource penalties.

This algorithmic discrepancy is visually striking. In Fig. \ref{fig:traj_ipm}, the thermal colormap reveals that the power allocation across the vast majority of BCD-IPM links is strictly bounded below $ 0.1 $ W, dominated by cool blue and light red links. In stark contrast, Fig. \ref{fig:traj_gg} is heavily saturated with dark red links, exposing the massive energy consumption of the GG heuristic. This severe power inflation is particularly severe along the backbone route $(10 \to 48 \to 0 \to 8 \to 26 \to 14 \to 59 \to 42 \to 54 \to 24 \to 21 \to 45 \to 17)$ for commodity~$17$.

To deconstruct this phenomenon, Table \ref{tab:microscopic_bottleneck} records the microscopic resource allocation profile along this specific backbone route. The empirical data exposes a severe resource over-allocation within the GG baseline. For instance, across the link $(48,0)$, GG forces an aggregation of $5$ distinct commodities. To push this massive load through the bottleneck and maintain a superficially low link delay of $0.023$ s for Commodity $ 17 $, GG inflates this link's transmission power and bandwidth to $0.067$ W and $1.909$ MHz, respectively.
In stark contrast, BCD-IPM overcomes this resource over-allocation via load balancing. It traverses Commodity $ 17 $ through different paths in favor of underutilized topological regions (such as nodes $ 31 $ and $ 47 $), bypassing the heavily congested nodes (such as nodes $ 59, 42, $ and $ 54 $) . On this critical link $(48,0)$, BCD-IPM slashes the required transmission power to a mere $ 0.001 $ W by limiting the local aggregation to only $ 2 $ commodities. 

Scaling this microscopic observation to the network level, to service all $20$ commodities, the GG baseline consumes all the $100$ MHz of total network bandwidth and $1.679$ W of accumulated transmission power. Conversely, the BCD-IPM algorithm utilizes merely $45.821$ MHz of bandwidth and $0.136$ W of total power to service these commodities. This constitutes an astonishing {$ 12.3 $-fold reduction in total transmission power} (an approximately $91.9\%$ decrease) and a {$54.2\%$ saving in bandwidth consumption}. While the BCD-IPM yields a  higher maximum end-to-end delay ($3.84$ s versus the $1.78$ s achieved by GG), it prevents the severe network-wide resource depletion.

\begin{figure*}[t]
	\centering
	
	\begin{subfigure}[b]{0.32\textwidth}
		\centering
		\includegraphics[width=\textwidth]{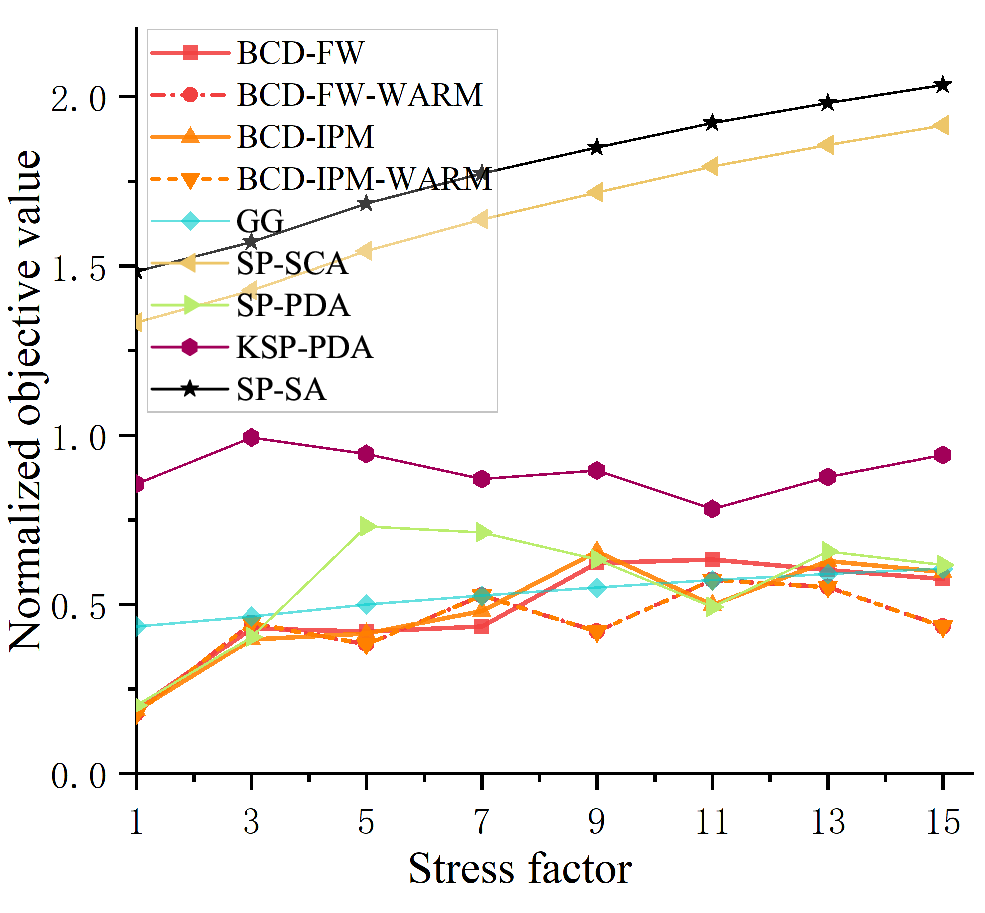}
		\caption{Normalized objective value in \eqref{normalized_obj}}
		\label{fig:obj}
	\end{subfigure}
	\hspace{0.00001\textwidth}
	\begin{subfigure}[b]{0.32\textwidth}
		\centering
		\includegraphics[width=\textwidth]{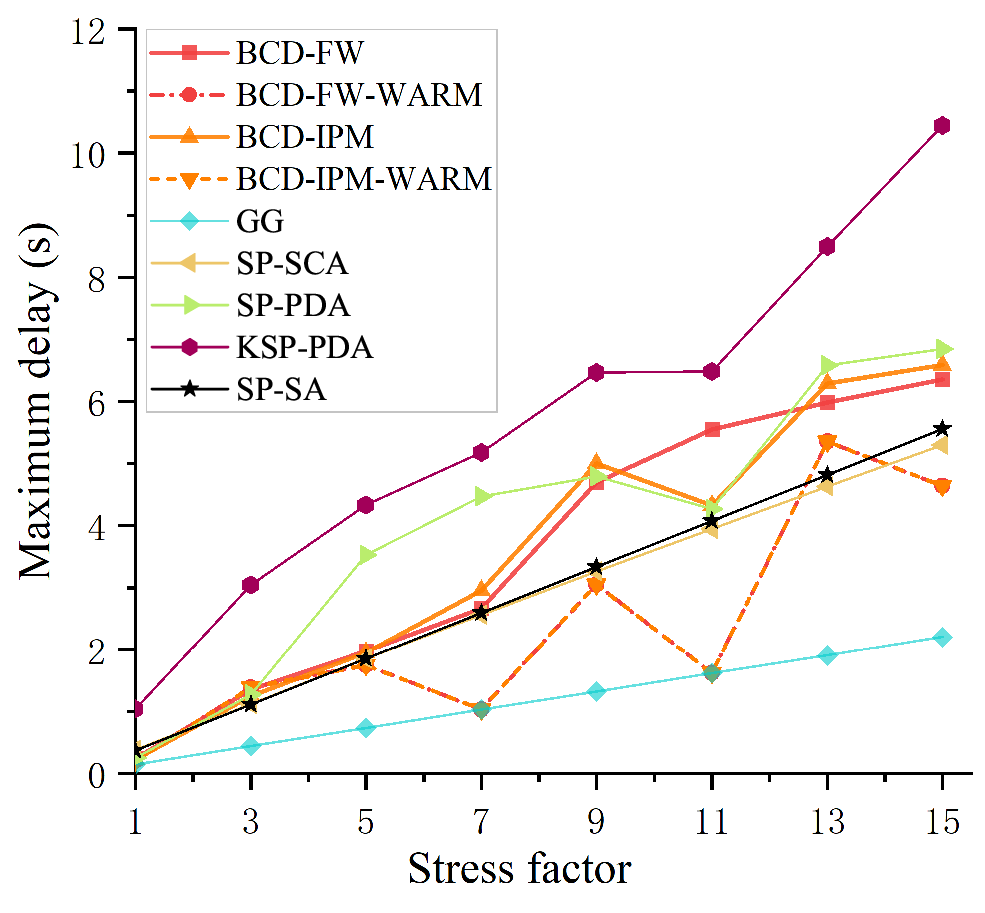}
		\caption{Maximum delay}
		\label{fig:delay_resilience}
	\end{subfigure}
	\hspace{0.00001\textwidth}
	\begin{subfigure}[b]{0.32\textwidth}
		\centering
		\includegraphics[width=\textwidth]{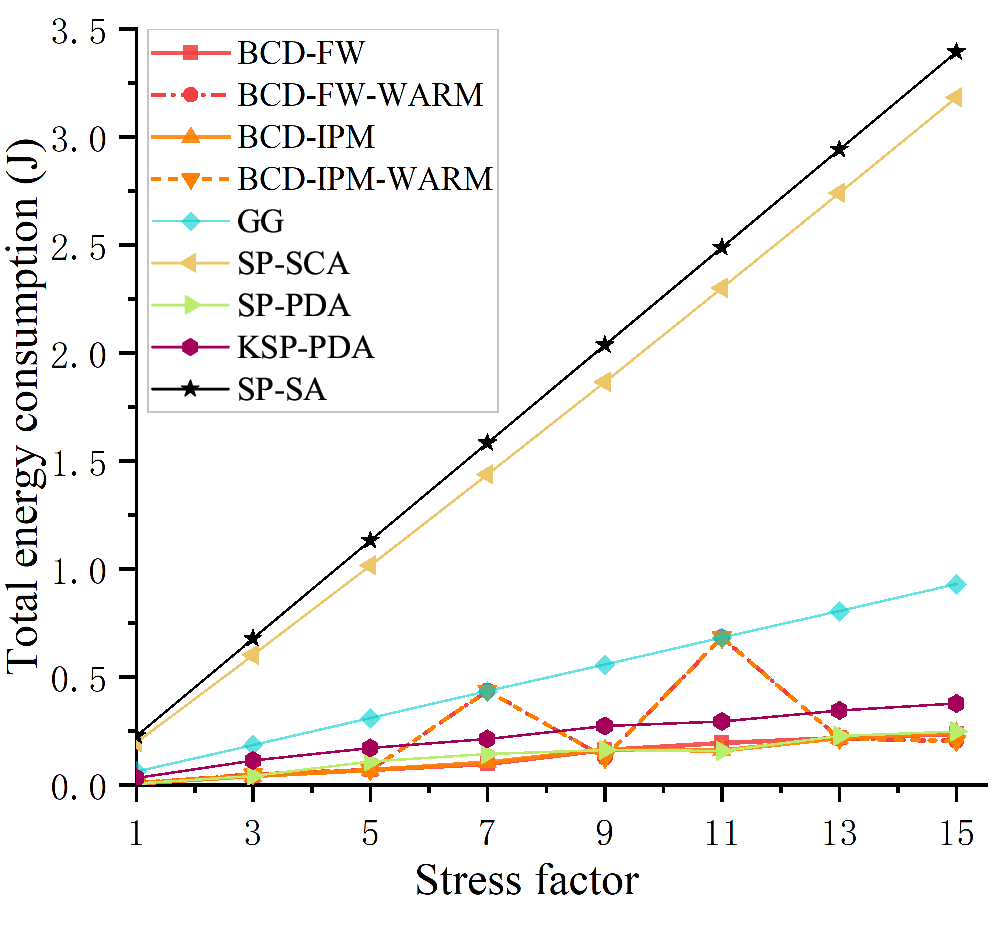}
		\caption{Total energy consumption}
		\label{fig:energy}
	\end{subfigure}
	
	\vspace{1.5em} 
	
	\begin{subfigure}[b]{0.32\textwidth}
		\centering
		\includegraphics[width=\textwidth]{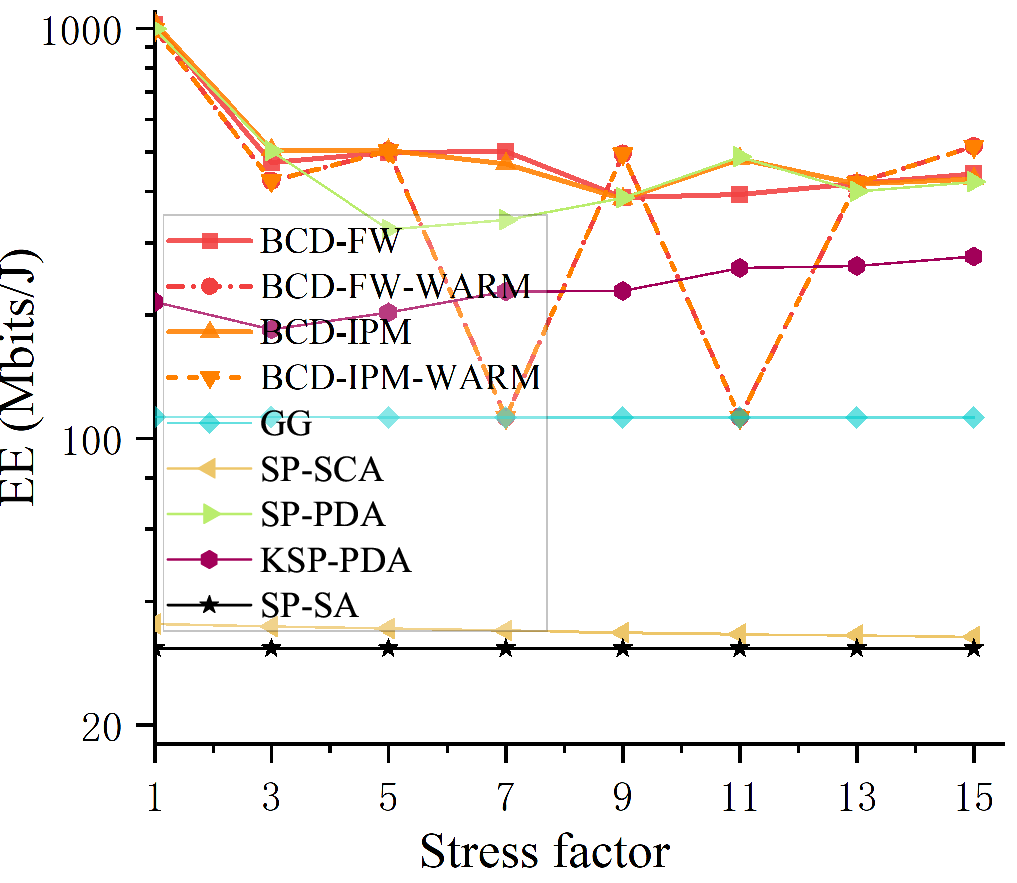}
		\caption{Energy efficiency (EE)}
		\label{fig:EE}
	\end{subfigure}
	\hspace{0.00001\textwidth}
	\begin{subfigure}[b]{0.32\textwidth}
		\centering
		\includegraphics[width=\textwidth]{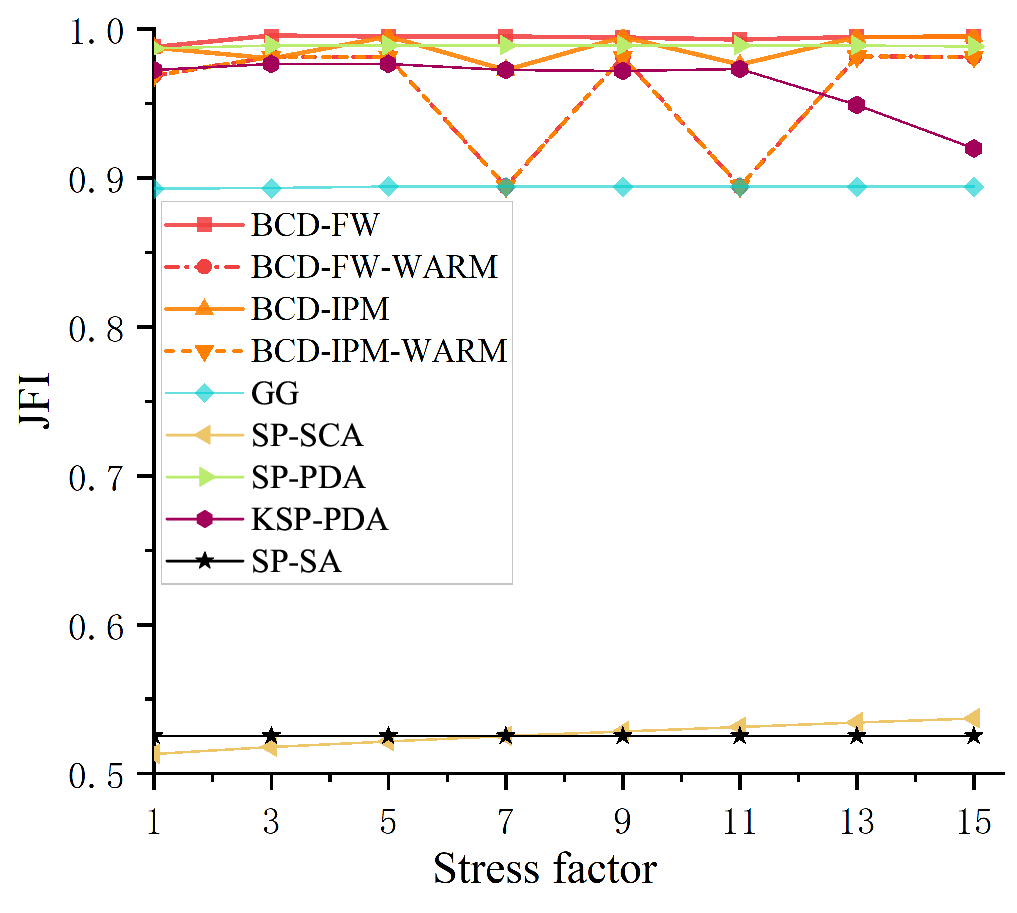}
		\caption{Jain's fairness index (JFI)}
		\label{fig:JFI}
	\end{subfigure}
	\hspace{0.00001\textwidth}
	\begin{subfigure}[b]{0.32\textwidth}
		\centering
		\includegraphics[width=\textwidth]{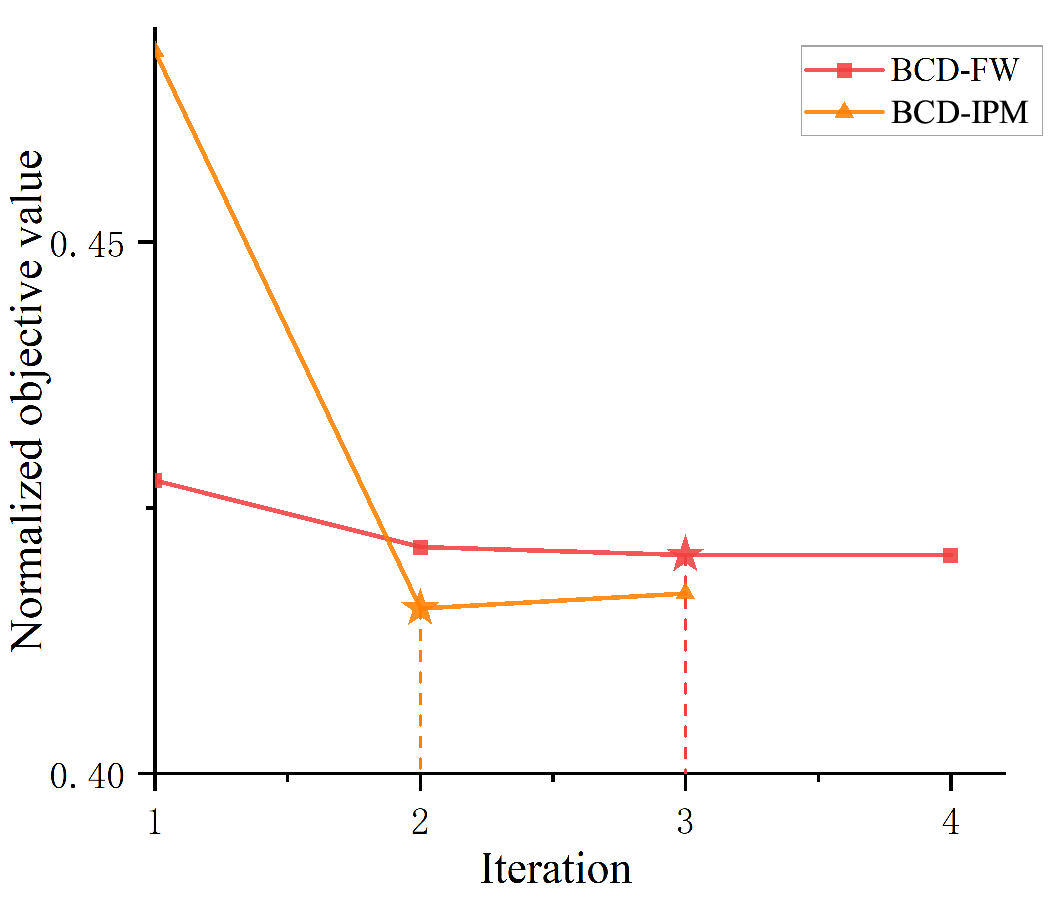}
		\caption{Convergence trajectory under $ \xi = 5 $}
		\label{fig:convergence}
	\end{subfigure}
	
	\caption{Comprehensive performance evaluation of the proposed four BCD algorithms and baselines under varying stress factors. }
	\label{fig:stress_evaluation_all}
\end{figure*}

\subsection{Algorithmic resilience against traffic stress}
\label{subsec:stress_evaluation}


We evaluate the algorithmic resilience by varying the stress factor $\xi \in [1, 15]$ on the traffic demand $ M^k $.
As illustrated in Fig. \ref{fig:obj}, the proposed four BCD algorithms consistently dominates the optimal performance frontier. Under low-to-moderate traffic loads ($\xi \le 7$), the BCD-FW and BCD-IPM outperform the baseline heuristics, maintaining minimized normalized objective values as defined in \eqref{normalized_obj}. As the network stress factor intensifies ($\xi \ge 9$), these algorithms experience slight inflation, caused by navigating a severely congested feasible region. Remarkably, the warm-started algorithms (BCD-FW-WARM and BCD-IPM-WARM) neutralize this degradation, securing superior minima in most cases. This stability stems from the synergy with a high-quality initial solution derived from the GG heuristic. Conversely, while the GG heuristic exhibits relatively stable behavior under extreme stress ($\xi \ge 9$), it operates on a sub-optimal baseline. Under low-to-moderate conditions ($\xi \le 7$), GG consistently yields inferior objective values compared to the proposed BCD framework.

As the stress factor increases ($\xi \ge 9$), an intriguing phenomenon emerges. Although the warm-started algorithms stabilize the objective trajectory by leveraging the priors of the GG heuristic, they are occasionally susceptible to the local optimum trap. Specifically, at $\xi = 11$, both BCD-FW-WARM and BCD-IPM-WARM converge to an objective value of $0.573$, identical to the GG baseline. This indicates that the GG initialization anchores these two algorithms within a local minimum. In contrast, the proposed BCD-IPM and the baseline SP-PDA discover superior solutions, yielding optimal objective values of $0.498$ and $0.492$, respectively.
Despite these localized anomalies, the SP-PDA and KSP-PDA heuristics occupy the second performance tier in the normalized objective. This confirms that the optimal PDA resource allocation remains effective at mitigating congestion and compressing the energy consuption, even when constrained to single-path routing geometries. 
The SP-SCA and SP-SA baselines consistently collapse into the lowest performance tier, exhibiting a massive optimality gap. Ultimately, this comprehensive resilience evaluation proves that the superiority of the proposed BCD framework stems from the synergy between the multi-path routing and optimal PDA resource allocation.

Fig. \ref{fig:delay_resilience} and  \ref{fig:energy} illustrate the individual objective performance for the evaluated algorithms. The  proposed BCD-FW and BCD-IPM dominate in total energy consumption. While these two algorithms exhibit the second performance tier regarding maximum delay, they consistently maintain a bounded gap between $1.55\times$ and  $3.78\times$ relative to the GG baseline.  To achieve its delay advantage, the GG heuristic consumes $3.41\times$ to $9.14\times$ more energy compared to the  BCD-FW and BCD-IPM. 
By anchoring their initialization with the GG baseline, the warm-started BCD-FW-WARM and BCD-IPM-WARM achieve highly competitive maximum delays.   They consistently guarantee a worst-case trade-off (e.g. $\xi = 7,11$) bounded by the GG heuristic.
The SP-PDA and KSP-PDA consistently secure the   minimum in total energy consumption.

In contrast, the SP-SCA and SP-SA baselines exhibit massive energy waste. While the failure of SP-SA is driven by its blind, static equal-partitioning of resources, the behavior of SP-SCA  exposes the  fragility of SCA strategy in heavily congested, non-convex landscapes. The proximal gradient descent mechanism within SP-SCA struggles with compounding approximation errors and poor local optima, ultimately consuming immense energy (peaking at $3.183$ J at $\xi = 15$). This proves that without the exact gradient-based resource optimization inherent to the BCD framework, the network loses its ability to efficiently translate consumed energy into data transmission.

As illustrated in Fig. \ref{fig:EE}, an evaluation of EE across varying network stress unveils the resource-utilization characteristics of the evaluated algorithms. The proposed BCD-FW and BCD-IPM algorithms consistently dominate this metric, maintaining an exceptionally high EE (bounded between a minimum of $383.73$ and a maximum of $1028.95$) across the entire stress spectrum. Following closely are the warm-started BCD-FW-WARM and BCD-IPM-WARM, which exhibit highly competitive EE profiles except at $\xi = 7$ and $\xi = 11$, where they degenerate into the initial GG solutions. This profound superiority demonstrates that the exact multi-path routing strategy does not blindly flood the network; instead, it maximizes the throughput-to-energy ratio by allocating fractional flows only to the paths with superior transmission rates.

An insightful phenomenon emerges among the baseline algorithms. The single-path SP-PDA algorithm exhibits remarkable energy efficiency (ranging from $322.80$ to $998.04$), significantly outperforming the multi-path KSP-PDA (which stagnates between $184.15$ and $277.52$). This discrepancy exposes the fatal flaw of heuristic multi-path routing: by partitioning traffic across all the $ K $-shortest paths, KSP-PDA inevitably forces flows through secondary links with severe congestion. This flaw penalizes its energy efficiency. In contrast, SP-PDA restricts transmission to the shortest path with the highest channel gain ($1/h_{ij}$) and optimizes resource allocation via PDA, thereby minimizing energy waste. 
Finally, the conventional heuristics experience a complete EE collapse. The GG remains trapped at EE of approximately $112.50$ regardless of load, further illustrating its strong reliance on aggressive power and bandwidth consumption. The SP-SCA and SP-SA  yield the lowest EE values of approximately $35.26$ and $30.82$, respectively. This validates previous analyses: while SP-SA is restricted by its equal-partitioning of resources, SP-SCA  struggles with compounding approximation errors and poor local optima.  Both baselines exhibit immense energy waste.

As illustrated in Fig. \ref{fig:JFI}, the JFI metric provides a microscopic view into the detailed delay distribution and service equity across all  $ |\mathcal{K}| $ commodities. The proposed BCD algorithms consistently achieve near-perfect fairness (JFI $> 0.894$), with the  BCD-FW and BCD-IPM peaking at $0.995$. This demonstrates that the routing subproblem model \eqref{eq:smoothed_routing} does not merely optimize the maximum delay; it guarantees the delay equalization for each commodity.
An intriguing phenomenon emerges within the SP-PDA and KSP-PDA baselines, which initially appear to occupy the top performance tier. While the optimal PDA resource allocation enforces delay equalization across commodities, this fairness is achieved at the severe expense of the inflated delays previously observed in Fig. \ref{fig:delay_resilience}. 
In the subsequent tiers, the greedy nature of the GG heuristic becomes evident. Stagnating at a JFI of approximately $0.894$ under high stress ($\xi \ge 5$), GG exhibits significant delay fluctuations. It prioritizes certain routes while starving others due to its greedy path selection based on the marginal-delay. Finally, the SP-SCA and SP-SA baselines suffer from severe unfairness, achieving similar JFI values bounded between $0.513$ and $0.537$. This massive disparity in end-to-end delay can be attributed respectively to compounding approximation errors and the equal partition strategy in single-path routing.

As illustrated in Fig. \ref{fig:convergence}, the convergence trajectories of the proposed BCD-FW and BCD-IPM algorithms are evaluated under a medium congestion scenario ($\xi=5$). Remarkably, both algorithms exhibit ultra-fast convergence, stabilizing at their respective optima within merely $4$ outer iterations.   Specifically, as denoted by the star markers, the optimal objective values are secured at the $3$rd and $2$nd iterations for BCD-FW and BCD-IPM, respectively.
A deeper inspection of the convergence trajectories reveals the theoretical distinction between the first-order and second-order algorithms. Initially, BCD-IPM starts at a higher normalized objective value ($0.4679$) compared to BCD-FW ($0.4276$). This behavior is expected: unlike FW, which operates on the extreme vertices of the feasible polytope, IPM strictly enforces interior feasibility via barrier penalties. However, empowered by the exact second-order Hessian information, BCD-IPM captures the exact geometric curvature of the objective.  BCD-IPM ultimately discovers a superior high-quality solution ($0.4155$ versus $0.4206$ of BCD-FW). This validates the advantage of exploiting Hessian-guided search directions in such highly coupled non-linear problems.

It is crucial to evaluate the approximation gap between the original objective \eqref{eq:original_obj} and its upper bound surrogate \eqref{obj_app}. Remarkably, across all $50$ test instances, encompassing the Pareto frontier, resilience, scalability, and robustness scenarios, the approximation gap of the proposed BCD framework is predominantly zero. Specifically, the proposed BCD-FW and warm-started BCD-FW-WARM and BCD-IPM-WARM algorithms achieve a zero gap across all $50$ instances. The BCD-IPM yields a non-zero gap in only $4$ out of $50$ cases, with an average relative gap of $ 0.017 $.  These minimal approximation gaps stem from the topological nature of the network bottleneck. In the vast majority of congestion scenarios, the maximum delay is dominated by a specific bottleneck commodity that is routed along a single path. Under such single-path bottleneck conditions, the delay relaxation on each commodity perfectly tightens, rendering the approximate objective \eqref{obj_app} exactly equivalent to the original objective \eqref{eq:original_obj}. This confirms that the proposed upper-bound relaxation is not merely a theoretical convenience, but a highly accurate representation of the underlying routing bottlenecks. 

In summary, the preceding resilience analysis establishes the decisive superiority of the proposed BCD framework in achieving the optimal delay-energy trade-off. This joint optimization framework maximizes the energy-to-throughput translation efficiency (yielding up to an order of magnitude improvement); it guarantees bounded transmission times (less than $ 3.78\times $ relative to the baseline GG) while operating at the minimized energy expenditure (up to a $ 9.14\times $ improvement versus GG). Having validated the algorithms' extreme resilience under severe traffic congestion, we subsequently transition to scalability experiments.

\begin{figure*}[t]
	\centering
	
	\begin{subfigure}[b]{0.32\textwidth}
		\centering
		\includegraphics[width=\textwidth]{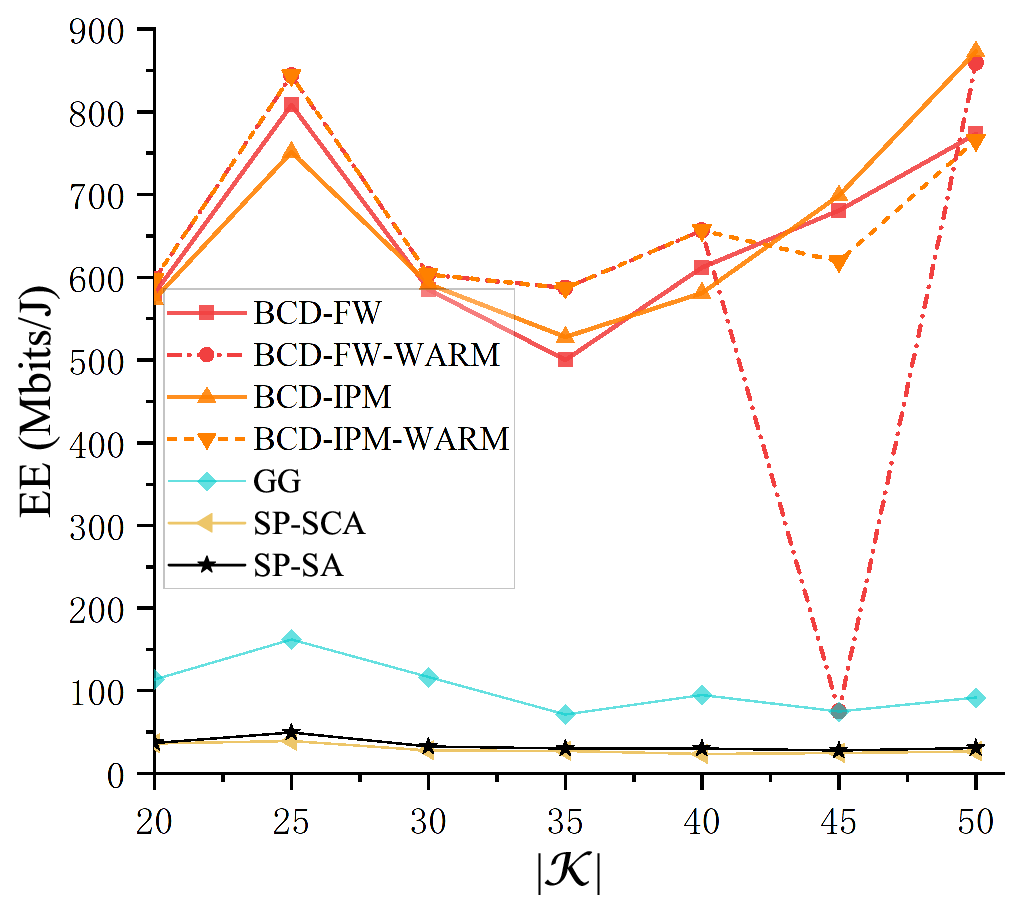}
		\caption{Energy efficiency vs. $ |\mathcal{K}| $}
		\label{fig:EE_vs_K}
	\end{subfigure}
	\hspace{0.0001\textwidth} 
	\begin{subfigure}[b]{0.32\textwidth}
		\centering
		\includegraphics[width=\textwidth]{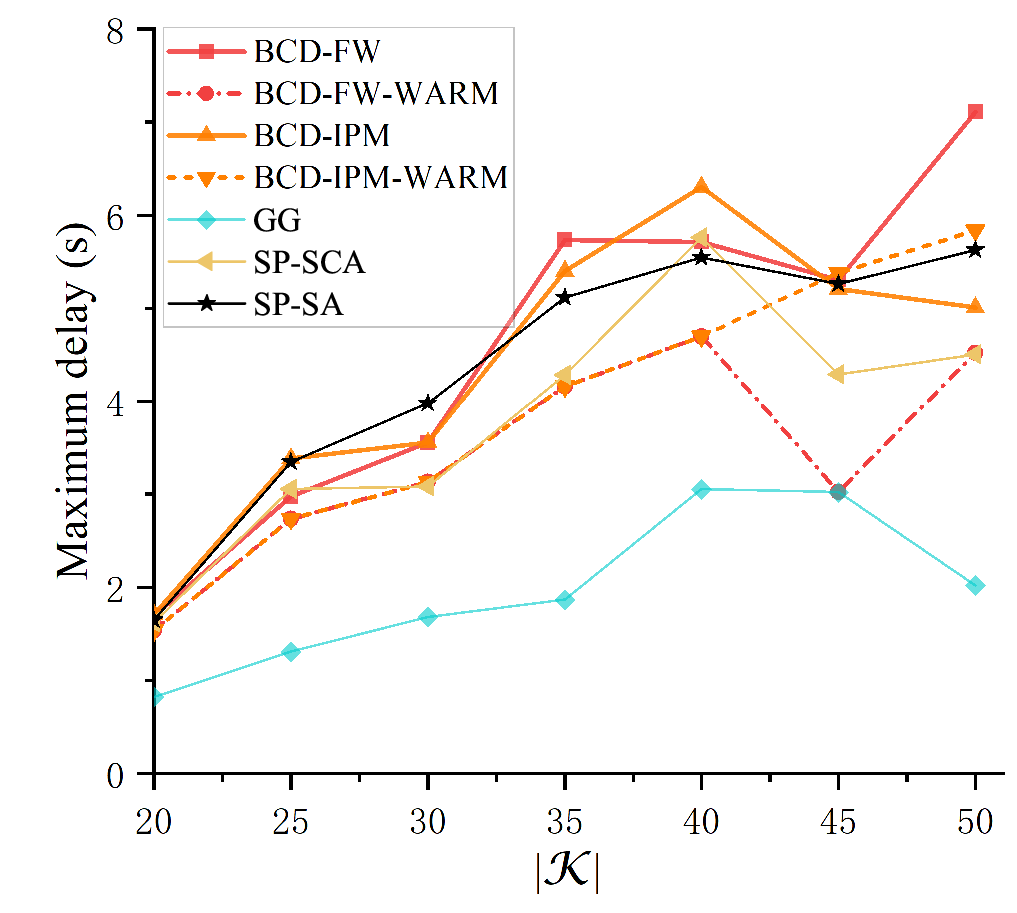}
		\caption{Maximum delay vs. $ |\mathcal{K}| $}
		\label{fig:Tmax_vs_K}
	\end{subfigure}
	\hspace{0.0001\textwidth} 
	\begin{subfigure}[b]{0.32\textwidth}
		\centering
		\includegraphics[width=\textwidth]{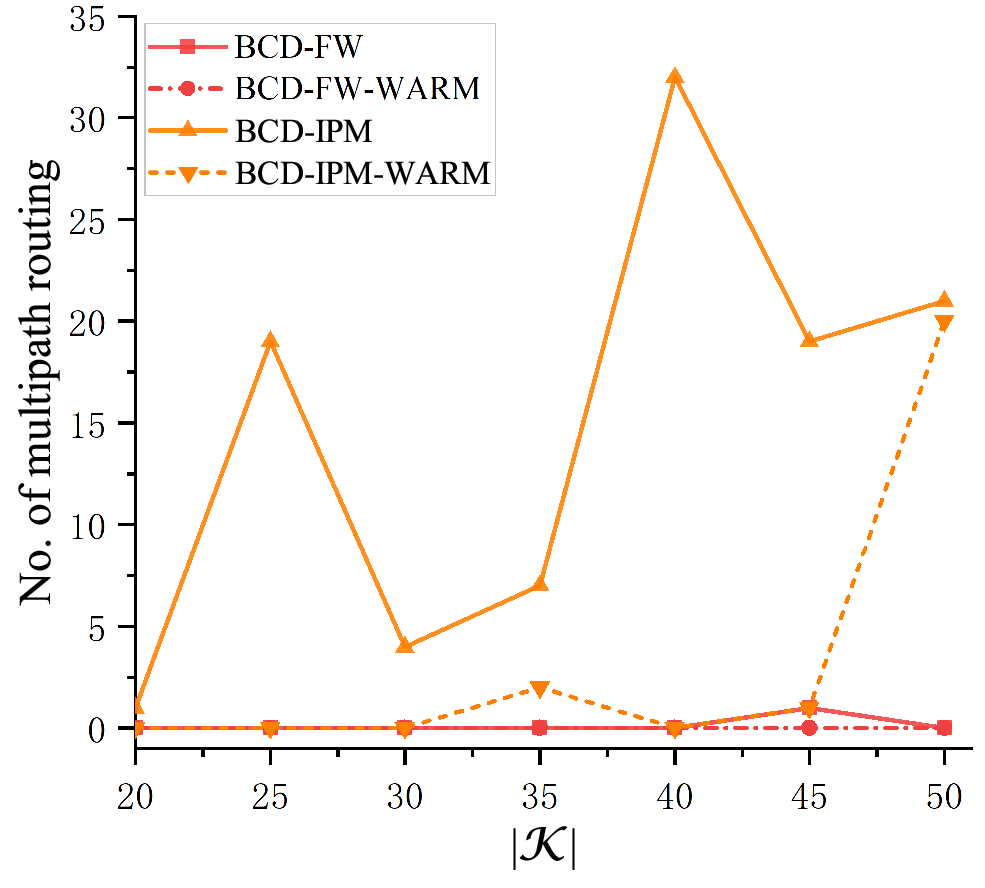}
		\caption{Multipath count vs. $ |\mathcal{K}| $}
		\label{fig:Multipaths_vs_K}
	\end{subfigure}
	
	\vspace{1.5em} 
	
	\begin{subfigure}[b]{0.32\textwidth}
		\centering
		\includegraphics[width=\textwidth]{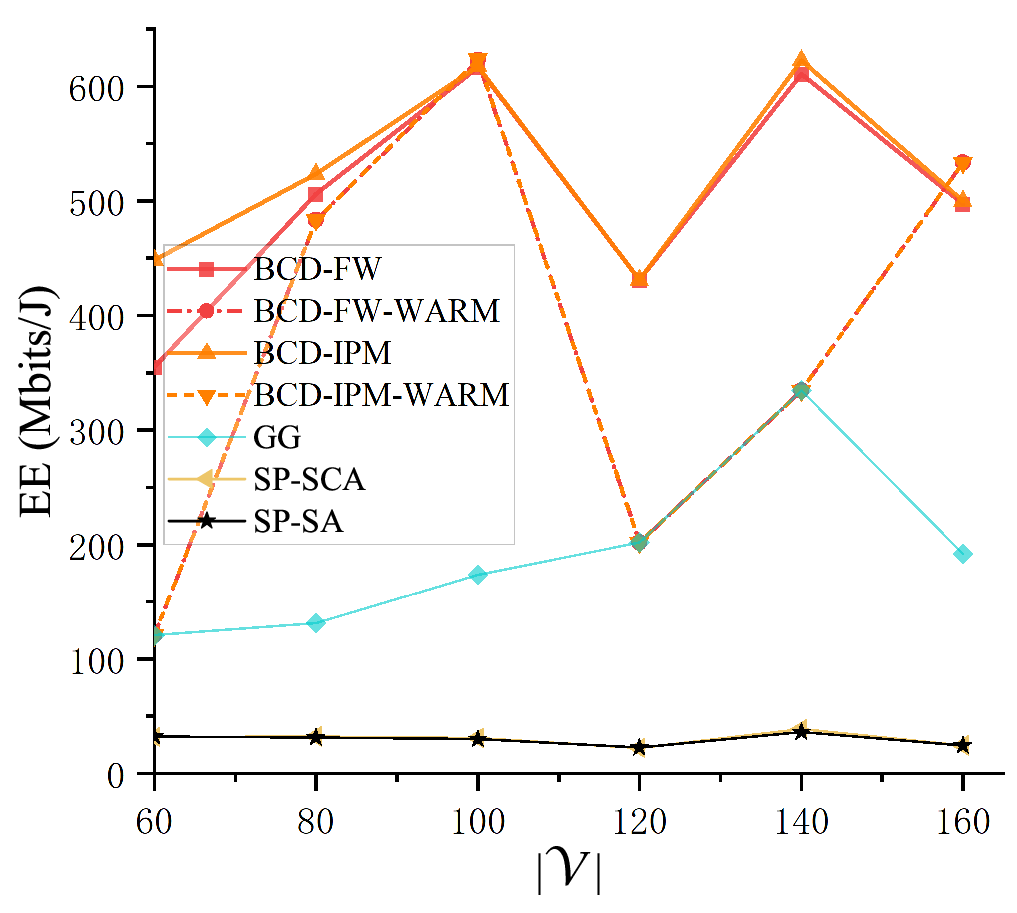}
		\caption{Energy efficiency vs. $|\mathcal{V}|$}
		\label{fig:EE_vs_N}
	\end{subfigure}
	\hspace{0.0001\textwidth} 
	\begin{subfigure}[b]{0.32\textwidth}
		\centering
		\includegraphics[width=\textwidth]{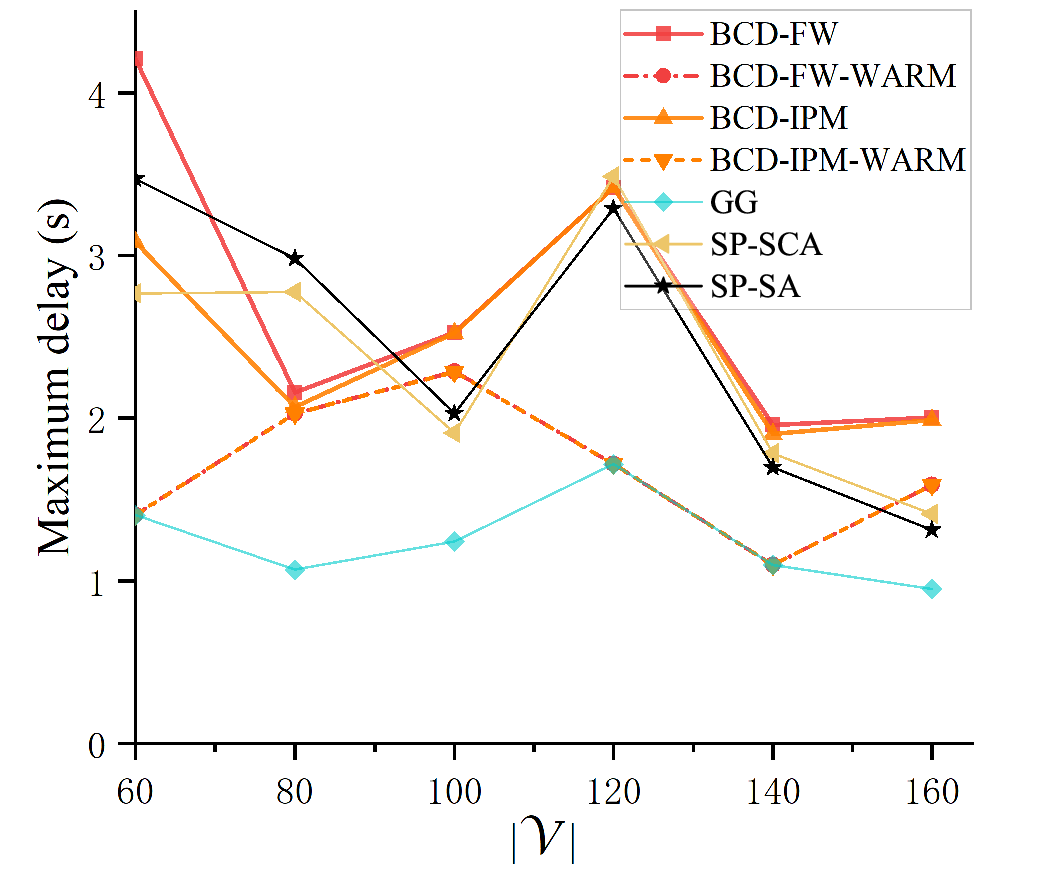}
		\caption{Maximum delay vs. $|\mathcal{V}|$}
		\label{fig:Tmax_vs_N}
	\end{subfigure}
	\hspace{0.0001\textwidth} 
	\begin{subfigure}[b]{0.32\textwidth}
		\centering
		\includegraphics[width=\textwidth]{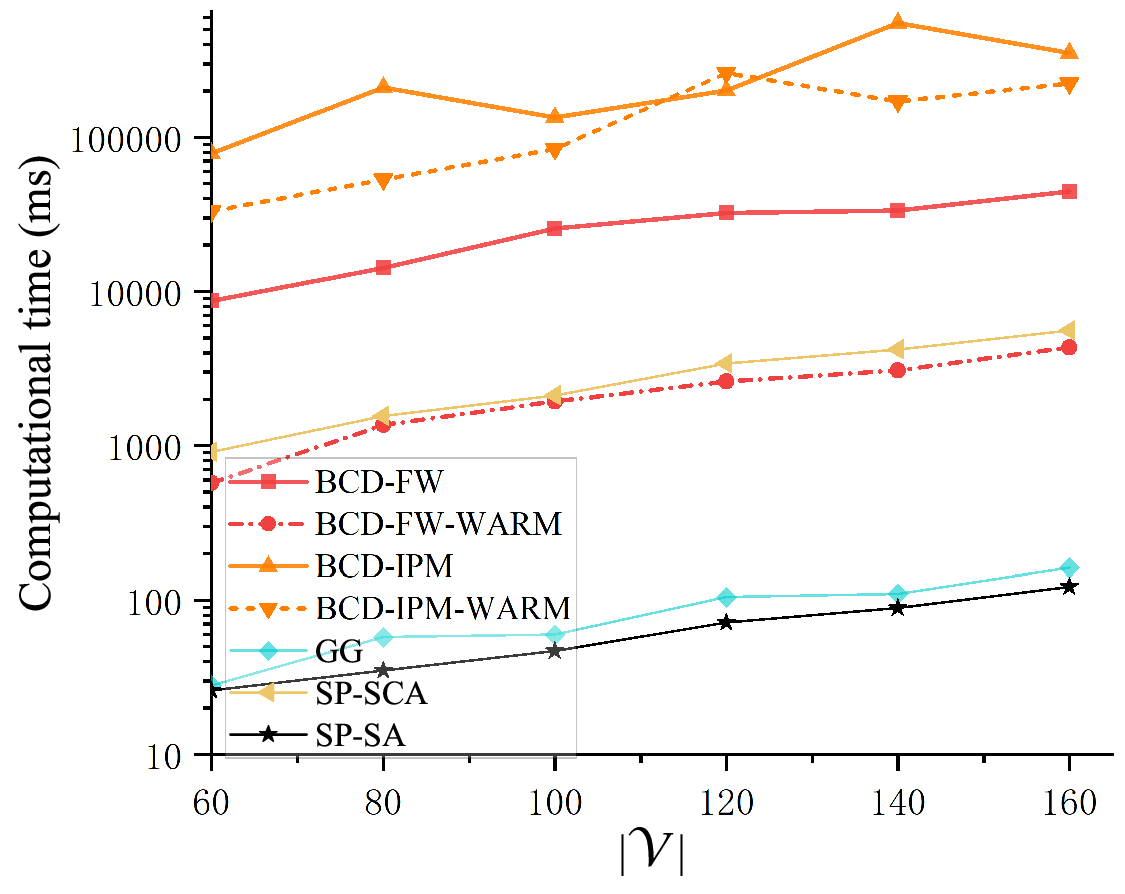}
		\caption{Computational time vs. $|\mathcal{V}|$}
		\label{fig:Time_vs_N}
	\end{subfigure}
	
	\caption{Scalability performance evaluation of the proposed BCD algorithms and baselines.} 
	\label{fig:scalability_results_all}
\end{figure*}

\subsection{Algorithmic scalability against network scale}
\label{subsec:scalability_evaluation}

We expand the network load by first varying the number of commodities $|\mathcal{K}| \in [20, 50]$ in increments of $5$, followed by scaling the device node cardinality $|\mathcal{V}| \in [60, 160]$ in increments of $20$. It is imperative to note that SP-PDA and KSP-PDA are omitted in subsequent analyses, as they exhibit maximum delays surging drastically above $89$ s (e.g., $89.92$ s for KSP-PDA at $|\mathcal{K}|=30$, and  $402.41$ s for SP-PDA at $|\mathcal{K}|=50$).
As illustrated in Fig. \ref{fig:EE_vs_K} and  \ref{fig:Tmax_vs_K}, BCD-IPM-WARM demonstrates profound scalability. It consistently dominates the EE metrics in most cases while occupying the second performance tier regarding maximum delay. Remarkably, even under the extreme network load of $|\mathcal{K}| = 50$, its maximum delay experiences only a controlled inflation to $5.84$ s, while preserving an exceptionally high EE of $765.08$. Consequently, compared to the GG baseline at this peak stress, BCD-IPM-WARM  confines its maximum delay gap to less than $2.89\times$, while achieving a massive $8.35\times$ improvement in energy efficiency. BCD-FW-WARM exhibits a similar performance, though it degenerates to the initial GG solution at $|\mathcal{K}| = 45$ (yielding identical EE and delay values of $74.80$ and $3.02$ s, respectively).

Furthermore, the BCD-FW and BCD-IPM algorithms exhibit the robust energy efficiency under high-density traffic, guaranteeing up to a $9.52\times$ EE improvement over GG ($|\mathcal{K}|=50$). They secure reasonable maximum delays, bounding the gap relative to the GG baseline within a worst-case factor of $3.53\times$ ($|\mathcal{K}|=50$).
Conversely, while the GG heuristic manages to maintain a lower maximum delay, it suffers from severely degraded energy efficiency. This exposes the drawback of GG: to compensate for severe single-path congestion bottlenecks, it aggressively consumes energy to brute-force data transmission, resulting in massive and unsustainable resource waste.

As illustrated in Fig. \ref{fig:Multipaths_vs_K}, evaluating the number of commodities employing multi-path routing across varying $|\mathcal{K}|$ unveils the underlying topological behavior of the proposed algorithms. (Note: single-path baselines consistently yield zero and are therefore omitted). The proposed BCD framework exhibits a highly adaptive multi-path activation strategy. Rather than blindly splitting flows under light traffic, this framework exclusively activates secondary paths when localized bottlenecks experience severe congestion and resource starvation. For instance, in response to escalating network loads, the BCD-IPM scales its multi-path routing from $1$ (at $|\mathcal{K}|=20$) to a peak of $32$ (at $|\mathcal{K}|=40$), before stabilizing at $21$ under extreme stress ($|\mathcal{K}|=50$). Conversely, BCD-FW maintains an exceptionally low multi-path activation quantity (peaking marginally at $1$ when $|\mathcal{K}|=45$). This algorithm prefers to explore single-path routing to alleviate the energy consumption instead of activating multi-path routing. This stems from the nature of the first-order FW mechanism, which relies on sequential shortest-path evaluations over the gradient direction.

The BCD-FW-WARM is constrained by initial single-path solution, yielding exactly $0$ across all instances.
However, BCD-IPM-WARM exhibits a superior capability to explore multi-path routing. Even when anchored by the GG initialization, it overcomes the local bottleneck, activating $2$ and $20$ multi-paths at $|\mathcal{K}|=35$ and $|\mathcal{K}|=50$, respectively. This stems from the exploitation of the exact second-order Hessian matrix coupled with the logarithmic barrier penalties within the LR-PDIPM (Algorithm \ref{algo:pd_ipm}).  These techniques capture the non-linear curvature of the congestion landscape, pushing the flow distribution away from the heavily congested solutions and deeper into the multi-path routing space. This unique characteristic guarantees the exceptional algorithmic stability previously validated in Fig. \ref{fig:EE_vs_K} and  \ref{fig:Tmax_vs_K}.

As illustrated in Fig. \ref{fig:EE_vs_N} and  \ref{fig:Tmax_vs_N}, a comprehensive evaluation of EE and maximum delay is conducted under a scaling $|\mathcal{V}| \in [60, 160]$. As $|\mathcal{V}|$ expands, the average node degree surges from $6.1$ at $|\mathcal{V}|=60$ to $15.8$ at $|\mathcal{V}|=160$. This generates a highly dense topological structure. For the EE metric, the BCD-FW and BCD-IPM algorithms consistently define the optimal performance tier, securing peak EE values up to $616.48$ and $622.46$, respectively. The warm-started BCD-FW-WARM and BCD-IPM-WARM closely follow in the second tier; however, they degenerate into the initial GG solutions at specific topologies ($|\mathcal{V}| = 60, 120,$ and $140$), yielding identical performance metrics. Conversely, the GG heuristic stagnates at an inferior EE tier, exhibiting a worst-case gap of $3.99\times$ compared to the BCD-IPM optimal bound (specifically at $|\mathcal{V}|=80$). Furthermore, SP-SCA and SP-SA yield the lowest performance, demonstrating a significant inefficiency in translating consumed energy into effective data transmission.

Concurrently, Fig. \ref{fig:Tmax_vs_N} illustrates the algorithmic performance in maximum delay. While the GG heuristic forces a lower bound in maximum delay by penalizing energy efficiency, the proposed  BCD-FW and BCD-IPM achieve comparable transmission times. They maintain a controlled gap of no more than $3.00\times$ relative to GG (peaking at $|\mathcal{V}|=60$ for BCD-FW). Finally, the BCD-FW-WARM and BCD-IPM-WARM deliver a highly competitive intermediate delay performance while preserving massive energy advantages over the baselines.

\begin{figure*}[htbp]
	\centering
	
	\begin{subfigure}[b]{0.32\textwidth}
		\centering
		\includegraphics[width=\textwidth]{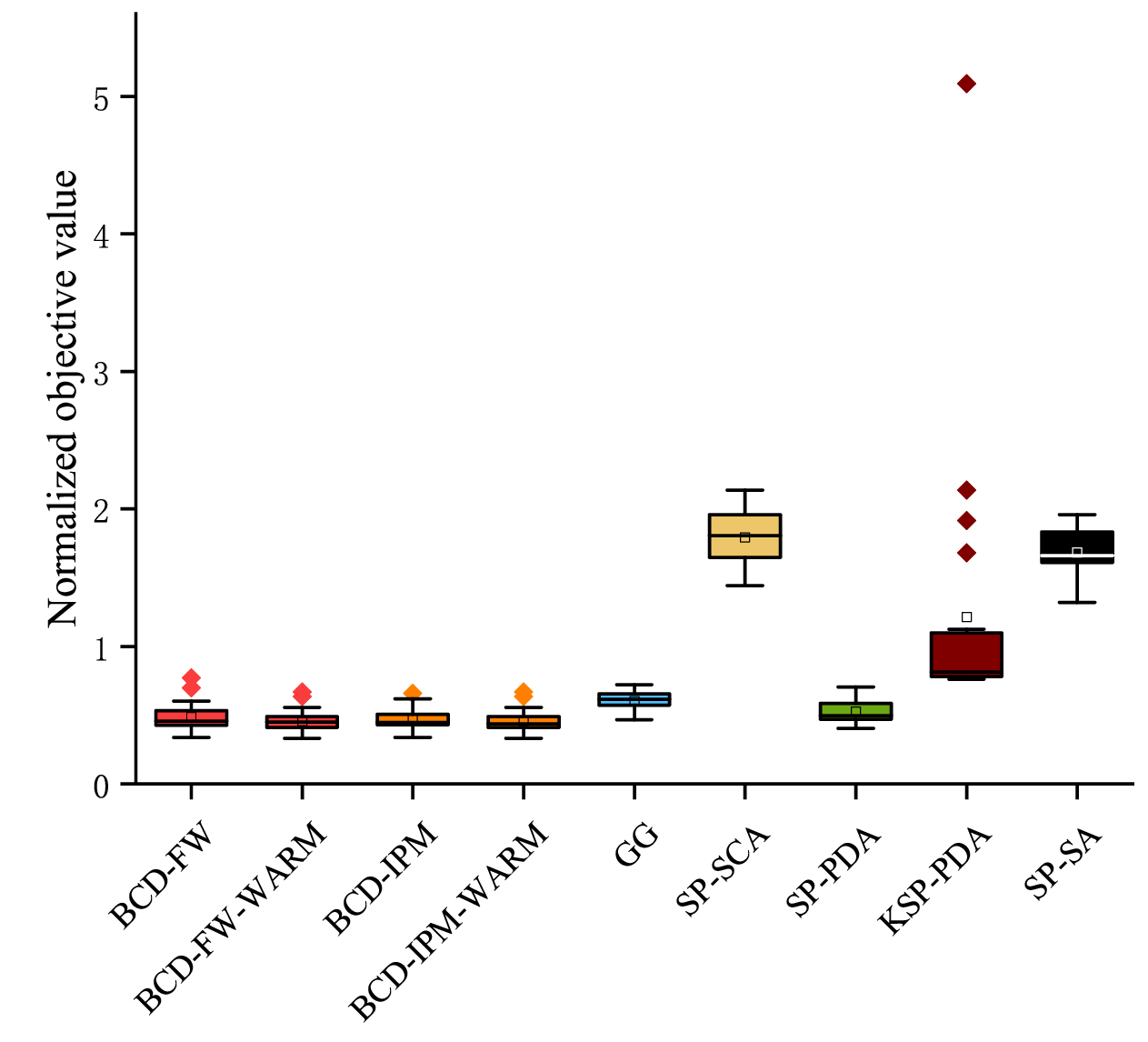}
		\caption{Normalized objective value distributions}
		\label{fig:Obj_box}
	\end{subfigure}
	\hspace{0.0001\textwidth} 
	\begin{subfigure}[b]{0.32\textwidth}
		\centering
		\includegraphics[width=\textwidth]{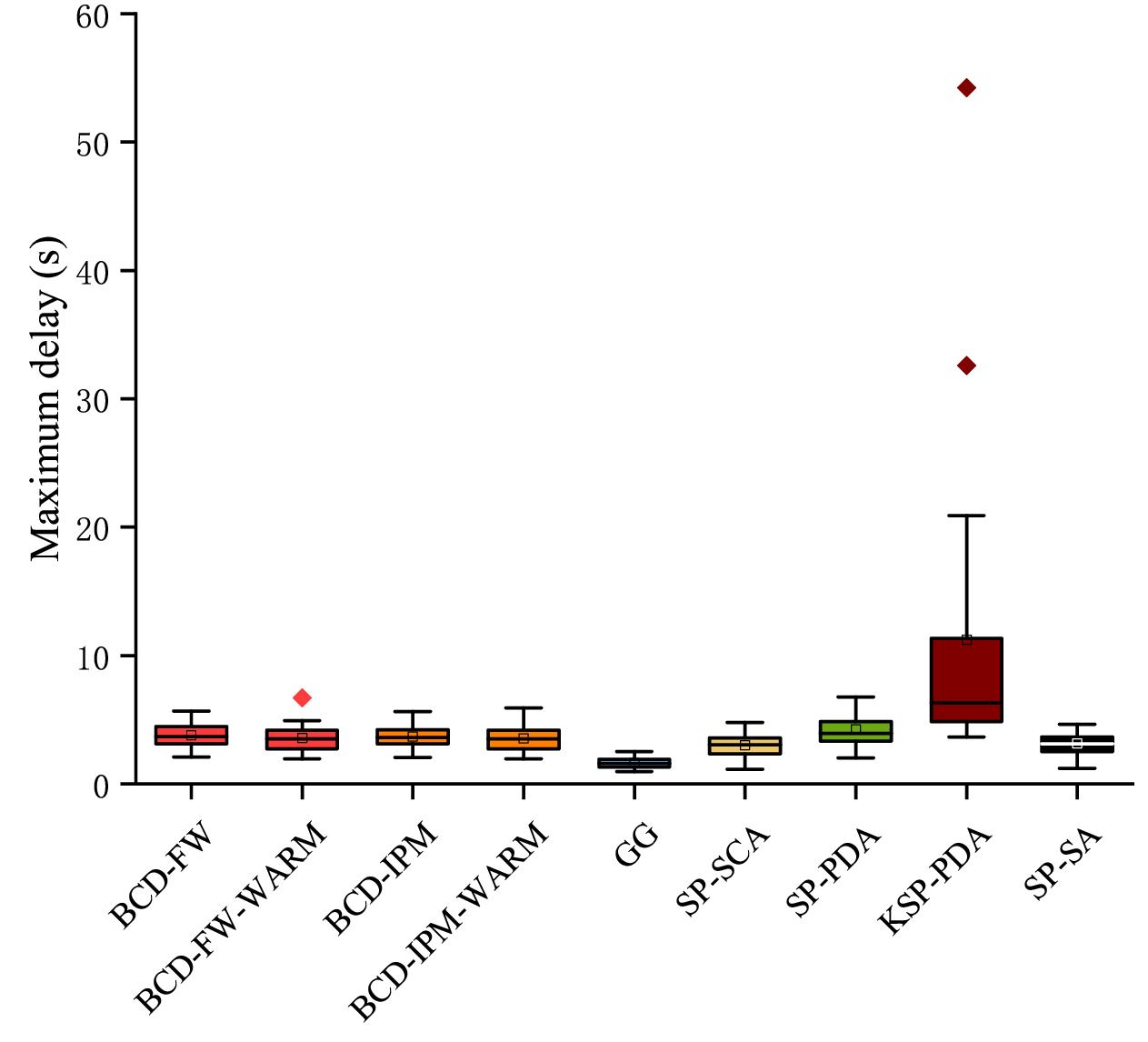}
		\caption{Maximum delay distributions}
		\label{fig:Tmax_box}
	\end{subfigure}
	\hspace{0.0001\textwidth} 
	\begin{subfigure}[b]{0.32\textwidth}
		\centering
		\includegraphics[width=\textwidth]{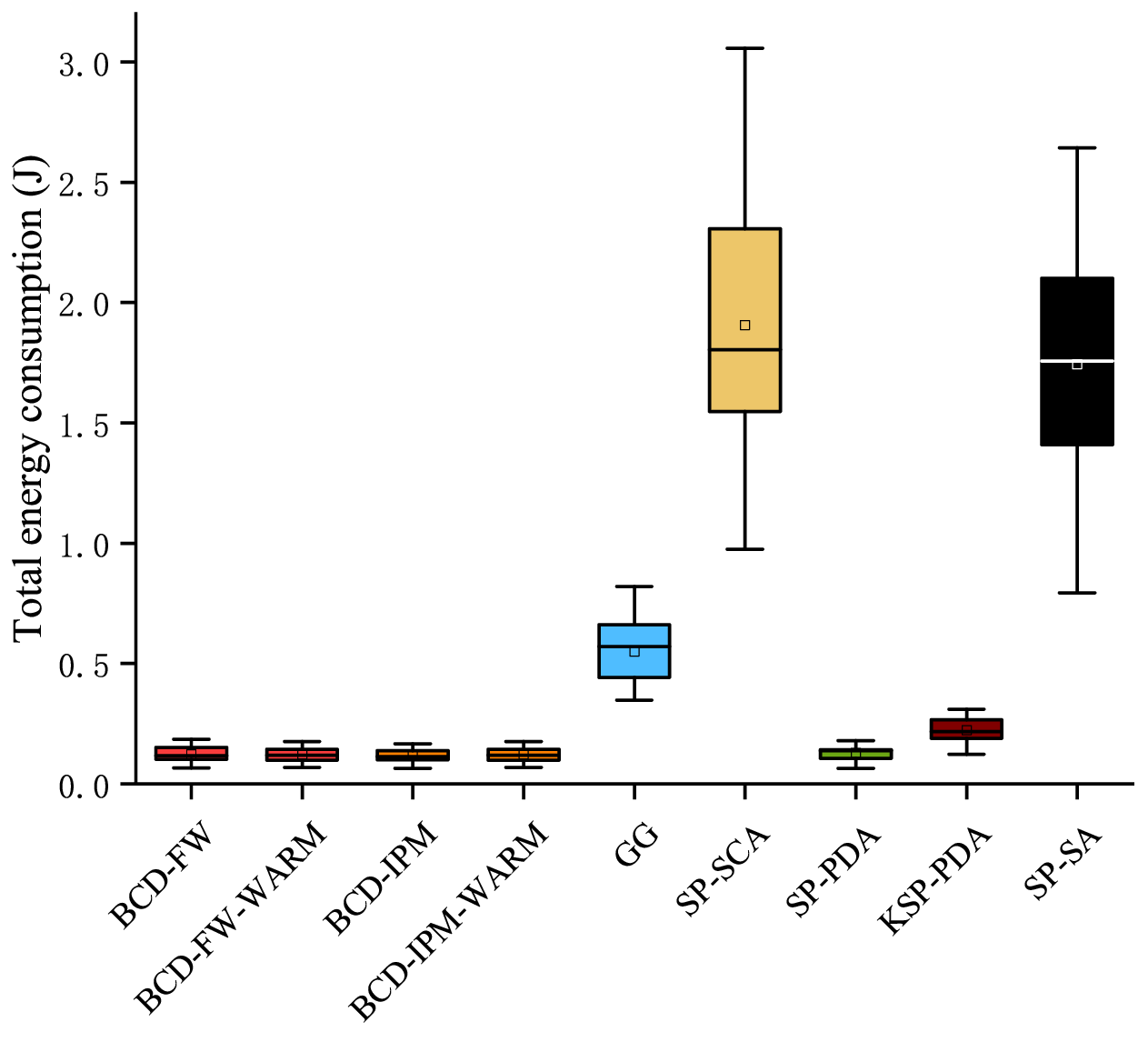}
		\caption{Total energy consumption distributions}
		\label{fig:Etotal_box}
	\end{subfigure}
	
	\caption{Algorithmic robustness evaluation across $20$ randomly generated network topologies under the same scenario where $|\mathcal{V}|=60$, $|\mathcal{K}|=20$, $\xi=5$ and $\alpha=0.4$.} 
	\label{fig:robustness_results_all}
\end{figure*}
As illustrated in Fig. \ref{fig:Time_vs_N}, an evaluation of computational scalability is conducted by recording the execution times (in milliseconds) across varying $|\mathcal{V}| \in [60, 160]$. As anticipated, a clear gap in computational times emerges between the proposed BCD algorithms and baselines. The BCD-IPM incurs the most substantial computational overhead, peaking at an extreme $551,196$ ms under dense scenarios ($|\mathcal{V}|=140$). This is driven by the demanding complexity of computing and inverting the exact second-order Hessian matrix within highly dense topologies. 

However, the proposed BCD-FW-WARM emerges as a profound computational breakthrough. By anchoring the BCD-FW with a high-quality initial solution, and exclusively relying on computationally lightweight first-order gradient information, BCD-FW-WARM achieves its peak execution time of merely $4328$ ms at ($|\mathcal{V}|=160$). Crucially, across the entire evaluated instances, this warm-started algorithm consistently executes faster than the baseline SP-SCA ($5588$ ms at $|\mathcal{V}|=160$). This algorithm even compresses the computational discrepancy with the GG heuristic to a single order of magnitude. This validates that BCD-FW-WARM achieves the ultimate synthesis of near-optimal solution and real-world deployment feasibility.
Finally, while the heuristic baselines (GG and SP-SA) execute in strictly under $162$ ms, this is achieved by abandoning exact iterative resource optimization and spatial multi-path routing. As demonstrated in the preceding resilience analyses, such compensations inevitably lead to severe resource waste, restricting their practical deployment.

In summary, the comprehensive scalability analysis establishes the profound robustness and superiority of the proposed BCD framework. As the number of commodities or the overall topological scale expands, the BCD-FW and BCD-IPM algorithms guarantee near-optimal solutions. They bound the maximum delay within a controlled gap (less than $ 3.53 \times $) relative to the baseline GG, while sustaining the  peak in energy efficiency (up to a $9.52\times$ improvement versus  GG). Among the proposed algorithms, BCD-FW-WARM emerges as a computational breakthrough. By leveraging computationally lightweight first-order gradient information under the high-quality initialization, it delivers a near-optimal solution within seconds.
Conversely, the GG baseline exposes the fatal flaw of heuristic designs. While GG offers computationally cheap execution and competitive maximum delay, it aggressively employs immense transmission power and bandwidth to force data through congested single-path bottlenecks. These results prove that the proposed BCD framework stands as a profoundly sound paradigm, capable of  ultra-reliable, highly scalable, and energy-efficient communications in dense overlay D2D networks.

\subsection{Algorithmic robustness against topological uncertainty}
\label{subsec:robustness_evaluation}

To evaluate the algorithmic robustness against topological uncertainty, we conduct comprehensive empirical evaluations across a diverse set of $20$ independently generated random topologies. The system settings are identical to the settings in Section \ref{subsec:stress_evaluation} with $\xi=5$. 
As illustrated in the boxplots of Fig. \ref{fig:Obj_box}, the distribution of the normalized objective values \eqref{normalized_obj} establishes the superiority and stability of the proposed BCD framework. 
The interquartile ranges (IQRs) of the four BCD algorithms exhibit profound compression compared to the baselines. This minimized statistical dispersion demonstrates that the proposed algorithms maintain optimal delay-energy trade-offs across random topologies. Furthermore, the upper whiskers and outlier data points reveal a critical performance boundary. 

While the BCD-FW algorithm experiences a localized objective inflation (peaking at $0.772$ in Topology 12), the other three BCD algorithms completely neutralize this vulnerability. Specifically, the extreme values of BCD-IPM and both warm-started algorithms are bounded below $0.668$, outperforming the worst-case value of $0.722$ yielded by the GG heuristic.
In contrast, heuristic baselines such as KSP-PDA exhibit severe volatility, with extreme outliers surging to $5.09$ (Topology $ 10 $) due to blind multi-path routing in random topologies. 
While the single-path GG and SP-PDA baselines maintain  compact IQRs comparable to the proposed BCD framework, their median objective values and overall distributional boundaries are elevated. This demonstrates their inferior capabilities in achieving the optimal delay-energy trade-off.

To further deconstruct the objective values, Fig. \ref{fig:Tmax_box} and  \ref{fig:Etotal_box} depict the distributions of its two components: maximum delay and total energy consumption, respectively.  The proposed BCD framework consistently secures highly compact IQRs across both objectives, with merely a single outlier recorded for BCD-FW-WARM in maximum delay (specifically $6.69$ s at Topology $20$). This profound consistency validates the superior robustness against topological uncertainty.
While the GG and SP-SCA baselines exhibit compact IQRs regarding maximum delay, they yield relatively elongated boxes and higher median values for total energy consumption. This corroborates the conclusions drawn in Sections \ref{subsec:stress_evaluation} and \ref{subsec:scalability_evaluation}: these heuristic baselines achieve competitive transmission times through an aggressive over-consumption of transmission power and bandwidth.

Furthermore, all algorithms integrating the optimal PDA algorithm (including the BCD framework, SP-PDA, and KSP-PDA) consistently exhibit remarkably compressed IQRs in total energy consumption (e.g., SP-PDA consistently operating below $0.18$ J). This validates that the PDA algorithm consistently bounds the energy waste, regardless of the underlying topological structures.
However, as evidenced by the severe delay fluctuations of KSP-PDA in Fig. \ref{fig:Tmax_box} (surging to $54.22$ s at Topology 10), simply applying optimal resource allocation atop heuristic multi-path routing fails to stabilize overall network performance.  This contrast validates the necessity of the proposed BCD framework, jointly optimizing the multi-path routing via MF-FW/LR-PDIPM and  resource allocation via PDA algorithm.


\section{Conclusion}
\label{sec:conclusion}
In this paper, we investigated the joint routing and resource allocation  problem in overlay D2D networks.  Because the network-layer traffic flows and physical-layer resources are intricately intertwined via the non-convex Shannon capacity, the resulting  formulation is computationally intractable. To resolve this, we proposed a novel BCD framework to iteratively optimize the routing and resource domains.
For the  routing subproblem, we first employed LSE smoothing to approximate the non-differentiable minimax delay objective. For the requirement of rapid execution, the MF-FW algorithm was developed to guide the descent direction and optimize step sizes based on the directional curvature. For scenarios demanding high-precision solutions, the LR-PDIPM was  developed. By exploiting the Sherman-Morrison rank-1 formula to accelerate the Newton system inversion, this method reduces the per-iteration complexity from $\mathcal{O}(|\mathcal{K}|^3 |\mathcal{V}|^3)$ to $\mathcal{O}(|\mathcal{K}| |\mathcal{V}|^3)$. 
For the physical-layer resource allocation subproblem, we exploited the inherent convexity by utilizing a time-domain perspective transformation. To bypass the scalability bottlenecks of standard commercial solvers, we designed a PDA algorithm. This approach updates Lagrangian multipliers via gradients and solves the  2D link-level subproblems using an inner bisection search. 

Theoretical analysis guarantees that the proposed BCD framework converges to an $\varepsilon$-neighborhood of a stationary point in polynomial time.
Through comprehensive empirical evaluations encompassing resilience, scalability and robustness, the proposed BCD framework establishes its superiority in achieving the optimal delay-energy trade-off. Specifically, BCD-IPM emerges as the most robust variant, guaranteeing the optimal or near-optimal trade-off even under extremely dense topologies. This algorithm achieves a maximum $9.14\times$ reduction in total energy consumption and up to an order of magnitude improvement in energy efficiency, while maintaining a bounded maximum delay gap (up to $3.78\times$) relative to the best baseline GG.  Meanwhile, the BCD-FW-WARM algorithm distinguishes itself as a highly practical engineering solution, synergizing heuristic initializations with lightweight first-order updates to deliver near-optimal solutions in mere seconds. 

\appendix

\section{Theoretical proofs for resource allocation}

\subsection{Proof of Lemma \ref{lemma:core_func}} \label{app:proof_lemma_core}
\begin{proof}
	The proof proceeds in two distinct parts, establishing the joint convexity of the generalized perspective function and the strict joint convexity of the specific bivariate composite function, respectively.
	
	\paragraph{Joint convexity of the perspective function $P(u,v)$}
	
	Given a convex function $g(\cdot)$ and the perspective mapping $P(u,v) = v g(u/v)$ for $v > 0$, we introduce the auxiliary variable $z = u/v$. Computing the second-order partial derivatives yields the Hessian matrix of $P(u, v)$:
	\begin{equation*}
		\nabla^2 P(u, v) = \frac{g''(z)}{v} 
		\begin{bmatrix} 
			1 & -z \\ 
			-z & z^2 
		\end{bmatrix}.
	\end{equation*}
	To establish joint convexity, we verify that the Hessian is positive semi-definite ($\nabla^2 P \succeq \mathbf{0}$). For any arbitrary non-zero vector $\mathbf{w} = [w_1, w_2]^{\mathsf{T}}$, the associated quadratic form evaluates to:
	\begin{equation*}
		\mathbf{w}^{\mathsf{T}} \nabla^2 P(u,v) \mathbf{w} = \frac{g''(z)}{v} (w_1^2 - 2z w_1 w_2 + z^2 w_2^2) = \frac{g''(z)}{v} (w_1 - z w_2)^2.
	\end{equation*}
	Since $g(\cdot)$ is  convex, its second derivative satisfies $g''(z) \ge 0$. Coupled with the domain definition $v > 0$ and the non-negative squared term $(w_1 - z w_2)^2 \ge 0$, the quadratic form is non-negative for all $\mathbf{w}$. Consequently, $\nabla^2 P \succeq \mathbf{0}$, proving that $P(u, v)$ is jointly convex over its domain.
	
	\paragraph{Strict and joint convexity of $f(u,v)$}
	
	We now analyze the bivariate function $f(u, v) = uv(2^{1/uv}-1)$ for positive variables $u, v > 0$. Let $x = uv$ and define the univariate core function $q(x) = x(2^{1/x}-1)$, such that $f(u,v) = q(uv)$. Applying the multivariate chain rule, the Hessian matrix of $f(u,v)$ is formulated as:
	\begin{equation*}
		\nabla^2 f(u,v) = 
		\begin{bmatrix}
			v^2 q''(x) & x q''(x) + q'(x) \\
			x q''(x) + q'(x) & u^2 q''(x)
		\end{bmatrix}.
	\end{equation*}
	The first and second derivatives of $q(x)$ are given by:
	\begin{equation} \label{eq:qx_deriv}
		q'(x) = 2^{1/x} - 1 - \frac{\ln 2}{x} 2^{1/x}, \quad q''(x) = \frac{(\ln 2)^2}{x^3} 2^{1/x}.
	\end{equation}
	To prove strict joint convexity, we must guarantee that $\nabla^2 f(u,v)$ is  positive definite ($\nabla^2 f \succ \mathbf{0}$), which holds if and only if both its trace and determinant are positive.
	
	{1) Trace analysis:} 
	Evaluating the trace of the Hessian yields:
	\begin{equation*}
		\text{Tr}(\nabla^2 f) = (u^2+v^2)q''(x).
	\end{equation*}
	Since $u, v > 0 \implies x > 0$, equation \eqref{eq:qx_deriv} guarantees $q''(x) > 0$. Thus, $\text{Tr}(\nabla^2 f) > 0$.
	
	{2) Determinant analysis:} 
	The determinant of the Hessian is computed as:
	\begin{align*}
		\det(\nabla^2 f) &= u^2 v^2 (q''(x))^2 - (x q''(x) + q'(x))^2 \nonumber \\
		&= x^2 (q''(x))^2 - \left( x^2 (q''(x))^2 + 2x q''(x)q'(x) + (q'(x))^2 \right) \nonumber \\
		&= -q'(x) \left[ 2x q''(x) + q'(x) \right].
	\end{align*}
	To determine the sign of this product, we define the positive auxiliary variable $\omega = {\ln 2}/{x} > 0$, mapping the exponential term to $2^{1/x} = \exp^\omega$. 
	
	First, we analyze the sign of $q'(x)$. Let $\phi(\omega) = \exp^\omega(1-\omega) - 1$. Its derivative is $\phi'(\omega) = -\omega \exp^\omega < 0$ for all $\omega > 0$. Given the initial condition $\phi(0) = 0$, it follows that $\phi(\omega) < 0$. Therefore, $q'(x) < 0$, which yields $-q'(x) > 0$.
	
	Second, we analyze the sign of the bracketed term $S(x) = 2x q''(x) + q'(x)$. Substituting $\omega$, we obtain the mapped function $\psi(\omega) = \exp^\omega(2\omega^2 - \omega + 1) - 1$. Its derivative evaluates to $\psi'(\omega) = \exp^\omega(2\omega^2 + 3\omega) > 0$ for $\omega > 0$. Given $\psi(0) = 0$, it follows that $\psi(\omega) > 0$, implying $S(x) > 0$.
	
	Since both decomposed factors are positive, their product guarantees $\det(\nabla^2 f) > 0$. Concurrently satisfying $\text{Tr}(\nabla^2 f) > 0$ and $\det(\nabla^2 f) > 0$, the Hessian matrix is positive definite. This completes the proof that $f(u,v)$ is strictly jointly convex.
\end{proof}

\subsection{Proof of Theorem \ref{thm:convexity}} \label{app:proof_thm_convexity}
\begin{proof}
	To establish \eqref{eq:convex_global} as a convex optimization problem, we verify the convexity of its objective and feasible region. 
	
	{1) Objective function:} The objective \eqref{con_obj_global} is a sum of a linear term $\alpha T$ and energy functions ${E}_{ij}(l_{ij}, t_{ij}) = \frac{N_0}{h_{ij}} l_{ij} t_{ij} (2^{m_{ij}/(l_{ij} t_{ij})} - 1)$. Via the affine mapping $u = l_{ij}$ and $v = t_{ij}/m_{ij}$, we obtain ${E}_{ij} \propto f(u,v)$. By Lemma \ref{lemma:core_func}, $f(u,v) = uv(2^{1/uv}-1)$ is strictly jointly convex for positive $(u,v)$. Since convexity is preserved under affine mapping and positive scaling, the objective is convex.
	
	{2) Feasible region:} Delay and bandwidth constraints \eqref{con_delay_global}\eqref{con_spec_global} define affine spaces. The power constraints \eqref{con_power_global} depend on $p_{ij}(l_{ij}, t_{ij}) = \frac{N_0}{h_{ij}} l_{ij} (2^{m_{ij}/(l_{ij} t_{ij})} - 1)$. Let $z(t_{ij}) = m_{ij}/t_{ij}$, which is convex for $t_{ij} > 0$. The perspective $P(z, l_{ij}) = l_{ij} (2^{z/l_{ij}} - 1)$ is jointly convex for $l_{ij} > 0$ and monotonically increasing in $z$ as $\frac{\partial P}{\partial z} > 0$. By the scalar composition theorem, substituting the convex $z(t_{ij})$ into $P(z, l_{ij})$ preserves joint convexity. Thus, the feasible region is an intersection of convex sets, completing the proof.
\end{proof}

\section{Theoretical proofs for MF-FW}
\subsection{Proof of Lemma \ref{lem:L_smooth}} \label{app:proof_lemma_L_smooth}
\begin{proof}
	The curvature of the smoothed objective $F(\mathbf{x})$ is governed by its LSE component. By the multivariate chain rule, the gradient is $\nabla F(\mathbf{x}) = \alpha \sum_{k \in \mathcal{K}} \beta_k \nabla T_k(\mathbf{x})$, where $\beta_k = \exp^{\mu T_k} / \sum_{j \in \mathcal{K}} \exp^{\mu T_j}$ satisfies $\beta_k > 0$ and $\sum_{k \in \mathcal{K}} \beta_k = 1$. The Hessian matrix is derived as:
	$$
	\nabla^2 F(\mathbf{x}) = \alpha \mu \left[ \sum_{k \in \mathcal{K}} \beta_k \nabla T_k \nabla T_k^{\mathsf{T}} - \left( \sum_{k \in \mathcal{K}} \beta_k \nabla T_k \right) \left( \sum_{k \in \mathcal{K}} \beta_k \nabla T_k \right)^{\mathsf{T}} \right].
	$$
	
	{1) Positive semi-definiteness:} For any arbitrary vector $\mathbf{v}$, the quadratic form is given by:
	\begin{equation}\label{pos_semidef}
		\mathbf{v}^{\mathsf{T}} \nabla^2 F(\mathbf{x}) \mathbf{v} = \alpha \mu \left[ \sum_{k \in \mathcal{K}} \beta_k (\nabla T_k^{\mathsf{T}} \mathbf{v})^2 - \left( \sum_{k \in \mathcal{K}} \beta_k \nabla T_k^{\mathsf{T}} \mathbf{v} \right)^2 \right].
	\end{equation}
	Since $\sum_{k \in \mathcal{K}} \beta_k = 1$, we can apply Jensen's inequality to the function $f(z) = z^2$, which yields $\left( \sum_{k \in \mathcal{K}} \beta_k z_k \right)^2 \le \sum_{k \in \mathcal{K}} \beta_k z_k^2$ by setting $z_k = \nabla T_k^{\mathsf{T}} \mathbf{v}$. Consequently, $\mathbf{v}^{\mathsf{T}} \nabla^2 F(\mathbf{x}) \mathbf{v} \ge 0$. Thus, $\nabla^2 F(\mathbf{x}) \succeq \mathbf{0}$, ensuring the convexity of $F(\mathbf{x})$.
	
	{2) Lipschitz continuity:} The $L$-smoothness constant corresponds to the spectral norm $\|\nabla^2 F(\mathbf{x})\|_2$. To bound this, we analyze the Rayleigh quotient for any unit vector $\mathbf{v}$ (i.e., $\|\mathbf{v}\|_2 = 1$). From \eqref{pos_semidef}, we obtain an upper bound:
	$$
	\mathbf{v}^{\mathsf{T}} \nabla^2 F(\mathbf{x}) \mathbf{v} \le \alpha \mu \sum_{k \in \mathcal{K}} \beta_k (\nabla T_k^{\mathsf{T}} \mathbf{v})^2.
	$$
	Applying the Cauchy-Schwarz inequality, $$(\nabla T_k^{\mathsf{T}} \mathbf{v})^2 \le \|\nabla T_k\|_2^2 \|\mathbf{v}\|_2^2 = \|\nabla T_k\|_2^2.$$ This yields:
	\begin{align*}
		\mathbf{v}^{\mathsf{T}} \nabla^2 F(\mathbf{x}) \mathbf{v}& \le \alpha \mu \sum_{k \in \mathcal{K}} \beta_k \|\nabla T_k\|_2^2    \\
		&\le \alpha \mu \max_{k \in \mathcal{K}} \|\nabla T_k\|_2^2 \sum_{k \in \mathcal{K}} \beta_k=\alpha \mu \max_{k \in \mathcal{K}} \|\nabla T_k\|_2^2.
	\end{align*}
	Since $T_k(\mathbf{x})$ is linear, $\nabla T_k$ is a constant vector of path delay coefficients. Therefore, the spectral norm is  bounded by:
	$$
	\|\nabla^2 F(\mathbf{x})\|_2 \le \alpha \mu \left( \max_{k \in \mathcal{K}} \sum_{(i,j) \in \mathcal{E}} \frac{M^k}{r_{ij}} \right)^2 = \alpha \mu C_{\max}^2.
	$$
	Setting $L = \alpha \mu C_{\max}^2$ completes the proof.
\end{proof}

\subsection{Proof of Theorem \ref{thm:fw_rate}} \label{app:proof_fw_rate}
\begin{proof}
	Let $h_n = F(\mathbf{x}^{(n)}) - F(\mathbf{x}^*)$ denote the primal optimality gap at iteration $n$, and let $C = L D^2$. 
	
	Because the objective function $F(\mathbf{x})$ is $L$-smooth (as established in Lemma \ref{lem:L_smooth}), the standard descent lemma bounds the objective value at the next iteration for any $\gamma \in [0,1]$:
	\begin{equation} \label{eq:descent_lemma}
		F(\mathbf{x}^{(n+1)}) \le F(\mathbf{x}^{(n)}) + \gamma \nabla F(\mathbf{x}^{(n)})^{\mathsf{T}} (\mathbf{y}^{(n)} - \mathbf{x}^{(n)}) + \frac{\gamma^2}{2} L \|\mathbf{y}^{(n)} - \mathbf{x}^{(n)}\|_2^2.
	\end{equation}
	Due to $D = \max_{\mathbf{x}, \mathbf{y} \in \mathcal{X}} \|\mathbf{x} - \mathbf{y}\|_2$, we have $\|\mathbf{y}^{(n)} - \mathbf{x}^{(n)}\|_2^2 \le D^2$. Furthermore, due to the convexity of $F(\mathbf{x})$ and the optimality of  $\mathbf{y}^{(n)}$ under current gradient $ \nabla F(\mathbf{x}^{(n)}) $, the gradient term bounds the negative primal gap: $\nabla F(\mathbf{x}^{(n)})^{\mathsf{T}} (\mathbf{y}^{(n)} - \mathbf{x}^{(n)}) \le F(\mathbf{x}^*) - F(\mathbf{x}^{(n)}) = -h_n$.
	
	Substituting these bounds into \eqref{eq:descent_lemma} yields:
	\begin{equation*} \label{eq:global_bound}
		F(\mathbf{x}^{(n)} + \gamma \mathbf{d}^{(n)}) \le F(\mathbf{x}^{(n)}) - \gamma h_n + \frac{\gamma^2}{2} C.
	\end{equation*}

	The truncated Newton step size $\lambda^*$ is the exact minimizer of the local quadratic surrogate defined by the directional curvature $C^{(n)}$. Because $F(\mathbf{x})$ is $L$-smooth over the compact domain $\mathcal{X}$, the local curvature is upper-bounded by the global curvature constant, i.e., $C^{(n)} \le C$. Thus, for $  \forall \gamma \in [0,1] $,
	$$F(\mathbf{x}^{(n)} + \lambda^* \mathbf{d}^{(n)}) \le F(\mathbf{x}^{(n)}) - \gamma h_n + \frac{\gamma^2}{2} C. $$ 
	Then, 
	\begin{equation} \label{eq:recurrence}
		h_{n+1} \le (1 - \gamma) h_n + \frac{\gamma^2}{2} C.
	\end{equation}
	This inequality guarantees a primal gap $h_{n+1}$ that is no greater than the gap produced by the standard  diminishing step size $\gamma^{(n)} = {2}/({n+2})$. Therefore, it is sufficient to prove the theoretical upper bound $h_n \le {2C}/({n+1})$ via mathematical induction using the recurrence \eqref{eq:recurrence} generated by $\gamma^{(n)}$.
	
	{Base case ($n=0$):}
	For the initial iteration, applying $\gamma^{(0)} = 1$ to  \eqref{eq:recurrence} gives $h_1 \le (1 - 1) h_0 + \frac{C}{2}(1)^2 = \frac{C}{2}$. The theorem necessitates $h_1 \le \frac{2C}{0+2} = C$. Since $\frac{C}{2} \le C$ strictly holds, the base case is verified.
	
	{Inductive step:}
	Assume the bound holds for iteration $n$, i.e., $h_n \le {2C}/({n+1})$. Substituting this inductive hypothesis and $\gamma^{(n)} = {2}/({n+2})$ into \eqref{eq:recurrence} for $h_{n+1}$ yields:
	\begin{align*}
		h_{n+1} &\le \left(1 - \frac{2}{n+2}\right) \frac{2C}{n+1} + \frac{C}{2} \left(\frac{2}{n+2}\right)^2 \\
		&= \frac{n}{n+2} \cdot \frac{2C}{n+1} + \frac{2C}{(n+2)^2} \\
		&= \frac{2C}{n+2} \left[ \frac{n}{n+1} + \frac{1}{n+2} \right].
	\end{align*}
	Analyzing the bracket term, we derive:
	\begin{equation*}
		\frac{n}{n+1} + \frac{1}{n+2} = \frac{n(n+2) + (n+1)}{(n+1)(n+2)} = \frac{n^2+3n+1}{n^2+3n+2} <1.
	\end{equation*}
	Therefore, 
	\begin{equation*}
		h_{n+1} \le \frac{2C}{n+2}.
	\end{equation*}
	This completes the mathematical induction, establishing the global $\mathcal{O}(1/n)$ sublinear convergence rate of MF-FW algorithm.
\end{proof}

\section{Theoretical proofs for LR-PDIPM}

\subsection{Proof of Lemma \ref{lemma:lipschitz_hessian}} \label{app:proof_lemma_hessian}
\begin{proof}
	By the Mean Value Theorem for matrix-valued functions, establishing the Lipschitz continuity of the Hessian matrix $\nabla^2 F(\mathbf{x})$ is  equivalent to  bounding the spectral norm of the third-order derivative $\nabla^3 F(\mathbf{x})$ over the feasible domain $\mathcal{X}$.
	
	Recall from the proof of Lemma \ref{lem:L_smooth} that the Hessian is written as:
	\begin{equation*}
		\nabla^2 F(\mathbf{x}) = \alpha \mu \left( \sum_{k \in \mathcal{K}} \beta_k \nabla T_k \nabla T_k^{\mathsf{T}} - \bar{\nabla T} \bar{\nabla T}^{\mathsf{T}} \right),
	\end{equation*}
	where $\bar{\nabla T} = \sum_{k \in \mathcal{K}} \beta_k \nabla T_k$.
	
	To obtain the third-order derivative, we compute the directional derivative of $\nabla^2 F(\mathbf{x})$ along an arbitrary unit vector $\mathbf{d} \in \mathbb{R}^{|\mathcal{E}||\mathcal{K}|}$ (with $\|\mathbf{d}\|_2 = 1$). By applying the chain rule, the derivative of $ \beta_k $ is given by $\nabla \beta_k = \mu \beta_k (\nabla T_k - \bar{\nabla T})$. Consequently, the third-order directional derivative is calculated as
	\begin{equation*} \label{eq:third_derivative}
		\nabla^3 F(\mathbf{x})[\mathbf{d}] = \alpha \mu^2 \sum_{k \in \mathcal{K}} \beta_k (\nabla T_k - \bar{\nabla T}) (\nabla T_k - \bar{\nabla T})^{\mathsf{T}} \left( (\nabla T_k - \bar{\nabla T})^{\mathsf{T}} \mathbf{d} \right).
	\end{equation*}
	Then,
	\begin{equation*}
		\left| \nabla^3 F(\mathbf{x})[\mathbf{d}, \mathbf{d}, \mathbf{d}] \right| = \alpha \mu^2 \left| \sum_{k \in \mathcal{K}} \beta_k \left( (\nabla T_k - \bar{\nabla T})^{\mathsf{T}} \mathbf{d} \right)^3 \right|.
	\end{equation*}
	Recall $\beta_k \in (0,1)$ and $\sum \beta_k = 1$. Applying the triangle inequality and the standard operator norm bound, we obtain:
	\begin{equation} \label{cubic_inequality}
		\sup_{\|\mathbf{d}\|_2 = 1} \left| \nabla^3 F(\mathbf{x})[\mathbf{d}, \mathbf{d}, \mathbf{d}] \right| \le 2 \alpha \mu^2 \max_{k \in \mathcal{K}} \|\nabla T_k\|_2^3.
	\end{equation}

	Given that  $T_k(\mathbf{x})$ is linear, its gradient vector $\nabla T_k \in \mathbb{R}^{|\mathcal{K}||\mathcal{E}|}$ is highly sparse and the non-zero entries are  the constant  $M^k / r_{ij}$. The $L_2$-norm of the gradient is bounded by the network scale, specifically $\|\nabla T_k\|_2 \le \sqrt{|\mathcal{E}|} R_{\max}$ and $R_{\max} = \max_{k, (i,j)} (M^k / r_{ij})$. Substituting this bound into \eqref{cubic_inequality} yields:
	\begin{equation*}
		\|\nabla^3 F(\mathbf{x})\|_2 \le 2 \alpha \mu^2 \left( \sqrt{|\mathcal{E}|} R_{\max} \right)^3.
	\end{equation*}
	
	Setting the Lipschitz constant to $L_H = 2 \alpha \mu^2 |\mathcal{E}|^{3/2} R_{\max}^3$, we obtain $L_H = \mathcal{O}(\alpha \mu^2 |\mathcal{E}|^{3/2} R_{\max}^3)$. This guarantees that the Hessian is Lipschitz continuous, revealing that the Newton-step curvature variation is governed jointly by the smoothing parameter and network scale.
\end{proof}

\subsection{Proof of Lemma \ref{lemma:step_size}} \label{app:proof_step_size}
\begin{proof}
	Let $\tilde{\mu} = \mu_{\text{gap}}^{(n)}$ denote the current duality gap. For any iterate within the wide neighborhood $\mathcal{N}_{-\infty}(\gamma)$, the variables satisfy the centrality condition $x_i s_i \ge \gamma \tilde{\mu}$. 
	
	{1)  Bounding the second-order terms.} \\
	From the Newton system of the PDIPM, the search directions satisfy:
	\begin{equation} \label{eq:newton_system}
		\mathbf{S} \Delta \mathbf{x} + \mathbf{X} \Delta \mathbf{s} = - \mathbf{X} \mathbf{s} + \sigma\tilde{\mu} \mathbf{1}.
	\end{equation}
	Because the linear equality constraints dictate $\mathbf{A} \Delta \mathbf{x} = \mathbf{0}$, the primal direction lies in the null space of $\mathbf{A}$. Correspondingly, the dual direction satisfies $\Delta \mathbf{s} = \nabla^2 F(\mathbf{x}) \Delta \mathbf{x} - \mathbf{A}^{\mathsf{T}} \Delta \mathbf{y}$. Therefore, the inner product yields $\Delta \mathbf{x}^{\mathsf{T}} \Delta \mathbf{s} = \Delta \mathbf{x}^{\mathsf{T}} \nabla^2 F(\mathbf{x}) \Delta \mathbf{x}$. Since the Hessian is positive semi-definite (as established in Lemma \ref{lem:L_smooth}), we have $\Delta \mathbf{x}^{\mathsf{T}} \Delta \mathbf{s} \ge 0$.
	
	By dividing both sides of \eqref{eq:newton_system} by $(\mathbf{X}\mathbf{S})^{1/2}$ and applying the triangle inequality along with the wide neighborhood lower bound $x_i s_i \ge \gamma \tilde{\mu}$, standard interior-point algebraic manipulations \cite{wright1997primal} ensure that the scaled Newton directions are bounded. Consequently, the Euclidean norm of the product $ \Delta \mathbf{X} \Delta \mathbf{s} $ is bounded by the duality gap:
	\begin{equation*}
		\|\Delta \mathbf{X} \Delta \mathbf{s}\|_2^2 = \sum_{i} (\Delta x_i \Delta s_i)^2 \le \mathcal{C}\tilde{\mu}^2,
	\end{equation*}
	where $\mathcal{C} > 0$ is a constant dependent on $\gamma$ and the network scale, independent of the iteration count.
	
	{2)  Guaranteeing a positive step size.} \\
	To maintain the iterates within $\mathcal{N}_{-\infty}(\gamma)$ at the next iteration, the condition $ \mathbf{x}(\lambda) \circ \mathbf{s}(\lambda)=(\mathbf{x} + \lambda \Delta \mathbf{x}) \circ (\mathbf{s} + \lambda \Delta \mathbf{s}) \ge \gamma\tilde{\mu}(\lambda) \mathbf{1}$ must hold. 
	Because the Hessian is $L_H$-Lipschitz continuous (Lemma \ref{lemma:lipschitz_hessian}), the deviation of the nonlinear KKT residuals \eqref{eq:kkt_residuals} from the linear Newton prediction is bounded by $\mathcal{O}(\lambda^2 L_H \|\Delta \mathbf{x}\|_2^2)$. The bounded cross terms and the finite Lipschitz constant $L_H$  guarantee that the neighborhood violation grows quadratically with $\lambda$. Thus, the backtracking line search will accept a step size $\lambda$ before it shrinks to zero, proving the existence of a lower bound $\lambda \ge \lambda_{\min} > 0$.
	
	{3) Monotonic gap reduction.} \\
	The duality gap at the candidate step $\lambda$ is calculated as:
	\begin{equation} \label{eq:gap_eval}
		\tilde{\mu}(\lambda) = \frac{(\mathbf{x} + \lambda \Delta \mathbf{x})^{\mathsf{T}} (\mathbf{s} + \lambda \Delta \mathbf{s})}{|\mathcal{K}||\mathcal{E}|} = \frac{\mathbf{x}^{\mathsf{T}} \mathbf{s} + \lambda (\mathbf{x}^{\mathsf{T}} \Delta \mathbf{s} + \mathbf{s}^{\mathsf{T}} \Delta \mathbf{x}) + \lambda^2 \Delta \mathbf{x}^{\mathsf{T}} \Delta \mathbf{s}}{|\mathcal{K}||\mathcal{E}|}.
	\end{equation}
	Multiplying \eqref{eq:newton_system} by $\mathbf{1}^{\mathsf{T}}$, we obtain the sum of the cross terms: $\mathbf{x}^{\mathsf{T}} \Delta \mathbf{s} + \mathbf{s}^{\mathsf{T}} \Delta \mathbf{x} = -\mathbf{x}^{\mathsf{T}} \mathbf{s} + \sigma\tilde{\mu} |\mathcal{K}||\mathcal{E}| = -|\mathcal{K}||\mathcal{E}|\mu(1-\sigma)$. Substituting this into \eqref{eq:gap_eval} yields:
	\begin{equation*}
		\tilde{\mu}(\lambda) =\tilde{\mu} \big( 1 - \lambda(1-\sigma) \big) + \frac{\lambda^2}{|\mathcal{K}||\mathcal{E}|} \Delta \mathbf{x}^{\mathsf{T}} \Delta \mathbf{s}.
	\end{equation*}
	The backtracking line search controls the second-order term $\lambda^2 \Delta \mathbf{x}^{\mathsf{T}} \Delta \mathbf{s}$. Enforcing the search criteria  yields the upper bound for the updated gap:
	\begin{equation*}
		\mu_{\text{gap}}^{(n+1)} \le \big( 1 - \lambda_{\min}(1-\sigma_{\max}) \big) \mu_{\text{gap}}^{(n)}.
	\end{equation*}
	This confirms the monotonic reduction of the optimality gap, concluding the proof.
\end{proof}

\subsection{Proof of Theorem \ref{thm:ipm_complexity}} \label{app:proof_thm_complexity}
\begin{proof}
	The time complexity of the LR-PDIPM is determined by the dimension of the inequality constraints and the reduction rate of the duality gap. We establish the proof in two sequential steps.
	
	{1) Gap reduction rate.} \\
	For the smoothed routing subproblem \eqref{eq:smoothed_routing}, the inequality bounds originate from the non-negativity constraints $\mathbf{x} \ge \mathbf{0}$. Since $\mathbf{x}$ encompasses the flow of each commodity $k \in \mathcal{K}$ over each edge $(i,j) \in \mathcal{E}$, the total number of inequality constraints is  $|\mathcal{K}||\mathcal{E}|$. 
	
	By initializing the IPM at $(\mathbf{x}^{(0)}, \mathbf{s}^{(0)}) > \mathbf{0}$ within the wide neighborhood $\mathcal{N}_{-\infty}(\gamma)$, Lemma \ref{lemma:step_size} guarantees that the algorithm generates a sequence of iterates that are strictly feasible with respect to the inequality bounds. Furthermore, it establishes a  monotonic reduction of the duality gap:
	\begin{equation} \label{eq:gap_geometric}
		\mu^{(n+1)} \le \left( 1 - \delta \right) \mu^{(n)},
	\end{equation}
	where the single-iteration reduction factor is bounded by $\delta = \lambda_{\min}(1-\sigma_{\max})$. According to the standard path-following interior-point theory \cite{wright1997primal}, to maintain the iterates within the neighborhood and ensure the quadratic boundedness of the cross terms, the allowable step size $\lambda_{\min}$ scales inversely with the square root of the inequality dimension, yielding $\delta = \Omega(1/\sqrt{|\mathcal{K}||\mathcal{E}|})$.
	
	{2) Iteration complexity bound.} \\
	Applying the recursive relation \eqref{eq:gap_geometric} over $n$ iterations yields:
	\begin{equation*}
		\mu^{(n)} \le (1 - \delta)^n \mu^{(0)}.
	\end{equation*}
	To achieve an $\varepsilon$-optimal solution, the algorithm must satisfy $\mu^{(n)} \le \varepsilon$. Taking the natural logarithm of both sides gives:
	\begin{equation*}
		n \ln(1 - \delta) \le \ln\left( \frac{\varepsilon}{\mu^{(0)}} \right).
	\end{equation*}
	Using the standard logarithmic inequality $\ln(1 - \delta) \le -\delta$ for $\delta \in (0, 1)$, we obtain the required number of iterations:
	\begin{equation*}
		n \ge \frac{1}{\delta} \ln\left( \frac{\mu^{(0)}}{\varepsilon} \right).
	\end{equation*}
	Since $\delta^{-1} = \mathcal{O}(\sqrt{|\mathcal{K}||\mathcal{E}|})$, and the initial gap $\mu^{(0)}$ is a given constant, the maximum number of iterations is  bounded by:
	\begin{equation*}
		n = \mathcal{O}\left(\sqrt{|\mathcal{K}||\mathcal{E}|} \ln\left(\frac{1}{\varepsilon}\right)\right).
	\end{equation*}
	Because the smoothed objective $ F(\mathbf{x}) $ is convex over the feasible domain $\mathcal{X}$, the satisfaction of the $\varepsilon$-KKT residual conditions guarantees that the limit of the iterate sequence is a global $\varepsilon$-optimal solution, concluding the proof.
\end{proof}

\section{Theoretical proofs for PDA}
\subsection{Proof of Lemma \ref{lemma:cd_conv}} \label{app:proof_cd_conv}
\begin{proof}
	To establish the convergence of the alternating coordinate descent for the link-level subproblem, we analyze the structural properties of $\Phi_{ij}$ defined in \eqref{eq:lagrangian_func}.
	
	{1) Strict convexity and differentiability.} \\
	The subproblem objective is a positive linear combination: $\Phi_{ij}(l_{ij}, t_{ij}) = (1-\alpha){E}_{ij} + \varphi_i p_{ij} + \omega l_{ij} + (\sum_{k \in \mathcal{K}} \theta^k \rho_{ij}^k) t_{ij}$. By Lemma \ref{lemma:core_func} and Theorem \ref{thm:convexity}, the composite energy function ${E}_{ij}$ is strictly jointly convex, and the transmission power function $p_{ij}$ is jointly convex with respect to $(l_{ij}, t_{ij})$. Given that $\alpha \in (0,1)$ and the dual multipliers are non-negative, the strict convexity is  preserved in their sum. Thus, $\Phi_{ij}$ is strictly convex and continuously differentiable. The feasible domain $\mathcal{F}_{ij} = [l_{\min}, B] \times [t_{\min}, T_{\max}^{\text{QoS}}]$ is a compact Cartesian product domain, ensuring the constraints on $l_{ij}$ and $t_{ij}$ are completely decoupled.
	
	{2) Monotonic descent.} \\
	Let $\mathbf{z}^{(q)} = (l_{ij}^{(q)}, t_{ij}^{(q)})$ denote the iterate at the $q$-th inner iteration. The exact 1D bisection updates yield:
	\begin{equation*}
		\Phi_{ij}(l_{ij}^{(q)}, t_{ij}^{(q)}) \le \Phi_{ij}(l_{ij}^{(q)}, t_{ij}^{(q-1)}) \le \Phi_{ij}(l_{ij}^{(q-1)}, t_{ij}^{(q-1)}).
	\end{equation*}
	Due to the strict convexity of $\Phi_{ij}$, the sequence of objective values $\{\Phi_{ij}(\mathbf{z}^{(q)})\}$ is  monotonically decreasing until the minimum is reached. Bounded below over the compact set $\mathcal{F}_{ij}$, the sequence is guaranteed to converge to a limit point.
	
	{3) Uniqueness of the limit point.} \\
	For continuously differentiable functions optimized over a Cartesian product domain, any limit point of the coordinate descent sequence is guaranteed to be a stationary point. Because $\Phi_{ij}$ is strictly convex, this point is the global minimum. Therefore, the sequence converges to the unique minimum, concluding the proof.
\end{proof}

\subsection{Proof of Lemma \ref{lemma:bounded_g}} \label{app:proof_bounded_g}
\begin{proof}
	To prove the uniform boundedness of the dual gradient $\|\mathbf{g}^{(n)}\|_2$, we formulate its normalized components based on the dual update rules \eqref{eq:dual_update_resources} and bound them individually.
	
	{1) Bounding normalized linear residuals.} \\
	According to the bandwidth update \eqref{eq:update_omega} and the routing delay update \eqref{eq:update_theta}, the corresponding gradient components are formulated as:
	\begin{align}
		g_\omega^{(n)} &= \frac{\sum_{(i,j)\in \mathcal{E}_{act}} l_{ij}^{(n)} - B}{B}, \label{eq:g_omega} \\
		g_{\theta^k}^{(n)} &= \sum_{(i,j)\in \mathcal{E}_{act}} \rho_{ij}^k t_{ij}^{(n)}, \quad \forall k \in \mathcal{K}. \label{eq:g_theta}
	\end{align}
	During the primal update, the allocated bandwidth variables $l_{ij}^{(n)}$ are restricted within $[l_{\min}, B]$. Thus, the maximal theoretical sum is bounded by $|\mathcal{E}_{act}|B$, yielding a strict bound on the relative bandwidth violation: $|g_\omega^{(n)}| \le |\mathcal{E}_{act}|$. Similarly, because the time variables are bounded by the QoS delay tolerance $t_{ij}^{(n)} \le T_{\max}^{\text{QoS}}$, and the routing fractions $\rho_{ij}^k \in [0,1]$, the delay gradient is bounded by a finite constant: $|g_{\theta^k}^{(n)}| \le |\mathcal{E}_{act}| T_{\max}^{\text{QoS}}$.
	
	{2) Bounding the normalized power residual.} \\
	According to the node power update \eqref{eq:update_phi}, the normalized power gradient for node $i$ is formulated as:
	\begin{equation}\label{eq:g_phi}
		g_{\varphi_i}^{(n)} = \frac{\sum_{j \in \mathcal{V}_{out}(i)} p_{ij}(l_{ij}^{(n)}, t_{ij}^{(n)}) - P_i^{\max}}{P_i^{\max}}. 
	\end{equation}
	During the primal update, the resource variables $(l_{ij}^{(n)}, t_{ij}^{(n)})$ are restricted within the compact set $\mathcal{F}_{ij} = [l_{\min}, B] \times [t_{\min}, T_{\max}^{\text{QoS}}]$. Because the transmission power $p_{ij}(l, t)$ is a continuous function over these bounded variables, its output cannot diverge to infinity. Consequently, the maximal  power consumed by any single link is bounded by a finite constant, defined as $P_{\text{bound}} = \max_{(l,t) \in \mathcal{F}_{ij}} p_{ij}(l,t)$. Thus, the maximal theoretical power sum for node $i$ is bounded by $|\mathcal{V}_{out}(i)| P_{\text{bound}}$, yielding a bound on the relative power violation:
	\begin{equation*}
		|g_{\varphi_i}^{(n)}| \le \frac{|\mathcal{V}_{out}(i)| P_{\text{bound}}}{P_i^{\max}} + 1.
	\end{equation*}
	
	{3) Global gradient bound.} \\
	Equations \eqref{eq:g_omega}-\eqref{eq:g_phi} demonstrate that every scalar element of the normalized gradient vector $\mathbf{g}^{(n)}$ is bounded by constants dictated by the network topology ($|\mathcal{V}_{out}(i)|, |\mathcal{E}_{act}|$) and the finite domain bounds ($P_{\text{bound}} / P_i^{\max}$). Therefore, their $L_2$-norm is uniformly bounded, verifying the existence of a finite constant $G > 0$ such that $\|\mathbf{g}^{(n)}\|_2 \le G$.
\end{proof}

\bibliographystyle{elsarticle-num} 
\bibliography{reference}

\end{document}